\documentclass{amsart}

\usepackage{amsmath}
\usepackage{amsthm}
\usepackage{amssymb}
\usepackage{mathrsfs}
\usepackage{graphicx}
\usepackage{verbatim}
\usepackage{tikz}
\usepackage[all]{xy}
\usepackage{graphics}

\newtheorem{teor}{Theorem}[section]
\newtheorem{lemma}[teor]{Lemma}
\newtheorem{corol}[teor]{Corollary}
\newtheorem{prop}[teor]{Proposition}

\theoremstyle{definition}

\newtheorem{defin}{Definition}[section]

\theoremstyle{remark}
\newtheorem{rem}[teor]{Remark}
\newtheorem{Levi-Civita}[teor]{On the Levi-Civita transform}

\newcommand{\R}{\mathbb{R}}
\newcommand{\Z}{\mathbb{Z}}
\def\mfu{\mathfrak{u}}
\def\mfv{\mathfrak{v}}

\def\beq{\begin{equation}}
\def\eeq{\end{equation}}

\def\pa{\partial}
\def\t{\theta}

\def\d{\delta}
\def\g{\gamma}

\def\wt{\widetilde}
\def\wh{\widehat}
\def\f{\varphi}
\def\l{\lambda}
\def\a{\alpha}

\def\n{\nabla}
\def\o{\omega}
\def\eps{\epsilon}
\def\r{\rho}
\def\wc{\rightharpoonup}
\title{\sc Symbolic dynamics for the $N$-centre problem at negative energies}
\author{%
Nicola Soave
\and
Susanna Terracini}

\address{Dipartimento di Matematica e Applicazioni, Universit\`a degli Studi
di Milano-Bicocca, Via Bicocca degli Arcimboldi, 8, 20126 Milano,
Italy}
\email{n.soave@campus.unimib.it}
\email{susanna.terracini@unimib.it}

\thanks{Work partially supported by the
italian PRIN2009 grant ``Critical Point Theory and Perturbative
Methods for Nonlinear Differential Equations"}

\begin{document}
\numberwithin{equation}{section}

\begin{abstract}
We consider the planar $N$-centre problem, with homogeneous potentials of degree $-\a<0$, $\a \in [1,2)$. We prove the existence of infinitely many collisions-free periodic solutions with negative and small energy, for any distribution of the centres inside a compact set.
The proof is based upon topological, variational and geometric arguments. The existence result allows to characterize the associated  dynamical system  with a symbolic dynamics, where the  symbols are the partitions of the $N$ centres in two non-empty sets.
\end{abstract}

\keywords{$N$-centre problem, chaotic motions, symbolic dynemics, Levi-Civita regularization}
\subjclass{Primary: 70F10, 37N05; Secondary: 70F15, 37J30}

\maketitle


\section{Introduction}
The $N$-centre problem consists in the study of the motion of a  test particle moving under the action of the gravitational force fields of $N$ fixed heavy bodies (the centres of the problem). In this paper we deal with the more general case of a Newtonian-like potential with homogeneity degree $-\a<0$, with $\a \in [1,2)$. The motion equation is
\beq\label{1}
\ddot{x}(t)=-\sum_{j=1}^N \frac{m_j}{\left|x(t)-c_j\right|^{\a+2}} \left(x(t)-c_j\right),
\eeq
where $x=x(t)$ denotes the position of the particle at the instant $t \in \R$, and $c_j$ ($j=1,\ldots,N$) is the position of the $j$-th centre. Introduced the potential
\[
V(x)=\sum_{j=1}^N \frac{m_j}{\a \left|x-c_j\right|^\a}, \qquad x \in \R^2 \setminus \{c_1,\ldots,c_N\},
\]
we can rewrite equation \eqref{1} as a Newton equation
\[
\ddot{x}(t)=\nabla V(x(t)),
\]
which possesses a hamiltonian structure, of Hamiltonian
\[
\frac{1}{2}|p|^2-V(q)=h(p,q).
\]

This paper concerns the existence of infinitely many collision-free periodic solutions with  \emph{negative energies}, for the planar planar $N$-centre problem; as a by-product of our construction, we will prove the occurrence of symbolic dynamics. The presence of chaotic trajectories has been established for positive energies in \cite{Bo1} and \cite{Kn}, while for small negative energies only perturbations of the two-centre cases have been treated (\cite{BoNe} and \cite{Di}).

\medskip

 Here we consider the case  of small negative energies $h<0$, and an arbitrary (even infinite) number of centres arbitrarily located inside a compact subset of the plane. A major difficulty of the negative energy case is that the  energy shells
\[
 \mathcal{U}_h:= \left\{ (x,v)\in \mathbb{R}^2\setminus \{c_1,\ldots,c_N\} \times \R^2: \frac{1}{2}|v|^2-V(x)=h\right\}
\]
 have a non empty boundary where the Jacobi metric degenerates. We shall seek trajectories as (non minimal) geodesics for the Jacobi metric and we shall exploit a broken geodesics method.  Our method allows the simultaneous treatise of singularity of degree $\a=1$ (Coulombic) and degree $\a>1$  and provides collision-free trajectories.

\medskip

To describe our main result, we need some notation. Let us consider the possible partitions of the set of centres $\{c_1,\ldots,c_N\}$ in two disjoint non-empty sets (non ordered:\(
\{\{c_1\},\{c_2,\ldots,c_N\}\}\simeq\{\{c_2,\ldots,c_N\},\{c_1\}\}
\)). There are exactly
\[
\frac{1}{2}\left(\binom{N}{1}+ \ldots+\binom{N}{N-1}\right)=\frac{1}{2}\left(\sum_{k=0}^N \binom{N}{k}-2\right)=2^{N-1}-1
\]
such partitions. Each partition will be labeled within the set of the labels
\[
\mathcal{P}:=\{P_j: j=1,\ldots, 2^{N-1}-1\}.
\]
It is convenient to distinguish those partitions which separate a single $c_j$ from the others:
\[
Q_j:=\{\{c_j\},\{c_1,\ldots,c_N\}\setminus \{c_j\}\} \qquad j=1,\ldots,N.
\]
This special kind of partitions  define a subset of labels
\[
\mathcal{P}_1:=\{Q_j \in \mathcal{P}: j=1,\ldots,N\} \subset \mathcal{P}.
\]
\subsection{Periodic solutions}\label{subsection:periodic}
To any finite sequence of $n$  symbols we associate an $n$-periodic bi-infinite sequence. We define the \emph{right shift} in $\mathcal{P}^n$ as
\begin{equation*}
T_r\colon  \mathcal{P}^n \to \mathcal{P}^n
 \colon (P_{j_1},P_{j_2},\ldots,P_{j_n}) \mapsto (P_{j_n},P_{j_1},\ldots, P_{j_{n-1}}),
\end{equation*}
and we say that \emph{$(P_{j_1},\ldots,P_{j_n}) \in \mathcal{P}^n$ is equivalent to $(P_{j_1}',\ldots, P_{j_n}') \in \mathcal{P}^n$} if there exists $m \in \mathbb{N}$ such that
\[
(P_{j_1}',\ldots, P_{j_n}')=T_r^m \left((P_{j_1},\ldots,P_{j_n}) \right).
\]
Our first goal consists in proving the existence of infinitely many periodic solutions at negative energies:

\begin{teor}\label{esistenza di soluzioni periodiche}
Let $\a \in [1,2)$, $c_1, \ldots, c_N \in \mathbb{R}^2$, $m_1,\ldots,m_N \in \R^+$.
There exists $\bar{h}<0$ such that for every $h \in (\bar{h},0)$, $n \in \mathbb{N}$ and $(P_{j_1},\ldots,P_{j_n}) \in \mathcal{P}^n$ there exists a periodic solution $x_{(P_{j_1},\ldots,P_{j_n})}$ of the $N$-centre problem \eqref{1} with energy $h$, which depends on $(P_{j_1},\ldots,P_{j_n})$ in the following way.
There exist $\bar{R},\bar{\d}>0$ (independent of $(P_{j_1},\ldots,P_{j_n})$) such that $x_{(P_{j_1},\ldots,P_{j_n})}$ passes alternatively outside and inside $B_{\bar{R}}(0)$, and
\begin{itemize}
\item if in $(t_1,t_2)$ the solution stays outside $B_{\bar{R}}(0)$ and \\ $x_{(P_{j_1},\ldots,P_{j_n})}(t_1), x_{(P_{j_1},\ldots,P_{j_n})}(t_2) \in \pa B_{\bar{R}}(0)$, then
$$
|x_{(P_{j_1},\ldots,P_{j_n})}(t_1)-x_{(P_{j_1},\ldots,P_{j_n})}(t_2)    |<\bar{\d}.
$$
\item in its $k$-th passage inside $B_{\bar{R}}(0)$, if $x_{(P_{j_1},\ldots,P_{j_n})}$ does not collide in any centre, then it separates the centres according to the partition $P_{j_k}$.
\end{itemize}
To be precise:
\begin{itemize}
\item[($i$)] if $\alpha \in (1,2)$ then $x_{(P_{j_1},\ldots,P_{j_n})}$ is collisions-free.
\item[($ii$)] if $\alpha=1$ there are three possibilities:
\begin{itemize}
\item[$a$)] either $x_{(P_{j_1},\ldots,P_{j_n})}$ is collisions-free.
\item[$b$)] or  $x_{(P_{j_1},\ldots,P_{j_n})}$ has a collision with one centre $c_j$, covers a certain trajectory, rebounds against a second  centre $c_k$ (it may happen that $c_j=c_k$) and comes back along the same trajectory.  Note that this is possible just if $n$ is even and $(P_{j_1},\ldots, P_{j_n})$ is equivalent to $(P_{j_1}',\ldots, P_{j_n}')$ such that
\begin{gather*}
P_{j_1}'=Q_j \in \mathcal{P}_1, \quad  P_{j_{n/2+1}}' =Q_k \in \mathcal{P}_1 \quad \text{and (if $n > 2$)}\\ P_{j_{n}}'=P_{j_2}',\ P_{j_{n-1}}'=P_{j_3}',\ \ldots, \ P_{j_{n/2+2}}'=P_{j_{n/2}}'.
\end{gather*}
\item[$c$)] or else $x_{(P_{j_1},\ldots,P_{j_n})}$ has a collision with one centre $c_j$, covers a certain trajectory, "rebounds" against the surface $\left\{x \in \R^2: V(x)=-h\right\}$ with null velocity and comes back along the same trajectory. This is possible just if $n$ is odd and $(P_{j_1},\ldots, P_{j_n})$ is equivalent to $(P_{j_1}',\ldots, P_{j_n}')$ such that
\begin{gather*}
P_{j_1}'=Q_j \in \mathcal{P}_1 \quad \text{and (if $n>1$)} \\
P_{j_n}'=P_{j_2}', \ P_{j_{n-1}}'=P_{j_3}',\  \ldots, \ P_{j_{(n+1)/2+1}}'=P_{j_{(n+1)/2}}.
\end{gather*}
\end{itemize}
\end{itemize}
\end{teor}

Of course, by varying both the number $n \in \mathbb{N}$ and the choice of the partitions $(P_{j_1},\ldots,P_{j_n}) \in \mathcal{P}^n$,  we find infinitely many periodic solutions. Let us also note that the symbol sequences of the collision solutions have a reflectional symmetry: by choosing non symmetric sequences we can rule out the occurrence of collisions.

\medskip

The following pictures represent the case ($i$) or ($ii-a$), ($ii-b$), ($ii-c$) respectively.

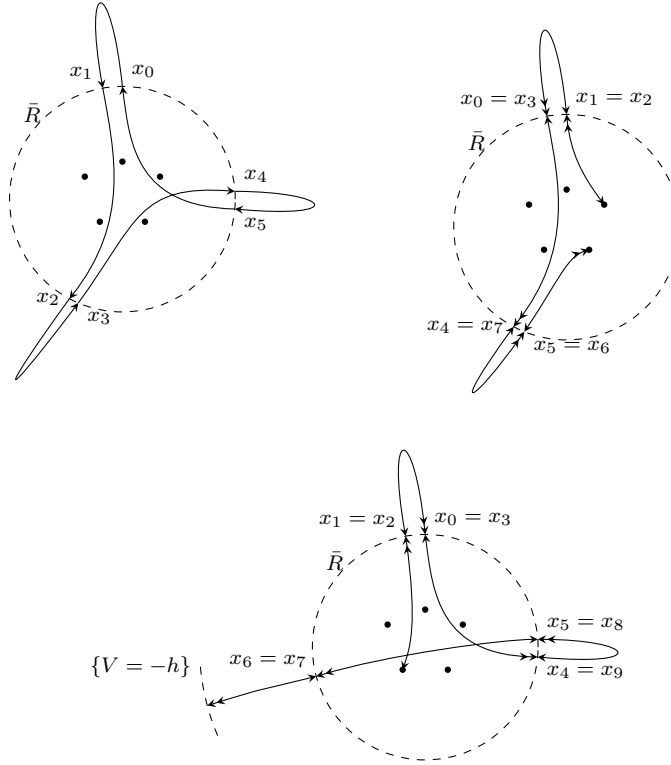
\begin{center}
\begin{tikzpicture}[>=stealth]
\draw[dashed] (0,0) circle (1.5cm);
\draw[font=\footnotesize] (130:1.5cm) node[anchor=east]{$\bar{R}$};
\filldraw (0.3,-0.3) circle (1pt)
          (0.5,0.3) circle (1pt)
          (0,0.5) circle (1pt)
          (-0.3,-0.3) circle (1pt)
          (-0.5,0.3) circle (1pt) ;
\draw[font=\footnotesize] (90:1.5cm) node[anchor=south west]{$x_0$};
\draw[->] (90:1.5cm)  .. controls (93:3cm) and (100:3cm)..(100:1.5cm);
\draw[font=\footnotesize] (100:1.5cm) node[anchor=south east]{$x_1$};
\draw[->] (100:1.5cm).. controls (-0.1,0.6) and (0.1,-0.3)..(242:1.5cm);
\draw[font=\footnotesize] (242:1.5cm) node[anchor=east]{$x_2$};
\draw[->] (242:1.5cm).. controls (238:3cm) and (239:3.4cm)..(247:1.5cm);
\draw[font=\footnotesize] (247:1.5cm) node[anchor=north west]{$x_3$};
\draw[->] (247:1.5cm).. controls (0.2,-0.3) and (0.2,0.2)..(4:1.5cm);
\draw[font=\footnotesize] (4:1.5cm) node[anchor=south west]{$x_4$};
\draw[->] (4:1.5cm).. controls (2:2.8cm) and (355:3cm)..(355:1.5cm);
\draw[font=\footnotesize] (355:1.5cm) node[anchor=north west]{$x_5$};
\draw[->] (355:1.5cm).. controls (0.2,-0.1) and (0.1,0.7)..(90:1.5cm);
\end{tikzpicture}
$\qquad$
\begin{tikzpicture}[>=stealth]
\draw[dashed] (0,0) circle (1.5cm);
\draw[font=\footnotesize] (130:1.5cm) node[anchor=east]{$\bar{R}$};
\filldraw (0.3,-0.3) circle (1pt)
          (0.5,0.3) circle (1pt)
          (0,0.5) circle (1pt)
          (-0.3,-0.3) circle (1pt)
          (-0.5,0.3) circle (1pt) ;
\draw[font=\footnotesize] (100:1.5cm) node[anchor=south east]{$x_0=x_3$};
\draw[<->>] (0.5,0.3)..controls (0.2,0.7) and (0.05,1)..(90:1.5cm);
\draw[font=\footnotesize] (90:1.5cm) node[anchor=south west]{$x_1=x_2$};
\draw[<->>] (90:1.5cm).. controls (93:3cm) and (100:3cm)..(100:1.5cm);
\draw[font=\footnotesize] (242:1.5cm) node[anchor= east]{$x_4=x_7$};
\draw[<->>] (100:1.5cm).. controls (-0.1,0.6) and (0.1,-0.3)..(242:1.5cm);
\draw[font=\footnotesize] (248:1.5cm) node[anchor=north west]{$x_5=x_6$};
\draw[<->>] (242:1.5cm)..controls (238:3cm) and (242:2.7cm)..(248:1.5cm);
\draw[<->>] (248:1.5cm)..controls (-0.1,-0.7) and (0,-0.4)..(0.3,-0.3);
\end{tikzpicture}
$\qquad$
\begin{tikzpicture}[>=stealth]
\draw[dashed] (0,0) circle (1.5cm);
\draw[font=\footnotesize] (130:1.5cm) node[anchor=east]{$\bar{R}$};
\filldraw[font=\footnotesize] (0.3,-0.3) circle (1pt)
          (0.5,0.3) circle (1pt)
          (0,0.5) circle (1pt)
          (-0.3,-0.3) circle (1pt)
          (-0.5,0.3) circle (1pt);
\draw[font=\footnotesize] (100:1.5cm) node[anchor=south east]{$x_1=x_2$}
                          (90:1.5cm) node[anchor=south west]{$x_0=x_3$}
                          (355:1.5cm) node[anchor=north west]{$x_4=x_9$}
                          (4:1.5cm) node[anchor=south west]{$x_5=x_8$}
                          (195:1.5cm) node[anchor=south east]{$x_6=x_7$};
\draw[<<->] (90:1.5cm).. controls (93:3cm) and (100:3cm)..(100:1.5cm);
\draw[<<->](4:1.5cm).. controls (2:2.8cm) and (355:3cm)..(355:1.5cm);
\draw[<<->] (355:1.5cm).. controls (0.2,-0.1) and (0.1,0.7)..(90:1.5cm);
\draw[<->>] (4:1.5cm)..controls (1,0.1) and (-0.5,-0.1)..(195:1.5cm);
\draw[<->>] (195:1.5cm)..controls (-2.3,-0.6) and (-2,-0.5)..(195:3cm);
\draw[<<->] (100:1.5cm)..controls (-0.1,0.3) and (-0.2,0)..(-0.3,-0.3);
\draw[dashed] (185:3cm) arc (185:205:3cm);
\draw[font=\footnotesize] (185:3cm) node[anchor=east]{$\{V=-h\}$};
\end{tikzpicture}
\end{center}

\subsection{Fixed ends problems}
To prove  Theorem \ref{esistenza di soluzioni periodiche} we shall make use of a broken geodesics argument, finally leading to a finite dimensional reduction.
A key step consists in solving fixed end problems having the desired topological characterization. This will yield to a constrained minimization for the Maupertuis' functional, the main difficulty being the possible interaction with the centres.  Similar minimization problems in the presence of topological constraints have been recently treated in the literature concerning the variational approach to the $N$-body problem (see e.g. \cite{BaTeVe1,Chen2, Chen3, FuGrNe}). We believe that the following intermediate result can be of independent interest.
\begin{teor}\label{teor_dinamica interna partizioni}
Let $\a \in [1,2)$, $c_1, \ldots, c_N \in \mathbb{R}^2$, $m_1,\ldots,m_N \in \R^+$.
There exist $\bar{h}<0$ and $R>0$ such that for every $h \in (\bar{h},0)$ and every pair of points $p_1,p_2 \in \pa B_R(0)$ and $P_j\in \mathcal{P}$, there exist $T>0$ and a solution $y_{P_j}(\cdot\,;p_1,p_2)$  of the $N$-centre problem \eqref{1} at energy $h$ such that $y_{P_j}(0)=p_1$, $y_{P_j}(T)=(p_2)$, $y_{P_j}(0,T)\subset B_R(0)$. Moreover:
\begin{itemize}
\item[($i$)] if $\a \in (1,2)$ then $y_{P_j}$ is collisions-free and self-intersections-free.
\item[($ii$)] if $\a=1$ we have to distinguish among
\begin{itemize}
\item[$a$)] $p_1 \neq p_2$; then $y_{P_j}$ is collisions-free and self-intersections-free.
\item[$b$)] $p_1=p_2$ and $P_j \in \mathcal{P} \setminus \mathcal{P}_1$; then $y_{P_j}$ is collisions-free and self-intersections-free.
\item[$c$)] $p_1=p_2$ and $P_j \in \mathcal{P}_1$; then $y_{P_j}$ can be a collisions-free and self-intersections-free solution, or can be an ejection-collision solution, with a unique collision with $c_j$.
\end{itemize}
\end{itemize}
Whenever it is collision free, then $y_{P_j}$ separates the centres according to the partition $P_j$.
\end{teor}

\subsection{Symbolic dynamics}\label{subsection:symbolic}
Let us consider the discrete metric space $S$ (endowed with the trivial distance: $d_1(s_j,s_k):= \d_{jk}$ $\forall s_j,s_k \in S$, , where $\d_{jk}$ is the Kronecker delta), and consider the bi-infinite sequences of elements of $S$:
\[
S^\Z:=\{(s_m)_{m \in \mathbb{Z}}: s_m \in S \ \forall m\}.
\]
It is a metric space with respect to the distance
\beq\label{distanza in S^Z}
d((s_m), (t_m)):= \sum_{m \in \Z} \frac{1}{2^{|m|}} d_1(s_m,t_m), \qquad \forall (s_m), (t_m) \in S^\Z.
\eeq
Of course, we can introduce a right shift in this space of sequences letting
\[
T_r((s_m)):=(s_{m+1}) \qquad \forall (s_m) \in S^\Z.
\]
\begin{defin}
Let $\Sigma$ be a metric space, $\sigma:\Sigma \to \Sigma$ a continuous map, $S$ a finite set. We say that the dynamical system  $(\Sigma,\sigma)$ has a \emph{symbolic dynamics with set of symbols $S$} if there exist a $\sigma$-invariant subset $\Pi$ of $\Sigma$ and a continuous and surjective map $\pi:\Pi \to S^\Z$ such that the diagram
\[
\xymatrix{
\Pi \ar[r]^\sigma \ar[d]^{\pi}  &\Pi \ar[d]^{\pi}\\
S^\Z \ar[r]^{T_r} & S^\Z
}
\]
commutes, i.e. the restriction $\sigma|_{\Pi}$ is topologically semi-conjugate to the right shift in the metric space $(S^\Z, d)$ ($d$ defined in \eqref{distanza in S^Z}).
\end{defin}
Let us write equation \eqref{1} as a first order autonomous Hamiltonian system:
\beq\label{sistema dinamico}
 \begin{cases}
\dot{x}(t)=v(t)\\
\dot{v}(t)=\n V(x(t))
 \end{cases} \Leftrightarrow \left(\begin{array}{c} \dot{x}(t) \\ \dot{v}(t)\end{array}\right) = \left(\begin{array}{cc} 0 & 1 \\ -1 & 0\end{array}\right)\left(\begin{array}{c} -\n V(x(t)) \\ v(t)\end{array}\right),
\eeq
whit Hamiltonian function given by the energy.
\begin{corol}\label{symbolic dynamic}
Let $\bar{h}$ be introduced in Theorem \ref{esistenza di soluzioni periodiche}, let $h \in (\bar{h},0)$.
Then there exist a subset $\Pi_h$ of the energy shell $\mathcal{U}_h$, a first return map $\mathfrak{R}:\Pi_h \to \Pi_h$, and a continuous and surjective map $\pi:\Pi_h \to \mathcal{P}^\Z$ such that the diagram
\[
 \xymatrix{
\Pi_h \ar[r]^{\mathfrak{R}} \ar[d]^\pi & \Pi_h \ar[d]^\pi \\
\mathcal{P}^\Z \ar[r]^{T_r} & \mathcal{P}^\Z,
}
\]
commutes; namely for every $h \in (\bar{h},0)$ the dynamical system associated with the $N$-centre problem on the energy shells $\mathcal{U}^h$ has a symbolic dynamics, with set of symbols $\mathcal{P}$.
\end{corol}

Although the planar $N$-centre problem appears as a simplified version of the restricted planar $(N+1)$-problem, in which the Coriolis' and the centrifugal forces are neglected, it is far away from being simple, except for the two-center problem, which is known to be integrable (see e.g. \cite{Whit}). To give an idea of its complexity, we list below some remarkable results which are related with ours (our intent is not to give an exhaustive bibliography, and we refer the interested reader to the quoted works and the references therein): in \cite{Bo1}, Bolotin proved the analytic non-integrability of the system for $N \geq 3$ on the energy shells for any $h> 0$. The question of analytic non-integrability has been faced also for the spatial problem by Knauf and Taimanov (\cite{KT}): they proved that an analytic integral which is independent with respect to the energy does not exist in case $n \geq 3$ and the energy is greater than some threshold $h_{th}$. The authors showed also the existence of smooth first integrals independent by the energy for both the planar and the spatial problem, with energy $h>0$ and $h>h_{th}$ respectively. Another crucial reference, always for positive energies, is the work of Klein and Knauf \cite{Kn}, which provides an accurate description of the scattering for a wide class of problems having singular potentials. The spatial case was treated in \cite{BoNe01}. As far as the negative energy case is concerned, the literature shows very few works,   and the most remarkable results are obtained with fairly restrictive assumptions: in \cite{BoNe} Bolotin and Negrini proved the occurrence of chaotic dynamics for the $3$-centre problem under the assumption of the third centre far away from the others two and a small absolute value of $h$; in \cite{Di} Dimare obtained a similar result for $h<0$, $|h|$ small enough, in the case when one centre has small mass with respect to the others.
In both papers the problem is approached in a perturbative setting.

\subsection{Plan of the paper}

Let us fix $\a \in [1,2)$, $c_1,\ldots,c_N \in \R^2$, $m_1,\ldots,m_N > 0$. In section \ref{problema normalizzato} we show  equivalence of the fixed energy problem for \eqref{1}  with small negative energies with the similar problem where the energy is normalized to $-1$ and the new centres lie
inside a ball of radius $\eps$. Hence we will deal with a rescaled potential
\[
V_{\eps}(y)= \sum_{k=1}^N \frac{m_k}{\a |y- c_k'|^\a}, \qquad \max_{1\leq k\leq N} |c_k'|<\eps.
\]

The advantage of the reformulation is that, if $\eps$ is chosen sufficiently small, outside a ball of radius $R>\eps>0$, the new problem is a small perturbation of the Kepler's problem with homogeneity degree $-\a<0$ (we will call it "$\a$-Kepler's problem").
This consideration leads to the search periodic solutions to
\beq\label{P}
\begin{cases}
\ddot{y}(t)=\n V_\eps(y(t)), \\
\frac{1}{2}|\dot{y}(t)|^2-V_\eps(y(t))=-1, \qquad \eps \in (0,\bar{\eps})
\end{cases}
\eeq
dividing the investigation inside/outside the ball of radius $R> \bar{\eps} > 0$. In section \ref{dinamica esterna} we will  find arcs of solutions to \eqref{P} lying in $\R^2 \setminus B_R(0)$ connecting any pair of points $(x_0,x_1) \in (\pa B_R(0))^2$, provided their distance is sufficiently small, via perturbative techniques. In section \ref{dinamica interna} we will study the dynamics inside the ball $B_R(0)$ and provide solutions of \eqref{P} which connect $x_1,x_2 \in \pa B_R(0)$ for every $x_1,x_2$. This will be made trough a variational approach under suitable topological constraints. Finally, in section \ref{incollamento}, we will collect the previous results to obtain periodic solutions of \eqref{P} which pass alternatively outside and inside $B_R(0)$, using a broken geodesics argument, through a finite dimensional reduction. Then, using the results of section \ref{problema normalizzato}, we will obtain a periodic solution of the original problem. Once we proved the existence Theorem \ref{esistenza di soluzioni periodiche}, we will focus on the symbolic dynamics, proving Corollary \ref{symbolic dynamic} in section \ref{dinamica simbolica}.

\subsection{Notations}
We will often identify a function $u$ with the its image  $u([a,b]) \subset \R^2$, with some abuse of notation. Throughout the paper $C$ will be a strictly positive constant which may refer to different quantities even in the same proof. Sometimes it will be necessary to define different constants $C_1,\ldots,C_m$; in this case we point out that constants defined in one proof are not defined outside that proof.\\
It is convenient to introduce the polar coordinates for a point $x \in \R^2$:
\[
x=r e^{i \t}, \qquad r>0 \text{ and } \t \in \R.
\]
The angle $\t$ is counted in counterclockwise sense, and $\t=0$ if $x=(1,0)$. For every continuous function $x:I \subset \R \to \R^2 \setminus \{0\}$, there exist continuous functions $r:I \to \R^+$ and $\t:I \to \R$ such that
\[
x(t)=r(t) e^{i \t(t)}.
\]
Dealing with the angular momentum of a $\mathcal{C}^1$ function $x$, we will write
\[
\mathfrak{C}_x(t):=|x(t)\land\dot{x}(t)|=|r^2(t)\dot{\t}(t)|
\]
We will use the notations $\| \cdot \|_{L^p([a,b])}$ for the $L_p\left([a,b],\R^2\right)$-norm and $\| \cdot \|_{H^1([a,b])}$ for the $H^1\left([a,b],\R^2\right)$-norm; when there will not be possibility of misunderstanding, we will briefly write $\| \cdot \|_p$ or $\| \cdot \|$, respectively. The symbol $\wc$ will denote the weak convergence in $H^1$.

\section{Preliminaries}\label{sezione 2}

Let us fix $\a \in [1,2)$, $c_1,\ldots,c_N \in \R^2$, $m_1,\ldots,m_N > 0$. and fix the origin in the center of mass. Here and in what follows $M:= \sum_{j=1}^N m_j$. In this subsection we prove that solving \eqref{1} with energy $h<0$ is equivalent to solve a rescaled $N$-centre problem on the energy level $-1$.
In this perspective the quadratic mean of the centers will replace the energy as a parameter. To be precise, we state the following elementary result.

\begin{prop}\label{problema normalizzato}
Let $x \in \mathcal{C}^2\left((a,b);\R^2\right)$ be a classical solution of \eqref{1} with energy $h<0$. Then the function
\begin{equation}\label{eq:scaling}
y(t)= \left( -h \right)^{\frac{1}{\a}} x\left(\left(-h\right)^{-\frac{\a+2}{2\a}} t\right), \qquad t \in \left(\left(-h\right)^{\frac{\a+2}{2\a}}a,\left(-h\right)^{\frac{\a+2}{2\a}}b\right)
\end{equation}
is a solution of energy $-1$ of a $N$-centre problem with centres
\[
c_j'= \left(-h\right)^{\frac{1}{\a}} c_j, \qquad j=1,\ldots,N.
\]
The converse holds true: let $y \in \mathcal{C}^2\left(\left(a',b'\right),\R^2\right)$ be a classical solution of energy $-1$ of a $N$-centres problem, with centres $c_j'$. Let us set
\[
c_j= \left(-h\right)^{-\frac{1}{\a}} c_j', \qquad j=1,\ldots,N.
\]
Then
\[
x(t)=\left(-h\right)^{-\frac{1}{\a}} y \left(\left(-h\right)^{\frac{\a+2}{2\a}}t\right), \qquad t \in \left(\left(-h\right)^{-\frac{\a+2}{2\a}} a', \left(-h\right)^{-\frac{\a+2}{2\a}}b'\right)
\]
is a classical solution of \eqref{1} with energy $h<0$.
\end{prop}

\begin{corol}\label{h--epsilon}
For every $\eps > 0$ there exists $\zeta(\eps) > 0$ such that if $-\zeta(\eps)<h<0$, then the centres $c_j'$ of the equivalent problem lye in $B_\eps(0)$. The function $\zeta$ is strictly decreasing in $\eps$.
\end{corol}
\begin{proof}
Given $\eps>0$ we find
\[
\zeta(\eps)=-\left(\frac{\eps}{\max_{1\leq j \leq N}|c_j|}\right)^\a. \qedhere
\]
\end{proof}

\begin{rem}\label{da eps a h}
Of course, periodic solutions of the problem \eqref{P} for every $\eps \in (0,\bar{\eps})$,  will correspond, via Proposition \ref{problema normalizzato}, to periodic solutions of \eqref{1} of energy $h=\zeta(\eps)$ for every $h \in (-\zeta(\bar{\eps}),0)$.
\end{rem}

As said, if $\eps$ is chosen sufficiently small, outside a ball of radius $R>\eps>0$ we can consider the new problem as a small perturbation of the $\a$-Kepler's problem: indeed let us consider the total potential: $V_\eps$; then, for $|y| \geq R>\eps$ and $\max_j |c_j'|<\eps$, we have the expansion (in  $\mathcal C^1(\R^2 \setminus B_R(0))$):

\[
\n V_\eps(y)= \sum_{j=1}^N \frac{m_j}{\left|y-c_j'\right|^{\a+2}} \left(y-c_j\right)=
\nabla \left(\frac{M}{\a |y|^\a}\right)+  \mathcal{W}_\eps(y)= \nabla \left(\frac{M}{\a |y|^\a}\right)+o(\eps),
\]
where
\[
\mathcal{W}_\eps(y) =\n \left(V_\eps(y)-\frac{M}{\a |y|^\a}\right)\Rightarrow \mathcal{W}_\eps(y)=\n W_\eps(y).
\]

\begin{rem}\label{remark su R}
If $y$ is a solution of $\ddot{y}=\n V_\eps(y)$ over an interval $I \subset \R$, there holds
\[
V_\eps(y(t)) \geq 1 \qquad \forall t \in I,
\]
so that to exploit the previous argument we have to check that, for every $\eps>0$ sufficiently small, there exists $R>0$ such that
\[
B_\eps(0) \subset B_R(0) \subset \left\lbrace y \in \R^2: V_\eps(y) \geq 1\right\rbrace.
\]
Of course, we can find such and $R$ independent of the choice of $\eps$, if $\eps$ is small enough.  Then, for $y \in B_R(0)$, and for every $j \in \left\lbrace 1,\ldots,N\right\rbrace $
\[
|y-c_j'|\leq R+\eps \Rightarrow V_\eps(y) \geq \frac{M}{\a (R+\eps)^\a},
\]
which is strictly greater than $1$ if and only if $R<\left(\frac{M}{\a}\right)^\frac{1}{\a}-\eps$. There exists $\eps_1>0$ such that
\[
0<\eps<\eps_1 \Rightarrow \left(\frac{M}{\a}\right)^\frac{1}{\a}-\eps> \eps.
\]
This argument shows that for every $\eps \in (0,\eps_1)$ there exists $R>0$ such that $\eps_1 < R < \left(\frac{M}{\a}\right)^\frac{1}{\a}-\eps_1$. Actually, we will make the further request
\[
\eps < \frac{R}{2} < R < \left(\frac{M}{\a}\right)^\frac{1}{\a}-\eps,
\]
which is satisfied for every $\eps \in (0, \eps_1 /2)$. \\
For reasons which appear clear in the section \ref{dinamica interna}, \emph{it is convenient to choose $R$ such that $\pa B_R(0)$ is the support of the circular solution of the $\a$-Kepler's problem with energy $-1$}; let $y(t) = R \exp{\left\lbrace i \omega t \right\rbrace }$; we find $R$ such that
\beq\label{pn1}
\ddot{y}(t)=-M \frac{y(t)}{|y(t)|^{\a+2}} \Leftrightarrow R \o^2 = \frac{M}{R^{\a+1}}.
\eeq
The conservation of the angular momentum $\mathfrak{C}_y(t)=|y(t) \land \dot{y}(t)|$ gives
\beq\label{pn2}
R^2 \omega = R \sqrt{2\left(-1+\frac{M}{\a R^\a}\right)}.
\eeq
Collecting \eqref{pn1} and \eqref{pn2} we obtain
\beq\label{scelta di R}
R:=\left(\frac{(2-\a)M}{2\a}\right)^{\frac{1}{\a}}.
\eeq
This is consistent with the previous restriction on $R$, if $\eps_1$ is sufficiently small (if this was not true, it is sufficient to replace $\eps_1$ with a smaller quantity).
\end{rem}

We end this section with an elementary but useful remark. Proposition \ref{problema normalizzato} and Corollary \ref{h--epsilon} enable us to switch from solutions $x$ of \eqref{1} with energy $h<0$, $|h|$ sufficiently small, to solutions $y$ of \eqref{P}. In this correspondence the topological properties of the solutions  with respect to the centres are obviously preserved.

\section{Outer dynamics}\label{dinamica esterna}
We are going to use a perturbation argument in order to find particular solutions of problem \eqref{P} lying in $\R^2 \setminus B_R(0)$, connecting pairs of neighbouring points of $\pa B_R(0)$ with a \emph{close to brake} arc. To be precise we will prove the following theorem.

\begin{teor}\label{teorema 0.1}
There exist $\d>0$ and $\eps_2>0$ such that for every $\eps \in (0,\eps_2)$, for every $p_0,p_1 \in \pa B_R(0):|p_1-p_0| < 2\d$, there exist $T>0$ and a unique solution $y_{\text{ext}}(\cdot\,;p_0,p_1;\eps)$ of \ref{P} such that
$|y(t)| > R$, for $ t \in (0,T)$ and
$y(0)=p_0$, $y(T)=p_1$. Moreover, $y$ depends in a $\mathcal C^1$ way on the endpoints $p_0$ and $p_1$.
\end{teor}

The proof requires some preliminary results. We start from our unperturbed problem
\beq\label{PB_e}
\begin{cases}
\ddot{y}(t)=-M\frac{y(t)}{|y(t)|^{\a+2}} \qquad &t \in [0,T], \\
\frac{1}{2}|\dot{y}(t)|^2-\frac{M}{\a |y(t)|^\a}=-1  \qquad &t \in [0,T],\\
|y(t)| > R  \qquad &t \in (0,T).\\
\end{cases}
\eeq

Let us solve the Cauchy problem
\[
\begin{cases}
\ddot{y}(t)=-M\frac{y(t)}{|y(t)|^{\a+2}} \\
y(0)=p_0=R \exp{\{i \t_0\}}, & \dot{y}(0)= \sqrt{2\left(-1+\frac{M}{\a R^\a}\right)}\left(\frac{p_0}{R}\right).
\end{cases}
\]
The trajectory returns at the point $p_0$ after a certain time $\bar{T}>0$, having swept the portion of the rectilinear brake orbit starting from $p_0$ and lying in $\R^2 \setminus B_R(0)$. Our aim is to catch the behaviour of the solutions under small variations of the boundary conditions. Hence, we consider
\beq\label{PC_e}
\begin{cases}
\ddot{y}(t)=-M\frac{y(t)}{|y(t)|^{\a+2}} \\
y(0)=p_0, & \dot{x}(0)=  \dot{r}_0 e^{i \t_0}+ R \dot{\t}_0 i e^{i \t_0},
\end{cases}
\eeq
where $\dot{r}_0$ is assigned as function of $\dot{\t}_0$ by means of the energy integral:
\[
\dot{r}_0=\dot{r}_0(\dot{\t}_0) = \sqrt{2\left(\frac{M}{\a R^\a}-1\right)-R^2 \dot{\t}_0^2}.
\]
We denote as $y(\cdot\,;\t_0,\dot{\t}_0)$ the solution of \eqref{PC_e}. For the brake orbit $y\left(\cdot\,; \t_0,0\right)$ it results
\[
\t(t;\t_0,0)\equiv \t_0 \qquad \forall t \in [0,\bar{T}].
\]
\textbf{Let us fix $p_0 \in \pa B_R(0)$}. We define
\begin{align*}
\psi:& \Theta \times I \to \R^2 \\
&(\dot{\t}_0, T) \mapsto y(T;\t_0,\dot{\t}_0),
\end{align*}
where $\Theta \times I \subset S^1 \times \R$ is a neighbourhood of $(0,\bar{T})$ on which $\psi$ is well defined (such a neighbourhood exists). We can assume
\beq\label{su Theta}
\max\left\{ \sup_{(\dot{\t}_0,T) \in \Theta \times I} 4|T \dot{\t}_0| , \sup_{(\dot{\t}_0,T) \in \Theta \times I} \left|\left(\frac{\a}{M}\right)^{\frac{2}{\a}} R^2 T \dot{\t}_0\right|\right\} < \frac{\pi}{2},
\eeq
otherwise it is sufficient to replace $\Theta \times I$ with a smaller neighbourhood.

\begin{lemma}\label{lemma 0.1}
The Jacobian  of $\psi$ in $(0,\bar{T})$ is invertible.
\end{lemma}

\begin{proof}
Since the $\a$-Kepler's problem is invariant under rotations, it isn't restrictive suppose $\t_0=\pi/2$, so that $\exp{\{i\t_0\}}= (0,1)=:e_2$. The function  $\psi \in \mathcal{C}^1(\Theta \times I)$  satisfies
\[
\frac{\pa \psi}{\pa T} \left(0,\bar{T}\right)= \dot{y}(\bar{T};p_0,0) = - \sqrt{2\left(\frac{M}{\a R^\a}-1\right)}e_2.
\]
Hence the Jacobian matrix of $\psi$ is invertible in $(0,\bar{T})$ if
\[
\left\langle \frac{\pa \psi}{\pa \dot{\t}_0}(0,\bar{T}), e_1 \right\rangle \neq 0,
\]
where $e_1:=(1,0)$.  By continuous dependence with respect to initial data we have, for every $(\dot{\t}_0,T) \in \Theta \times I$, and for every $t \in [0,T]$,
\beq\label{0.1**}
r(t;\t_0,\dot{\t}_0)  \geq \frac{R}{2}.
\eeq
We use the conservation of the angular momentum: for every $t \in [0,T]$ there holds
\[
\mathfrak{C}_y:=\mathfrak{C}_y(t)= |\mathfrak{C}_y(0)|, \Leftrightarrow \mathfrak{C}_y = r^2(t) |\dot{\t}(t)|= R^2 |\dot{\t}_0|.
\]
Assume $\dot{\t}_0 >0$; one has
\[
\t(t;\t_0,\dot{\t}_0)= \frac{\pi}{2}+ \int_0^t \frac{d\t}{ds}(s) \,ds= \frac{\pi}{2}+\int_0^t \frac{R^2 \dot{\t}_0}{r^2(s)}\,ds.
\]
If $(\dot{\t}_0, T) \in \Theta \times I$, from \eqref{su Theta}, \eqref{0.1**}, and the fact that $r(s)\leq \left(M/\a\right)^{1/\a}$, it follows
\[
\frac{\pi}{2} < \left(\frac{\a}{M}\right)^{\frac{2}{\a}} R^2 T \dot{\t}_0 +\frac{\pi}{2} \leq \t(T;\t_0,\dot{\t}_0) \leq 4 T \dot{\t}_0 + \frac{\pi}{2} < \pi.
\]
The function $\cos\left(\cdot\right)$ being decreasing over $(\pi/2,\pi)$, we obtain

\begin{multline*}
\left\langle \frac{\psi(\dot{\t}_0,\bar{T})-\psi(0,\bar{T})}{\dot{\t}_0} , e_1\right\rangle  = \frac{r(\bar{T};\t_0,\dot{\t}_0) \cos \left(\t(\bar{T},\t_0,\dot{\t}_0)\right)}{\dot{\t}_0} \\
 \leq \frac{ r(\bar{T};\t_0,\dot{\t}_0) \cos \left(\left(\frac{\a}{M}\right)^{\frac{2}{\a}} R^2 T \dot{\t}_0 +\frac{\pi}{2}\right)}{\dot{\t}_0}
= -r(\bar{T};\t_0,\dot{\t}_0)\left(\frac{\a}{M}\right)^{\frac{2}{\a}} R^2 T + o(\dot{\t}_0^2)<0,
\end{multline*}
for $\dot{\t}_0 \to 0$. Passing to the limit for $\dot{\t}_0 \to 0$ the inequality is preserved. Since the same argument works for $\dot{\t}_0<0$, the thesis follows.
\end{proof}

The previous discussion has to be refined in order to include the variations of the potential due to the presence of the centres, which are now included in the  $\eps$-disk. Recall that we fixed $p_0 \in \pa B_R(0)$. We know that
\[
\lim_{\eps \to 0^+} V_\eps(y) = \frac{M}{\a |y|^\a} \qquad \text{uniformly in $y \in \R^2 \setminus B_R(0)$}.
\]
So we define
\begin{align*}
\Psi : &\Theta \times I \times \left[0,\frac{\eps_1}{2}\right) \times \pa B_R(0) \to \R^2 \\
&(\dot{\t}_0,T,\eps,p_1) \mapsto y(T;\t_0,\dot{\t}_0;\eps)-p_1,
\end{align*}
where $y(\cdot \,;\t_0,\dot{\t}_0;\eps)$ is the solution of
\beq\label{PC_e2}
\begin{cases}
\ddot{y}(t)=\nabla V_\eps(y(t)) \\
y(0)=p_0, & \dot{y}(0)= \dot{r}_\eps e^{i \t_0}+ R \dot{\t}_0 i e^{i \t_0},
\end{cases}
\eeq
and
\[
\dot{r}_\eps=\dot{r}_\eps(\dot{\t}_0;\eps) = \sqrt{2\left(V_\eps(p_0)-1\right)-R^2 \dot{\t}_0^2}.
\]

\begin{lemma}\label{lemma 0.2}
There exist $\d > 0$ and $0<\eps_2<\eps_1/2$ such that for every $\eps \in (0,\eps_2)$, for every $p_1 \in \pa B_R(0):|p_1-p_0| < 2\d$, there exists a unique solution $y(\cdot\,;\t_0,\dot{\t}_0;\eps)$ of \eqref{PC_e2} defined in $[0,T]$ for a certain $T$ and satisfying
\beq\label{condizioni aggiuntive}
\begin{split}
\frac{1}{2}|\dot{y}(t;\t_0,\dot{\t}_0;\eps)|^2-V_{\eps}(y(t;\t_0,\dot{\t}_0;\eps))=-1 \qquad t \in [0,T],\\
|y(t;\t_0,\dot{\t}_0;\eps)|>R \qquad t \in (0,T), \qquad y(T;\t_0,\dot{\t}_0;\eps)=p_1.
\end{split}
\eeq
Moreover, it is possible to choose $\delta$ and $\eps_2$ independent on $p_0 \in \pa B_R(0)$.
\end{lemma}

\begin{proof}
We apply the implicit function theorem to the function $\Psi$, which is $\mathcal{C}^1$ in the variables $\dot{\t}_0$ and $T$ for the differentiable dependence of the solutions by time and initial data. There holds
\begin{equation*}
\Psi(0,\bar{T},0,p_0)=0, \quad
\frac{\pa \Psi}{\pa \dot{\t}_0} \left(0,\bar{T},0,p_0\right)= \frac{\pa \psi}{\pa \dot{\t}_0}\left(0,\bar{T}\right), \quad
\frac{\pa \Psi}{\pa T} \left(0,\bar{T},0,p_0\right)= \frac{\pa \psi}{\pa T}\left(0,\bar{T}\right),
\end{equation*}
so that from Lemma \ref{lemma 0.1} we deduce that the Jacobian matrix of $\Psi$ with respect to $(\dot{\t}_0,T)$ is invertible; hence the assumptions of the implicit function theorem are satisfied, and we can find a neighbourhood $\Theta' \times J \subset \Theta \times I$ of $(0,\bar{T})$, a neighbourhood $[0,\eps_2) \times B_{2\d}(p_0) \subset [0,\eps_1/2) \times \R^2$ of $(0,p_0)$,
and  a unique function $\eta:[0,\eps_2) \times B_{2\d}(p_0) \to \Theta' \times J$ such that
\begin{align*}
& 1) \eta\left(0,p_0\right)=(0,\bar{T}),\\
& 2) \Psi\left(\eta_1(\eps,p_1),\eta_2(\eps,p_1),\eps,p_1\right)=0 \quad \text{for every $(\eps,p_1) \in [0,\eps_2) \times B_{2\d}(p_0)$},\\
& 3) \Psi(\dot{\t}_0,T,\eps,p_1)=0 \quad \text{with $(\dot{\t}_0,T,\eps,p_1) \in \Theta' \times J \times [0,\eps_2) \times B_{2\d}(p_0)$} \\
& \qquad \Rightarrow (\dot{\t}_0,T)=\eta(\eps,p_1).
\end{align*}

This means that, if we fix $\eps \in (0,\eps_2)$, for every $p_1 \in \pa B_R(0) \cap B_{2\d}(p_0)$, we can find a solution $y(\cdot\,;\t_0, \dot{\t}_0 ;\eps)$ of \eqref{PC_e2}. This solution has constant energy $-1$ because of the definition of $\dot{r}_\eps$; moreover, $y(T;\t_0,\dot{\t}_0;\eps)=p_1$.
We remark that outside $B_R(0)$ the potential $V_\eps$ is a small perturbation of the $\a$-Kepler's one, so that $|y(t;\t_0,\dot{\t}_0;\eps)|>R$ for every $t \in (0,T)$.
It remains to prove that one can choose $\d$ and $\eps_2$ independent on $x_0$. This is a consequence of the proof of the implicit function theorem: the wideness of the neighbourhood of $(0,p_0)$ in $[0,\eps_1/2)\times \R^2$ in which we can guarantee the definition of the implicit function depends on the norm of
\[
\left(J_{(\dot{\t}_0, T)} \Psi (0,\bar{T},0,p_0)\right)^{-1},
\]
and for every $p_0 \in \pa B_R(0)$ this matrix is the same up to rotations.
\end{proof}

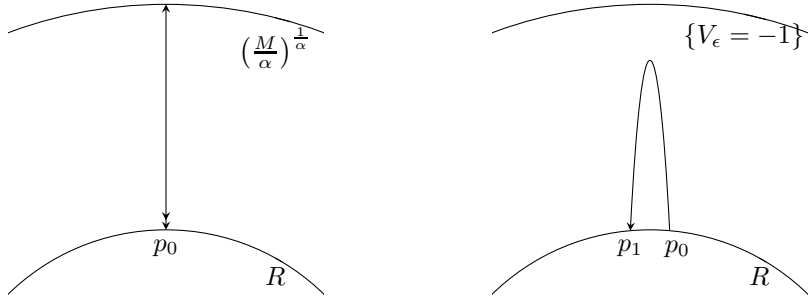
\begin{center}
\begin{tikzpicture}[scale=3,>=stealth]
\clip (-0.7,0.7) rectangle (0.7,2.3);
\draw (0,0) circle (1cm);
\draw (59:1cm)--(61:1cm) node[anchor=north]{$R$};
\draw (90:1cm) node[anchor=north]{$p_0$};
\draw (0,0) circle (2cm);
\draw (74:2cm)--(76:2cm) node[anchor=north]{$\left(\frac{M}{\alpha}\right)^{\frac{1}{\alpha}}$};
\draw[<<->] (90:1cm)--(90:2cm) ;
\end{tikzpicture} $\hspace{2cm}$ \begin{tikzpicture}[scale=3,>=stealth]
\clip (-0.7,0.7) rectangle (0.7,2.3);
\draw (0,0) circle (1cm);
\draw (59:1cm)--(61:1cm) node[anchor=north]{$R$};
\draw (83:1cm) node[anchor=north]{$p_0$};
\draw (95:1cm) node[anchor=north]{$p_1$};
\draw (0,0) circle (2cm);
\draw (75:2cm)--(78:2cm) node[anchor=north]{$\{V_\eps=-1\}$};
\draw[->] (85:1 cm).. controls (89:2cm) and (91:2cm)..(95:1 cm);
\end{tikzpicture}
\end{center}
The picture represents the portion of the brake rectilinear solution for the $\a$-Kepler's problem in comparison with a "perturbed" solution obtained for the potential $V_\eps$ via the implicit function theorem.

\vspace{1 em}

Theorem \ref{teorema 0.1} is a straightforward consequence of this lemma. The solutions obtained are uniquely determined and depends in a smooth way on the ends $p_0$  and $p_1$.
\section{Inner dynamics}\label{dinamica interna}

In this section we are going to seek arcs of solutions of \eqref{P} connecting two points $p_1, p_2 \in \pa B_R(0)$ and lying inside the disk $B_R(0)$. We admit the case $p_1=p_2$. Close to the center of the ball, the potential $V_\eps$ cannot be seen as a small perturbation of the $\a$-Kepler's one, so that we are lead to use variational methods rather than perturbative techniques. The first step is to introduce a suitable functional whose critical points are weak solutions of \eqref{P}; this will be made in subsection \ref{funzionale di Maupertuis}. Our trajectories will be local minimizers of the Maupertuis' functional or, equivalently, of the Jacobi length. More precisely,  in subsection \ref{main result} we will determine weakly closed sets in which we will minimize the functional and we will state the main theorem of the section. It will be proved in \ref{minimizzazione} and \ref{assenza di collisioni}; in the first one we will show that the direct method of the calculus of variations applies to provide weak solutions of \eqref{P}, while in the latter one we will describe the behaviour of the solutions, proving in particular the absence of collisions in case $\a \in (1,2)$. The case $\a=1$ deserves a special analysis.\\
In what follows we will consider $\eps \in (0,\eps_1/2)$ fixed, and we will write $c_j$ instead of $c_j'$ to ease the notation. We are going to seek solutions of
\beq\label{PI}
\begin{cases}
\ddot{y}(t)= \n V_\eps(y(t)) \qquad &t \in [0,T],\\
\frac{1}{2}|\dot{y}(t)|^2-V_\eps(y(t))=-1 \qquad &t \in [0,T],\\
|y(t)|<R \qquad  &t \in (0,T) ,\\
y(0)=p_1,\quad y(T)=p_2,
\end{cases}
\eeq
with $p_1,p_2 \in \pa B_R(0)$, and $T>0$ to be determined.

\subsection{The Maupertuis' principle}\label{funzionale di Maupertuis}

Dealing with a singular potential, we introduce the spaces on non-collision paths
\begin{multline}\label{H^_C.L.}
\wh{H}_{p_1 p_2}\left([a,b]\right):=\left\lbrace u \in H^1\left([a,b],\R^2\right): u(a)=p_1,\  u(b)=p_2, \right. \\
\left. u(t) \neq c_j \text{ for every $t\in [a,b]$, for every $j \in \left\lbrace 1,\ldots,N\right\rbrace$ }\right\rbrace ,
\end{multline}
and
\[
H_{p_1 p_2}\left([a,b]\right):=\left\lbrace u \in  H^1\left([a,b],\R^2\right): u(a)=p_1,\  u(b)=p_2\right\rbrace
\]
 (briefly $\wh{H}$ and $H$). Let us note that, since $H^1\left([a,b],\R^2\right)$ is embedded in $\mathcal{C}\left([a,b],\R^2\right)$, the definitions are well posed. We point out that, defining
\begin{multline*}
\mathfrak{Coll}_{p_1 p_2}\left([a,b]\right):=\left\lbrace u \in  H^1\left([a,b],\R^2\right): u(a)=p_1,\  u(b)=p_2, \right. \\
\left. \exists t \in [a,b]: u(t) = c_j \text{ for some $j \in \left\lbrace 1,\ldots,N\right\rbrace$ }\right\rbrace ,
\end{multline*}
the set of colliding functions in $H^1\left([a,b],\R^2\right)$ which connect $p_1$ with $p_2$, there holds
\[
H_{p_1 p_2}\left([a,b]\right)=\wh{H}_{p_1 p_2}\left([a,b]\right) \cup \mathfrak{Coll}_{p_1 p_2}\left([a,b]\right)
\]
and $H_{p_1 p_2}\left([a,b]\right)$ is the closure of $\wh{H}_{p_1 p_2}\left([a,b]\right)$ in the weak topology of $H^1$.
Let us define the \emph{Maupertuis' functional}
\begin{equation*}
M_h\left([a,b];\cdot\right) \colon  H_{p_1 p_2}\left([a,b]\right) \to \R \cup \left\lbrace +\infty\right\rbrace \qquad M_h\left([a,b];u\right) = \frac{1}{2}\int_a^b |\dot{u}(t)|^2 \,dt \int_a^b \left(V(u(t))+h\right)\,dt.
\end{equation*}
We will often write $M_h$ instead of $M_h\left([a,b];\cdot\right)$. If $M_h([a,b];u)>0$ both its factors are strictly positive and it makes sense to set
\beq\label{omega}
\o^2:=\frac{\int_a^b \left(V(u)+h\right)}{\frac{1}{2} \int_a^b |\dot{u}|^2}.
\eeq
The Maupertuis' functional is differentiable over $\wh{H}$ (seen as an affine space on $H_0^1$), and its critical points, suitably reparametrized, are solutions to our fixed energy problem (see \cite{AmCZ}).

\begin{teor}\label{teorema 4.1}
Let $u \in \wh{H}_{p_1 p_2}\left([a,b]\right)$ be a  critical point of $M_h$ at a positive level, i.e.
\[
dM_h\left([a,b];u\right)[v]=0 \ \  \forall v \in H_0^1\left([a,b],\R^2\right), \quad \text{and} \quad M_h\left([a,b];u\right)>0,
\]
and let $\o$ be given by \eqref{omega}.  Then $x(t):=u(\o t)$ is a classical solution of
\beq \label{P_h}
\begin{cases}
\ddot{x}(t) = \n V(x(t)) \qquad &t \in \left[\frac{a}{\o},\frac{b}{\o}\right],\\
\frac{1}{2}|\dot{x}(t)|^2-V(x(t))=h \qquad  &t \in \left[\frac{a}{\o},\frac{b}{\o}\right],\\
x\left(\frac{a}{\o}\right)=p_1, \quad &x\left(\frac{b}{\o}\right)=p_2,
\end{cases}
\eeq
while $u$ itself is a classical solution of
\beq \label{P_u}
\begin{cases}
\o^2 \ddot{u}(t) = \n V(u(t)) \qquad &t \in [a,b],\\
\frac{1}{2}|\dot{u}(t)|^2-\frac{V(u(t))}{\o^2}=\frac{h}{\o^2} \qquad &t \in [a,b],\\
u(a)=p_1, \quad u(b)=p_2.
\end{cases}
\eeq
\end{teor}

\begin{rem}\label{viceversa}
The converse of Theorem \ref{teorema 4.1} is also true: if $x \in \mathcal{C}^2\left([a',b'],\R^2\right)$ is a collisions-free solution  of \eqref{P_h}, setting $\o=1/(b'-a')$ and $u(t):= x(t/\o)$, $u$ is a classical solution of \eqref{P_u} defined in $[a'/(b'-a'),b'/(b'-a')]=:[a,b]$ and hence a critical point of $M_h\left([a,b];\cdot\right)$ at a strictly positive level. Also, the identity
\[
\omega^2=\frac{\int_a^b \left(V(u)+h\right)}{\int_a^b |\dot{u}|^2}.
\]
is fulfilled.
\end{rem}

In order to use variational methods it is worth working in $H$ rather than in $\wh{H}$, for $\wh{H}$ isn't weakly closed. The disadvantage is that
we will need some \textit{ad hoc} argument to rule out the occurrence of collisions and to apply Theorem \ref{teorema 4.1} and to obtain a classical solution of the motion equation. Nevertheless, although collision minimizers are not true critical points of the Maupertuis' functional on $H$, the following result allows to recover the conservation of the energy.

\begin{lemma}\label{conservazione dell'energia}
If $u \in H$ is a local minimizer of $M_h$, then
\[
\frac{1}{2} |\dot{u}(t)|^2-\frac{V(u(t))}{\o^2}= \frac{h}{\o^2}  \qquad \text{a.e.} \ t \in [a,b]
\]
\end{lemma}

\begin{rem}
The lemma says that the energy is constant almost everywhere even if $u$ has collisions. Of course, in this case $u$ could be not of class $\mathcal{C}^1$.
It is a classical result and it is a consequence of the extremality of $u$ with respect to time reparametrization keeping the ends fixed: if $\f \in \mathcal{C}_c^\infty((a,b),\R)$, setting $u_\l (t):=u(t+\l \f(t))$, there holds
\[
\left. \frac{d}{d\l} M_h(u_\l)\right|_{\l=0}=0.
\]
\end{rem}

\paragraph{\textbf{The Jacobi metric.}}
The original Maupertuis' principle states that solutions of \eqref{P_h} are obtained, after a suitable reparametrization, as non-constant critical points of the functional
\[
L_h(u)=L_h\left([a,b];u\right):= \int_a^b \sqrt{|\dot{u}(t)|^2\left(V(u(t))+h\right)}\,dt,
\]
which is defined on those $u \in H_{p_1 p_2}\left([a,b]\right)$ such that $V(u(t)) \geq -h$ for every $t \in [a,b]$. We define
\[
H^*=H^*_{p_1,p_2}\left([a,b]\right):= \left\lbrace u \in H : V(u(t)) > -h, |\dot{u}(t)|>0 \text{ for every $t \in [a,b]$}\right\rbrace,
 \]
so that the domain of $L_h$ is the closure of $H^*_{p_1 p_2}\left([a,b]\right)$ in the weak topology of $H^1$.\\
The functional $L_h(\gamma)$ has an important geometric meaning: it is
the length of the curve parametrized by $\gamma \in H^*$ with respect to the Jacobi's metric:
\[
g_{ij}(x):=\left(V(x)+h\right) \d_{ij}, \quad
\d_{ij}= \begin{cases}
          1 & i =j, \\
          0 & i \neq j
          \end{cases}, \quad i,j=1,2.
\]

 The explicit expression of the reparametrization needed to pass from critical points of $L_h$ to solution of \eqref{P_h} is the following. For $u \in H^1\left([a,b];\R^2\right)$ let us set
\[
\Gamma_u:=\left\lbrace \left([a',b'],f\right): f:[a',b'] \to [a,b], \text{$f \in \mathcal{C}^1\left([a',b'],\R\right)$}
\text{ and increasing, such that $u \circ f \in H^1\left([a',b'],\R^2\right)$}\right\rbrace .
\]
Since $L_h$ is a length, it is invariant under reparametrization: for every $u \in H^*_{p_1 p_2}\left([a,b]\right)$, for every $\left([a',b'],f\right) \in \Gamma_u$ it results
\[
L_h\left([a,b];u\right)= L_h\left([a',b'];u\circ f\right).
\]

\vspace{1 em}

It is well known that $u \in H^*$ is a critical point of $L_h$ with respect to variations with compact support if and only if $u$ solves the Euler-Lagrange equation
\beq\label{EL}
\frac{d}{dt} \left( \dot{u} \sqrt{\frac{V(u(t))+h}{|\dot{u}(t)|^2}}\right)-\frac{1}{2}\sqrt{\frac{|\dot{u}(t)|^2}{V(u(t))+h} } \n V(u(t))=0
\eeq
for almost every $t \in [a,b]$.

\begin{teor}\label{teorema su L}
Let $u \in H^*_{p_1 p_2}\left([a,b]\right) \cap \wh{H}_{x_1,x_2}\left([a,b]\right)$ be a non-constant critical point of $L_h\left([a,b];\cdot\right)$. \\
Then there exist a reparametrization $x$ of $u$ which is a classical solution of $\eqref{P_h}$ in a certain interval $[0,T/\sqrt{2}]$.
\end{teor}
\begin{proof}
The function $u$ is a collisions-free weak solution of \eqref{EL}, hence it is a strong solution. Define, for $t \in [a,b]$,
\[
\t(t):=\int_a^t \sqrt{\frac{|\dot{u}(z)|^2}{V(u(z))+h}}\,dz,
\]
and set $T=\t(b)$. It results $\left([0,T],\t\right) \in \Gamma_u$ and for every $s \in [0,T]$ (we denote with "$'$" the differentiation with respect to the new parameter $s$)
\[
\frac{dt}{ds}(s)= \left( \left.  \frac{d\t}{dt}(t)   \right|_{t=t(s)}\right)^{-1} = \sqrt{ \frac{ V(u (t (s)))+h }{ \left|\dot{u}(t(s))\right|^2} }.
\]
With this change of variable, setting $y(s)=u(t(s))$, the \eqref{EL} becomes
\[
\frac{1}{t'(s)} \frac{d}{ds} \left( \frac{y'(s)}{t'(s)} t'(s)\right)-\frac{1}{2t'(s)} \n V(y(s))=0,
\]
i.e.
\[
y''(s) = \frac{1}{2} \n V(y(s)).
\]
Setting $x(s):=y( \sqrt{2} s)$, $x$ is a solution of the first equation in \eqref{P_h} in $[0,T/\sqrt{2}]$. As far as the second equation is concerned, for every $s \in [0,T/\sqrt{2}]$
\[
|y'(s)|^2= |\dot{u}(t(s)) t'(s)|^2 = V(y(s))+h \Rightarrow \frac{1}{2}|x'(s)|^2=V(x(s))+h
\]
which completes the proof.
\end{proof}

\paragraph{\textbf{Relationship between $L_h$ and $M_h$.}}
It is convenient to establish a correspondence between minimizers of $M_h$ at positive level and minimizers of $L_h$. This can be done through the simple inequality
\beq\label{5 Maupertuis}
L_h^2(u) = \left( \int_a^b \sqrt{|\dot{u}|^2\left(V(u)+h\right)}\right)^2 \leq \int_a^b |\dot{u}|^2 \int_a^b \left(V(u)+h\right) = 2 M_h(u),
\eeq
for every $u \in H^*$. The equality holds true if and only if there exists $\l \in \R$ such that for almost every $t \in [a,b]$
\[
|\dot{u}(t)|^2=\l \left(V(u(t))+h\right).
\]

\begin{prop}\label{minimo M->L}
Let $u \in H^* \cap \wh{H}$ be a non-constant minimizer of $M_h$. Then $u$ is a minimizer of $L_h$ in $H^* \cap \wh{H}$.
\end{prop}
\begin{proof}
Since $u$ is a critical point of $M_h$ in $\wh{H}$ at a positive level, from Theorem \ref{teorema 4.1} we know that
\[
|\dot{u}(t)|^2=\frac{2}{\o^2}\left(V(u(t))+h\right).
\]
Hence there is equality in \eqref{5 Maupertuis}. If there existed $v \in H^* \cap \wh{H}$ such that $L_h(v)<L_h(u)$, then we could reparametrize $v$ to obtain a function (still denoted by $v$) satisfying
\[
|\dot{v}(t)|^2= \left(V(v(t))+h\right)
\]
(apply the argument in Theorem \ref{teorema su L}). So,
\[
0<2M_h(v)=L_h^2(v)<L_h^2(u)=2M_h(u),
\]
a contradiction.
\end{proof}

\begin{prop}\label{minimo L->M}
If $u \in H^* \cap \wh{H}$ is a non-constant minimizer of $L_h$ then, up to reparametrization, $u$ is a minimizer of $M_h$ on $H^* \cap \wh{H}$.
\end{prop}
\begin{proof}
We can assume from the beginning that there exists $\l \in \R$ such that for every $t \in [0,1]$
\[
|\dot{u}(t)|^2=\l \left(V(u(t))+h\right).
\]
Otherwise it is sufficient to perform a suitable reparametrization. Then there is equality in \eqref{5 Maupertuis}. Assume by contradiction that there existed $v \in H^* \cap \wh{H}$ such that $M_h(v)<M_h(u)$. We can reparametrize $v$ so that there is equality in \eqref{5 Maupertuis}. Therefore, we would deduce
\[
L_h^2(v)=2M_h(v) <2M_h(u) = L_h^2(u),
\]
a contradiction.
\end{proof}

\paragraph{\textbf{Final comments.}}
In this paper we will use both $M_h$ and $L_h$. It is clear that the Maupertuis' functional $M_h$ is easier to treat, so that it is convenient use it whenever possible. On the other hand the geometric meaning of the functional $L_h$ will come useful. Indeed the couple set-metric given by
\[
N=\left\lbrace x \in \R^2: V(x) > -h\right\rbrace, \quad g_{ij}(x) = \left(V(x)+h\right)\d_{ij}
\]
(called the Hill's region) defines a Riemaniann manifold and we will take advantage of this structure, in spite to the degeneration of the metric on the boundary of the Hill region. More precisely, we will often make use of the following facts:

\begin{itemize}
\item[1)] If $\gamma:[a,b] \to N$ is a piecewise differentiable curve, it is always possible to reparametrize it so that the length of the tangent vector
\[
\sqrt{|\dot{\gamma}(t)|^2\left(V(\gamma(t))+h\right)}
\]
is a constant $C \in \R^+ \cup \left\lbrace 0\right\rbrace$.
\item[2)] If a piecewise differentiable curve $\gamma:[a,b] \to N$, with parameter proportional to arc length, has length less or equal to the length of any other piecewise differentiable curve joining $\g(a)$ and $\g(b)$, then $\g$ is a geodesic. In particular, $\g$ is regular (recall that a geodesic is a curve satisfying the geodesics equation).
\item[3)] Let $p \in N$. We say that a subset $A \subset N$ is a \emph{totally normal neighbourhood} of $p$ if for every $p_1, p_2 \in \bar{A}$ there exists a unique minimizing geodesic $\g$ joining $p_1$ and $p_2$. If this geodesic is contained in $A$, we say that $A$ is a \emph{strongly convex} neighbourhood.\\
For any $p \in N$ there exist a totally normal neighbourhood $U$ of $p$. It is possible to choose $U$ in such a way that $U$ is strongly convex. If $\gamma$ is the minimizing geodesic connecting $p_1$ and $p_2$ in $U$, $\gamma$ depends smoothly on $p_1$ and $p_2$.
\end{itemize}

Furthermore we will strongly use the fact that, on contrarily to $M_h$, \emph{the functional $L_h$ is additive}. This is essential for the proof of the following Proposition.

\begin{prop}\label{localizzazione dei minimi_M}
Let $u \in H_{p_1 p_2}\left([a,b]\right)$ be a minimizer of $L_h\left([a,b];\cdot\right)$, let $[c,d] \subset [a,b]$. Then $u|_{[c,d]}$ is a minimizer of $L_h\left([c,d];\cdot\right)$ in $H_{u(c) u(d)}\left([c,d]\right)$. Moreover,
if $u$ is a minimizer of $M_h\left([a,b];\cdot\right)$ in $H_{p_1 p_2}\left([a,b]\right)$, then, for any  subinterval $[c,d] \subset [a,b]$, the restriction $u|_{[c,d]}$ is a minimizer of $M_h\left([c,d];\cdot\right)$ in $H_{u(c),u(d)}([c,d))$.
\end{prop}

\subsection{The existence theorem}\label{main result}

As said earlier, in order to find weak solutions of \eqref{PI},  we are going to minimize the Maupertuis' functional with some topological constraints. To this aim, the first step is to introduce suitable (possibly weakly closed) sets of functions. Let us fix $[a,b] \subset \R$ and $p_1,p_2 \in \pa B_R(0)$, $p_1 = R \exp{\left\lbrace i \t_1\right\rbrace }$, $p_2= R \exp{\left\lbrace i\t_2\right\rbrace}$. The paths in $\wh{H}$ can be classified according to their winding numbers with respect to each centre. This can be done by artificially closing it, in the following way:
\[
\Gamma(t):=\begin{cases}
                  \begin{cases}
                  u(t) & t \in [a,b]\\
                  R e^{i(t-b+\t_2)} & t \in (b,b+\t_1+2\pi-\t_2)
                  \end{cases} & \text{if $\t_1<\t_2$}\\
                  \, u(t) \qquad t \in [a,b] & \text{if $\t_1=\t_2$}\\
                  \begin{cases}
                  u(t) & t \in [a,b]\\
                  R e^{i(t-b+\t_2)} & t \in (b,b+\t_1-\t_2)
                  \end{cases} & \text{if $\t_1>\t_2$},
           \end{cases}
\]
i.e. if $p_1 \neq p_2$ we close the path $u$ with the arc of $\pa B_R(0)$ connecting $p_2$ and $p_1$ in counterclockwise sense. Then it is well defined the usual winding number
\[
\textrm{Ind}\left(u([a,b]),c_j\right)= \frac{1}{2\pi i}\int_{\Gamma} \frac{dz}{z-c_j}.
\]
Given $l=(l_1,\ldots,l_N) \in \Z^N$, a  component of $\wh{H}$ is of the form
\[
\wh{\mathfrak{H}}_l:=  \left\lbrace  u \in \wh{H}: \textrm{Ind}\left(u([a,b]),c_j\right)=l_j \quad \forall j =1,\ldots, N\right\rbrace .
\]
\begin{rem}
1) In general $\wh{\mathfrak{H}}_l$ may contain paths with self-intersections. Actually, $\wh{\mathfrak{H}}_l$ contains self-intersections-free paths lying completely in $B_R(0)$ if and only if $l_j \in \left\lbrace 0,1\right\rbrace$ for every $j$.\\
2) For every $l \in \Z^N$ the set $\wh{\mathfrak{H}}_l$ is not weakly closed in $H^1$.
\end{rem}
In the next subsection it will be useful to work on sets containing some self-intersections-free paths. For this reason we consider $l \in \mathbb{Z}_2^N$ instead of $l \in \Z^N$ and set
\[
\wh{H}_l := \left\lbrace  u \in \wh{H}: \text{Ind}\left(u([a,b]),c_j\right) \equiv l_j \mod 2 \quad \forall j=1,\ldots,N\right\rbrace,
\]
namely we collect together the components having winding numbers having the same parity with respect to each centre. We also assume that
\beq\label{scelta di l 2}
\exists j,k \in \left\lbrace 1,\ldots,N\right\rbrace,\  j \neq k, \text{ such that } l_j \neq l_k \mod 2.
\eeq
With this choice of $l$, if $u \in \wh{H}_l$ then $u$ has to pass through the ball $B_\eps(0)$ which contains the centres. In particular $u \in \wh{H}_l$ cannot be constant even if $p_1=p_2$, so that all the results stated in subsection \ref{funzionale di Maupertuis} hold true even in this case. From now on, we will say that $l \in \Z_2^N$ is a \emph{winding vector}.

\vspace{1 em}

In order to succeed in minimizing, we need to close those sets with respect to the weak $H^1$ topology. To this aim, we need to allow collisions with the centres. For $j \in \left\lbrace 1,\ldots,N\right\rbrace $, let us set
\begin{multline*}
\mathfrak{Coll}_l^j := \left\lbrace u \in H: \text{Ind}\left(u([a,b]),c_k\right) \equiv l_k \mod 2\;, \forall k \in \left\lbrace 1,\ldots,j-1,j+1,\ldots, N\right\rbrace, \right.\\
\left.  \text{ and there exists}\;  t \in [a,b]: u(t)=c_j\right\rbrace .
\end{multline*}
A path $u \in \mathfrak{Coll}_l^j$ behaves as a path of $\wh{H}_l$ with respect to $c_k$ for
$k \in \left\lbrace 1,\ldots,j-1,j+1,\ldots, N\right\rbrace$ and collides in $c_j$ at a certain instant.
Analogously, for $j_1,j_2 \in \left\lbrace 1,\ldots,N\right\rbrace$ we define
\begin{multline*}
\mathfrak{Coll}_l^{j_1,j_2} = \left\lbrace u \in H: \text{Ind}\left(u\left([a,b]\right),c_k\right)\equiv l_k \mod 2 \;,  \forall k \in \left\lbrace 1,\ldots, N\right\rbrace \setminus \left\lbrace j_1,j_2\right\rbrace\;,\right.\\
\left.  \text{ and there are } \; t_1,t_2 \in [a,b]: u(t_1)=c_{j_1}, u(t_2)=c_{j_2}\right\rbrace ,
\end{multline*}
the set of the paths behaving as paths of $\wh{H}_l$ with respect to $c_k$ for
$k \in \left\lbrace 1,\ldots N\right\rbrace \setminus \left\lbrace j_1,j_2\right\rbrace $ and colliding in $c_{j_1}$ and $c_{j_2}$; in the same way
\begin{align*}
& \mathfrak{Coll}_l^{j_1,j_2,j_3} := \dots ,\\
& \vdots \\
&  \mathfrak{Coll}_l^{1,\ldots,N} = \mathfrak{Coll}^{1,\ldots,N} :=\left\lbrace u \in H: \text{$u$ collides in each centre}\right\rbrace .
\end{align*}
Finally, we name
\[
\mathfrak{Coll}_l:= \bigcup_{j=1}^N  \mathfrak{Coll}_l^j \cup \bigcup_{1\leq j_1<j_2\leq N} \mathfrak{Coll}_l^{j_1,j_2} \cup \cdots \cup \mathfrak{Coll}_l^{1,\ldots,N}.
\]
\begin{prop}
The set
\[
H_l:= \wh{H}_l \cup \mathfrak{Coll}_l
\]
is weakly closed in $H^1\left([a,b],\R^2\right)$.
\end{prop}
\begin{proof}
Let $\left(u_n\right) \subset H_l$, $u_n \wc u$ in $H^1$. Since the weak convergence in $H^1$ implies the uniform one, if $u$ has a collision
\[
\left(u_n\right) \subset H_l \Rightarrow u \in \mathfrak{Coll}_l.
\]
If $u$ is collisions-free, the uniform convergence implies the existence of $n_0 \in \mathbb{N}$ such that
\[
u_n \in \wh{H}_l \quad \forall n \geq n_0 \Rightarrow u \in \wh{H}_l. \qedhere
\]
\end{proof}

To complete the choice of suitable sets, it is convenient to add a further requirement: since we search functions lying in $B_R(0)$, let us set
\begin{gather*}
\wh{K}_l=\wh{K}_l^{p_1 p_2}([a,b]):= \left\lbrace u \in \wh{H}_l: |u(t)| \leq R \ \forall t \in [a,b]\right\rbrace  \\
K_l=K_l^{p_1 p_2}([a,b]):=\left\lbrace u \in H_l: |u(t)| \leq R \ \forall t \in [a,b]\right\rbrace .
\end{gather*}

\begin{prop}\label{K_l debolmente chiuso}
The set $K_l$ is weakly closed in $H^1\left([a,b],\R^2\right)$.
\end{prop}
\begin{proof}
$K_l$ is a subset of the weakly closed set $H_l$, and it is stable under uniform convergence.
\end{proof}

Some examples of paths: the first path is a collisions-free path with winding vector $(0,0,1,1,0)$; the second one is a collision path of $K_l$ with $l=(0,1,1,0,0)$ or $l=(1,1,1,0,0)$; the third one is a path of $K_l$ with $l=(0,0,0,0,0)$, which does not satisfy \eqref{scelta di l 2}.
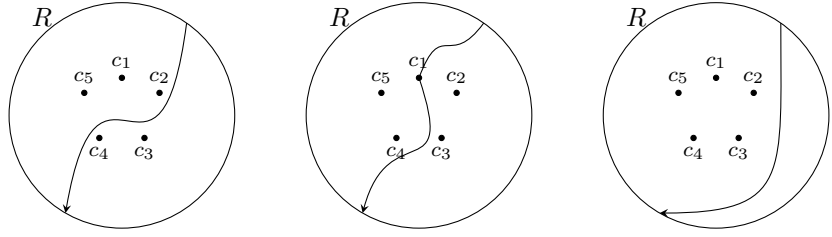
\begin{center}
\begin{tikzpicture}[>=stealth]
\draw (0,0) circle (1.5cm);
\draw[->] (55:1.5cm).. controls (0.5,-1.5) and (-0.3,1.3)..(240:1.5cm);
\filldraw[font=\footnotesize] (0.3,-0.3) circle (1pt) node[anchor=north]{$c_3$}
          (0.5,0.3) circle (1pt) node[anchor=south]{$c_2$}
          (0,0.5) circle (1pt) node[anchor=south]{$c_1$}
          (-0.3,-0.3) circle (1pt) node[anchor=north]{$c_4$}
          (-0.5,0.3) circle (1pt) node[anchor=south]{$c_5$};
\draw (135:1.5cm) node[anchor=south]{$R$};
\end{tikzpicture}
$\qquad$
\begin{tikzpicture}[>=stealth]
\draw (0,0) circle (1.5cm);
\draw (55:1.5cm)..controls (0.4,0.6) and (0.3,1.3)..(0,0.5);
\draw[->] (0,0.5)..controls (0.5,-1) and (-0.4,0)..(240:1.5cm) ;
\filldraw[font=\footnotesize] (0.3,-0.3) circle (1pt) node[anchor=north]{$c_3$}
          (0.5,0.3) circle (1pt) node[anchor=south]{$c_2$}
          (0,0.5) circle (1pt) node[anchor=south]{$c_1$}
          (-0.3,-0.3) circle (1pt) node[anchor=north]{$c_4$}
          (-0.5,0.3) circle (1pt) node[anchor=south]{$c_5$};
\draw (135:1.5cm) node[anchor=south]{$R$};
\end{tikzpicture}
$\qquad$
\begin{tikzpicture}[>=stealth]
\draw (0,0) circle (1.5cm);
\draw[->] (55:1.5cm)..controls (55:1.5cm |- 0,-0.8) and (1,0 |- 240:1.5cm)..(240:1.5cm);
\filldraw[font=\footnotesize] (0.3,-0.3) circle (1pt) node[anchor=north]{$c_3$}
          (0.5,0.3) circle (1pt) node[anchor=south]{$c_2$}
          (0,0.5) circle (1pt) node[anchor=south]{$c_1$}
          (-0.3,-0.3) circle (1pt) node[anchor=north]{$c_4$}
          (-0.5,0.3) circle (1pt) node[anchor=south]{$c_5$};
\draw (135:1.5cm) node[anchor=south]{$R$};
\end{tikzpicture}
\end{center}

The main result of this section is the following theorem.

\begin{teor}\label{teo dinamica interna}
There exists $\eps_3>0$ such that for every $\eps \in (0,\eps_3)$, $p_1,p_2 \in \pa B_R(0)$ and $l \in \Z_2^N$ satisfying \eqref{scelta di l 2}, there exist $T>0$ and a solution $y_{l}(\cdot\,;p_1,p_2;\eps) \in K_l^{p_1 p_2}([0,T])$ of problem \eqref{PI}, which is a reparametrization of a local minimizer of the Maupertuis' functional in $K_l^{p_1 p_2}([0,1])$. Moreover:
\begin{itemize}
\item[($i$)] if $\a \in (1,2)$ then $y_l$ is collisions-free and self-intersections-free.
\item[($ii$)] if $\a=1$ we have to distinguish among
\begin{itemize}
\item[$a$)] $p_1 \neq p_2$; then $y_l$ is collisions-free and self-intersections-free.
\item[$b$)] $p_1=p_2$ and $l$ is such that there exist $j_1,j_2,k_1,k_2 \in \{1,\ldots,N\}$:
\[
l_{j_1}=l_{j_2} \equiv 0 \mod 2 \quad l_{k_1}=l_{k_2} \equiv 1 \mod 2;
\]
then $y_l$ is collisions-free and self-intersections-free.
\item[$c$)] $p_1=p_2$ and $l$ such that there exists $j \in \{1,\ldots,N\}$:
\beq\label{105}
l_1=\cdots=l_{j-1}=l_{j+1}=\cdots=l_N \neq l_j \mod 2;
\eeq
then $y_l$ can be collisions-free and self-intersections-free or can be an ejection-collision solution, with a unique collision with one centre $c_j$.
\end{itemize}
\end{itemize}
\end{teor}

\begin{rem}
The statement motivates us to say that an element $l \in \Z_2^N$ is a \emph{collision winding vector} if it satisfies the \eqref{105}.
Let us also observe that the case ($ii$-$b$) makes sense just for $N \geq 4$.
\end{rem}

The proof consists in an application of Theorem \ref{teorema 4.1}, at least for the cases ($i$), ($ii$-$a$), ($ii$-$b$). We will check that all its assumptions are satisfied in the next two subsections; in the latter one, we will also discuss the classification.
\medskip
We recall the
\begin{defin}
An \emph{ejection-collision solution} of an equation
\[
\ddot{x}(t)=\n V(x(t))
\]
is a continuous function $x: I\subset \R \to \R^2$ such that
\begin{itemize}
\item there exists a collision set $T_c(x)\subset I$ such that for every $t^* \in T_c(x)$ there holds $x(t^*)=c_k$ for some $k=1, \ldots, N$,
\item the restriction $x|_{I \setminus T_c(x)}$ is a classical solution of
\[
\ddot{x}(t)=\n V(x(t)),
\]
\item the energy is preserved trough collisions,
\item at a collision instant, the trajectory is reflected:
\[
x(t+t^*)=x(t^*-t) \qquad \forall t^* \in T_c(x), \forall t \in I \setminus T_c(x).
\]
\end{itemize}
\end{defin}

Before proceeding into the proof, we translate Theorem \ref{teo dinamica interna} in the language of partitions. To do this, we note that if $u \in \wh{K}_l$ is self-intersections-free then it separates the centres in two different groups, which are determined by the particular choice of $l \in \Z_2^N$; namely, a self-intersections-free path in a class $\wh{K}_l$ induces a partition of the centres in two sets. Since we are assuming \eqref{scelta di l 2}, these sets are both non-empty. Hence it is well-defined an application $\mathcal{A}:\{l \in\Z_2^N: l \text{ satisfies \eqref{scelta di l 2}}\} \to \mathcal{P}$ which associate to a winding vector
\[
l=(l_1,\ldots,l_N) \text{ with } \begin{cases}
                                  l_k\equiv 0 \mod 2 & k \in A_0 \subset \{1,\ldots,N\}\\
                                  l_k\equiv 1 \mod 2 & k \in A_1 \subset \{1,\ldots,N\}
                                  \end{cases}
\]
the partition
\[
\mathcal{A}(l):=\{\{c_k:l_k \in A_0\},\{c_k:l_k \in A_1\}\}.
\]
This map is surjective but non injective, since for each couple $l,\wt{l} \in \mathbb{Z}_2^N$ such that
\[
l_k \neq \wt{l}_k \mod 2 \qquad \forall k=1,\ldots,N,
\]
then $\mathcal{A}(l)=\mathcal{A}(\wt{l})$. \\
Now it is natural to define
\begin{gather*}
\wh{K}_{P_j}=\wh{K}_{P_j}^{p_1 p_2}([a,b]):= \left\lbrace u \in \wh{K}_l: l \in \mathcal{A}^{-1}(P_j) \text{ and $u$ is self-intersections-free}\right\rbrace ,\\
K_{P_j}=K_{P_j}^{p_1 p_2}([a,b]):=\left\lbrace u \in K_l: l \in \mathcal{A}^{-1}(P_j) \text{ and $u$ is self-intersections-free}\right\rbrace .
\end{gather*}
They are respectively the set of the paths which connect $p_1$ and $p_2$ dividing the centres according to the partitions $P_j$, and its closure in the weak topology of $H^1$. \\
From Theorem \ref{teo dinamica interna}, we obtain
\begin{corol}\label{dinamica interna partizioni}
Let $\eps_3$ be introduced in Theorem \ref{teo dinamica interna}. For every $\eps \in (0,\eps_3)$, $p_1,p_2 \in \pa B_R(0)$ and $P_j\in \mathcal{P}$, there exist $T>0$ and a solution $y_{P_j}(\cdot\,;p_1,p_2;\eps) \in K_{P_j}^{p_1 p_2}([0,T])$ of problem \eqref{PI}, which is a reparametrization of a local minimizer of the Maupertuis' functional $M_{-1}$ in $K_{P_j}^{p_1 p_2}([0,1])$. Moreover:
\begin{itemize}
\item[($i$)] if $\a \in (1,2)$ then $y_{P_j}$ is collisions-free and self-intersections-free.
\item[($ii$)] if $\a=1$ we have to distinguish among
\begin{itemize}
\item[$a$)] $p_1 \neq p_2$; then $y_{P_j}$ is collisions-free and self-intersections-free.
\item[$b$)] $p_1=p_2$ and $P_j \in \mathcal{P} \setminus \mathcal{P}_1$; then $y_{P_j}$ is collisions-free and self-intersections-free.
\item[$c$)] $p_1=p_2$ and $P_j \in \mathcal{P}_1$; then $y_{P_j}$ can be a collisions-free and self-intersections-free solution, or can be an ejection-collision solution, with a unique collision with $c_j$.
\end{itemize}
\end{itemize}
\end{corol}

\subsection{Minimization inside $B_R(0)$}\label{minimizzazione}

Let us fix $l \in \mathbb{Z}_2^N$ satisfying \eqref{scelta di l 2}, and consider the restriction of the Maupertuis' functional $M_{-1}$ to the set $K_l$. In this subsection we are going to provide weak solutions of \eqref{PI} applying the direct method of the calculus of variations to $M_{-1}$. We fixed $[a,b]=[0,1]$, $p_1,p_2 \in \pa B_R(0)$; we will write $M$ and $L$ instead of $M_{-1}$ and $L_{-1}$, respectively.
\begin{rem}\label{l-eps}
In the statement of Theorem \ref{teo dinamica interna} the value $\eps_3$ depends neither on $p_1,p_2 \in \pa B_R(0)$, nor on $l \in \mathbb{Z}_2^N$, while here we fixed $l$ before finding $\eps_3$. Actually, once we found $\eps_3$, we will see that it is independent on $l$.
\end{rem}
The following statements are  by now standard results and can be proved by routine applications of Poincar\'e inequality, Fatou's lemma and weak compactness arguments (see, for instance, \cite{Vethesis,TeVe,BaFeTe}).

\begin{lemma}\label{lemma 1.4}
Assume \eqref{scelta di l 2} holds for  $l \in \Z_2^N$. There exists a constant $C > 0$ such that
\[
M(u) \geq C > 0 \qquad \forall u \in K_l.
\]
\end{lemma}

\begin{proof}
If $u \in K_l$, for every $j = 1, \ldots, N$ and every $t \in [0,1]$, we have
\[
|u(t)-c_j| \leq R+\eps  \Rightarrow V_\eps(u(t)) \geq \frac{M}{\a(R+\eps)^\a}.
\]
We required $\eps < R < \left(\frac{M}{\a}\right)^\frac{1}{\a}-\eps$, so there exists $\l_1>0$ such that
\[
R=\left(\frac{M}{\a}\right)^\frac{1}{\a}-\eps-\l_1.
\]
Thus for every $t \in [0,1]$
\[
V_\eps(u(t))-1 \geq \frac{M}{\a \left(\left(\frac{M}{\a}\right)^\frac{1}{\a}-\l_1\right)} -1 =: C > 0 \Rightarrow \int_0^1 \left(V_\eps(u)-1\right) \geq C > 0.
\]
Therefore the proof will be complete when we show the existence of $C>0$ such that, for all $u \in K_l$, there holds
\beq\label{*11.5}
\|\dot{u}\|_2 \geq C.
\eeq
Assume not: then there exists $\left(u_n\right) \subset K_l$ such that $\| \dot{u}_n\|_2 \to 0$. In particular $(\|\dot{u}_n\|_2) \subset \R$ is bounded. The sequence $(\|u_n\|_2)$ is bounded, too:
\[
\int_0^1 |u_n(t)|^2\,dt \leq R^2.
\]
Then the sequence $(u_n)$ is bounded in $H^1$, and this implies that up to subsequence $u_n \wc v \in K_l$. Note that $v$ must be a constant function.
Thus $v(0)=p_1=v(1)=p_2$. It is sufficient to note that, thanks to \eqref{scelta di l 2}, $v$ has to cross the ball $B_\eps(0)$ and therefore performs at least a distance $R-\eps$ and cannot be constant.
\end{proof}

\begin{prop}\label{metodo diretto}
Let $p_1,p_2 \in \pa B_R(0)$ (it is admissible $p_1=p_2$). Let $l \in \Z_2^N$ satisfying \eqref{scelta di l 2}. Then there exists a minimum of $M$ on $K_l$ at a positive level.
\end{prop}

\begin{proof}
Apply the direct method of the calculus of variations to the functional $M$ defined on $K_l$: use Proposition \ref{K_l debolmente chiuso} and Lemma \ref{lemma 1.4}, together with routine arguments of lower semi-continuity and coercivity.
\end{proof}

Let $l \in \Z_2^N$ satisfying \eqref{scelta di l 2} be fixed. If we show that the minimizer $u \in K_l$ is such that $|u(t)|<R$ for every $t \in (0,1)$, we can say that for every $\f \in \mathcal{C}_c^\infty([0,1],\R^2)$ there holds
\[
\left. \frac{d}{d\l} M(u+\l \f)\right|_{\l=0} = 0,
\]
so that $u$ is a critical point of $M$ at a positive level. In order to prove that $|u(t)|<R$, we follow the ideas in \cite{BaTeVe}; before proceeding, it is convenient recall a well known property of the solutions of the $\a$-Kepler's problem.

\begin{prop}\label{su Keplero}
Let  $\a \in [1,2)$ and let $x:(a,b) \subset \R \to \R^2$ be a collision solution for the $\a$-Kepler's problem with energy $h<0$:
\[
\lim_{t \to b^-} x(t)=0.
\]
Then the angular momentum $\mathfrak{C}_x$ of $x$ is $0$.
\end{prop}
\begin{proof}
In polar coordinates the energy is
\[
\frac{1}{2}\dot{r}^2(t)+\frac{\mathfrak{C}_x^2}{2r^2(t)}-\frac{M}{\a r^\a(t)}=h \qquad \forall t \in (a,b).
\]
In particular
\[
h-\frac{\mathfrak{C}_x^2}{2r^2(t)}+\frac{M}{\a r^\a(t)} \geq 0 \qquad \forall t \in (a,b),
\]
but if $\mathfrak{C}_x \neq 0$ then
\[
\lim_{t \to b^-} h-\frac{\mathfrak{C}_x^2}{2r^2(t)}+\frac{M}{\a r^\a(t)}=-\infty,
\]
a contradiction. Necessarily $\mathfrak{C}_x=0$.
\end{proof}

Let us term
\[
T_R(u):=\left\lbrace t \in [0,1]: |u(t)|=R\right\rbrace , \quad T_{R/2}^+(u):=\left\lbrace t \in [0,1]: |u(t)| > \frac{R}{2}\right\rbrace
\]
A connected component of $T_R(u)$ is an interval (possibly a single point) $[t_1,t_2]$ with $t_1 \leq t_2$. The complement $T_{R/2}^+(u) \setminus T_R(u)$ is the union of a finite or countable number of open intervals.

\begin{lemma}\label{prop del minimo}
A minimizer $u \in K_l$ of $M$ has the following properties:
\begin{itemize}
\item[($i$)] If $(a,b)$ is a connected component of $T_{R/2}^+(u) \setminus T_R(u)$, then $u|_{(a,b)}$ is of class $\mathcal{C}^2$ and is a solution of
\[
\o^2 \ddot{u}(t)=\n V_\eps(u(t)), \quad \text{where} \quad \o^2:= \frac{\int_0^1 \left(V_\eps(u)-1\right) }{\frac{1}{2}\int_0^1 |\dot{u}|^2}.
\]
\item[($ii$)] If $[t_1,t_2]$ is a connected component of $T_R(u)$, then $\t|_{(t_1,t_2)}$ is $\mathcal{C}^2$, strictly monotone, and solves
\beq\label{equazione per theta}
\ddot{\t}(t)=\frac{1}{R \o^2} \left\langle \n V_\eps(R e^{i \t(t)}),i e^{i\t(t)}\right\rangle .
\eeq
\item[($iii$)] If $[t_1, t_2]$ is a connected component of $T_R(u)$, and $(a,b)$ is a connected component of $T_{R/2}^+$ such that $[t_1,t_2] \subset (a,b)$, then one of the following situations occurs:
\begin{itemize}
\item[$a$)] $t_1<t_2$ and $u \in \mathcal{C}^1\left((a,b)\right)$,
\item[$b$)] $t_1=t_2$ and $u \in \mathcal{C}^1\left((a,b)\right)$,
\item[$c$)] $t_1=t_2$ and $\dot{u}(t_1^-) \neq \dot{u}(t_1^+)$; in such a case $u$ undergoes a radial reflection, i.e.
\[
\dot{r}(t_1^-) = -\dot{r}(t_1^+) \neq 0 \quad \text{and} \quad \dot{\t}(t_1^-)=\dot{\t}(t_1^+).
\]
\end{itemize}
\item[($iv$)] There exist $\eps_3 >0$ and $\tau >0$ such that, if $\eps \in (0,\eps_3)$, for $t_3$ and $t_4$ satisfying
\[
|u(t_3)| = R, \quad |u(t_4)|=\frac{R}{2}, \quad \frac{R}{2}<|u(t)|<R \quad \forall t \in \begin{cases} (t_3,t_4) & \text{if $t_3<t_4$}\\ (t_4,t_3) & \text{if $t_3>t_4$} \end{cases},
\]
there holds $|t_4-t_3|\leq \tau$.
\end{itemize}
\end{lemma}
\begin{proof}
($i$) It is a consequence of the minimality of $u$ with respect to variations of $u_n$ with compact support in $(a,b)$. These variation are compatible with the constraint $\{x \in \R^2:R/2 \leq |x| \leq R\}$.\\
($ii$) For $t \in (t_1,t_2)$, the energy integral becomes
\beq\label{energia in polari}
R^2 \dot{\t}^2(t)=-\frac{2}{\o^2}+\frac{2}{\o^2} V_\eps\left(R e^{i \t(t)}\right) \qquad \forall t \in [t_1,t_2];
\eeq
as a consequence $\t \in \mathcal{C}^2((t_1,t_2))$. Since $V_\eps\left(R\exp\{i \t\}\right) > 1$ for every $\t \in [0,2\pi]$, equation \eqref{energia in polari} implies that $\dot{\t}(t) \neq 0$ for every $t \in (t_1,t_2)$. To get \eqref{equazione per theta} it is sufficient to differentiate \eqref{energia in polari} with respect to $t$.\\
($iii$) See the proof of Proposition 3.6 in \cite{BaTeVe}.\\
($iv$) In polar coordinates the energy integral reads
\beq\label{di1}
\frac{1}{2}\dot{r}^2(t)+ \frac{\mathfrak{C}_u^2(t)}{2 r^2(t)} - \frac{V_\eps\left(r(t) e^{i \t(t)}\right)}{\o^2}=-\frac{1}{\o^2} \qquad \forall t \in [0,1].
\eeq
It results
\[
\frac{2}{\o^2}\left(-1+V_\eps\left(r(t)e^{i \t(t)}\right)\right)- \frac{\mathfrak{C}_u^2(t)}{r^2(t)} \geq \frac{2}{\o^2}\left(-1+ \frac{M}{\a (R+\eps)^\a}\right) + o(\eps);
\]
The last equality is due to the fact that if we makes $\eps \to 0^+$, $V_\eps$ uniformly converges in the circular crown $R/2\leq |x|\leq R$ to the potential of the Kepler's problem with homogeneity degree $-\a$. In particular, since $u$ has to pass through the ball $B_\eps(0)$, which collapses in the origin, the angular momentum of $u$ uniformly converges over the interval $[t_3,t_4]$ (or $[t_4,t_3]$) to $0$ (see Proposition \ref{su Keplero}).
\medskip
From \eqref{di1} we infer
\[
|t_4-t_3| \leq \int_{R/2}^R \frac{dr}{\sqrt{\frac{2}{\o^2}\left(-1+ \frac{M}{\a(R+\eps)^\a}\right)+o(\eps)}}.
\]
Since $-1+\frac{M}{\a(R+\eps)^\a}>0$ for every $\eps \in (0,\eps_1/2)$, there exists $0<\eps_3\leq \eps_1/2$ such that
\[
\frac{2}{\o^2}\left(-1+ \frac{M}{\a (R+\eps)^\a}\right) + o(\eps) \geq C>0 \qquad \forall \eps \in (0,\eps_3),
\]
and
\[
|t_4-t_3| \leq \frac{R}{2 C}=: \tau. \qedhere
\]
\end{proof}

\begin{rem}
From the proof of point ($iv$) it follows that $\eps_3$ does not depend on $p_1,p_2 \in \pa B_R(0)$ or on $l \in \mathbb{Z}_2^N$, cf. Remark \ref{l-eps}.
\end{rem}

\begin{lemma}\label{u C^1}
Let $u \in K_l$ be the minimizer found in Proposition \ref{metodo diretto}, let $(a,b)$ be a connected component of $T_{R/2}^+$. Then $u \in \mathcal{C}^1((a,b)$.
\end{lemma}
\begin{proof}
If $u \in K_l$ is a minimizer of $M$, we show that situation $c)$ of point ($iii$) of previous lemma cannot occur. Recall that $u$ is a minimizer also for $L$, which is the length with respect to the Jacobi's metric. If the situation $c)$ occurred, we could consider a totally normal neighbourhood $U$ of the point $u(t_1)$ such that
\[
\exists t_*,t_{**} \in (a,b): u(t_*), u(t_{**}) \in u((a,b)) \cap \pa U.
\]
If we connect $u(t_*)$ with $u(t_{**})$ with a minimizing arc for $L\left([t_*,t_{**}];\cdot\right)$, we get a uniquely determined geodesic segment $\gamma$ lying in $U$. In particular $\g$ is regular, so that cannot coincide with $u|_{(t_*,t_{**})}$. Then, the curve
\[
\wt{u}(t):= \begin{cases}
            u(t) & t \in [0,1] \setminus [t_*,t_{**}] \\
            \g(t)& t \in [t_*,t_{**}]
            \end{cases}
            \]
would be an element of $K_l$ with $L(\gamma)<L(u)$, a contradiction.
\end{proof}

\begin{prop}\label{|u|<R}
If $u \in K_l$ is the minimizer found in Proposition \ref{metodo diretto}, then
\[
|u(t)|<R \quad \forall t \in (0,1).
\]
\end{prop}
\begin{proof}
Let $[t_1, t_2]$ be a connected component of $T_R(u)$, let $(a,b)$ be a connected component of $T_{R/2}^+$ such that $[t_1,t_2] \subset (a,b)$. Let us consider $y(t):=u(\o t)$. Since $y \in \mathcal{C}^1\left((a/\o,b/\o)\right)$, it can lean against the circle $\left\lbrace y \in \R^2: |y|=R\right\rbrace$ with tangential velocity, and for every $\nu>0$ there exists $t_5>t_2$ (or $t_5<t_1$, and in this case the following inequality has to be changed in obvious way) such that
\[
\left|y\left(\frac{t_5}{\o}\right)-Re^{i \t(t_2/\o)}\right|< \nu \quad \text{and} \quad \left|\dot{y}\left(\frac{t_5}{\o}\right)- R \dot{\t}\left(\frac{t_2}{\o}\right)i e^{i \t(t_2/\o)}\right| < \nu.
\]
Thus
\begin{itemize}
\item $R$ is the radius of the circular solution of energy $-1$ for the Kepler's problem with homogeneity degree $-\a$,
\item outside $B_{R/2}(0)$, the $N$-centres problem can be seen as a small perturbation of the $\a$-Kepler's one: $V_\eps(y) = \frac{M}{\a |y|^\a}+ W_\eps(x)$.
\item $y$ is a solution of
\[
\begin{cases}
\ddot{y}(t)=\n V(y(t))  \\
y\left(\frac{t_5}{\o}\right) \simeq Re^{i \t(t_2/\o)}, & \dot{y}\left(\frac{t_5}{\o}\right) \simeq R \dot{\t}\left(\frac{t_2}{\o}\right)i e^{i \t(t_2/\o)}.
\end{cases}
\]
in an open neighbourhood of $t_5/\o$; these initial data are "more or less" the initial data of a circular solution.
\item the theorem of continuous dependence of the solutions with respect to the vector field and the initial data holds true for our problem outside $B_{R/2}(0)$.
\end{itemize}
Therefore $y$ cannot enter (or exit from) the ball $B_{R/2}(0)$ in a finite time, in contradiction with the choice of $l$ and point ($iv$) of Lemma \ref{prop del minimo}.
\end{proof}

\subsection{Classification of the minimizers}\label{assenza di collisioni}
So far, we obtained a set of extremals of the Maupertuis' functional $M$ at positive levels. In order to obtain classical solutions to \eqref{PI}, we need to show that these minimizers are collisions-free. In case $\a=1$ this fact isn't always true; however, we will be able to describe the behaviour of the collision-solutions, proving the classification in Theorem \ref{teo dinamica interna}.

The proof is by contradiction and requires several steps. In what follows we consider $l \in \Z_2^N$ satisfying \eqref{scelta di l 2} and fixed. We assume that the minimizer $u \in K_l$ has at least one collision; developing a blow-up analysis at the collision, we will reach a contradiction in the case $\a \in (1,2)$; the case $\a=1$ will be more difficult and will be treated separately by Levi-Civita regularization.

\paragraph{\textbf{Step 1)}} We prove that \emph{$u$ has no self-intersections at points different from the centres and that the set of collision times of $u$}
\[
T_c(u):=\left\lbrace t \in [0,1]:u(t)=c_j \text{ for some $j \in \left\lbrace 1,\ldots,N\right\rbrace $}\right\rbrace
\]
\emph{is finite}.

Since $M(u)<+\infty$, it follows immediately that $T_c(u)$ is a closed set of null measure. Hence $[0,1] \setminus T_c(u)$ is the union of a finite or countable number of open intervals. We recall that the energy of $u$ is constant and equal to $-1/\o^2$, see Lemma \ref{conservazione dell'energia}, where the value $\o$ has been already defined in Lemma \ref{prop del minimo}.

\begin{lemma}
If the interval $(a,b)$ is a connected component of $[0,1] \setminus T_c(u)$, then $u|_{(a,b)} \in \mathcal{C}^2((a,b),\R^2)$ and
\beq\label{equazione per u}
\o^2 \ddot{u}(t) = \n V_\eps(u(t)) \qquad \forall t \in (a,b).
\eeq
\end{lemma}
\begin{proof}
It is enough to repeat the proof of point ($i$) of Lemma \ref{prop del minimo}.
\end{proof}

\begin{prop}\label{no autointersezioni}
The minimizer $u$ parametrizes a path without self-intersections at points different from the centres $c_j$ ($j=1,\ldots,N$).
\end{prop}
\begin{proof}
Suppose by contradiction that $u$ has a self-intersection at a point $p \neq c_j$ for every $j$: $p=u(t_*)=u(t_{**})$, $t_*<t_{**}$. Let $(a,b)$ the connected component of $[0,1] \setminus T_c(u)$ containing $t_*$. We know that $u|_{(a,b)}$ is a classical solution of \eqref{equazione per u}, in particular it is of class $\mathcal{C}^2$.

\medskip
First we notice that, by the energy integral, $|\dot{u}(t)|>0$ for every $t$ such that $u(t)\in B_R(0)$, hence both $\dot{u}(t_*)$ and $\dot{u}(t_{**})$ are different from $0$. Let us define $v:[0,1] \to \R^2$ as follows
\[
v(t)=\begin{cases}
      u(t) & t \in [0,t_*] \cup (t_{**},1] ,\\
      u\left(\frac{t-t_*}{t_{**}-t_*} t_* +\left(1-\frac{t-t_*}{t_{**}-t_*}\right)t_{**}\right) & t \in (t_*,t_{**}].
      \end{cases}
\]
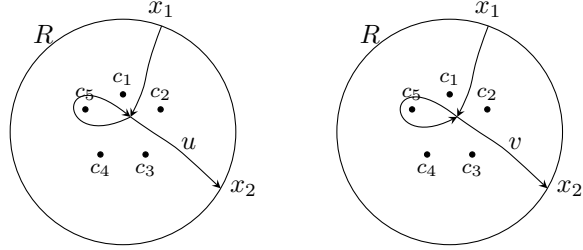
\begin{center}
\begin{tikzpicture}[>=stealth]
\draw (0,0) circle (1.5cm);
\draw (70:1.5cm) node[anchor=south]{$x_1$};
\draw (330:1.5cm) node[anchor=west]{$x_2$};
\draw[->] (70:1.5cm)..controls (0.2,0.6) and (0.4,0.8)..(0.1,0.2);
\draw[->] (0.1,0.2)..controls (-0.9,-0.3) and (-0.9,1)..(0.1,0.2);
\draw[->] (0.1,0.2)..controls (0.8,-0.3) and (0.5,0).. node[near end,above]{$u$} (330:1.5cm);
\filldraw[font=\footnotesize] (0.3,-0.3) circle (1pt) node[anchor=north]{$c_3$}
          (0.5,0.3) circle (1pt) node[anchor=south]{$c_2$}
          (0,0.5) circle (1pt) node[anchor=south]{$c_1$}
          (-0.3,-0.3) circle (1pt) node[anchor=north]{$c_4$}
          (-0.5,0.3) circle (1pt) node[anchor=south]{$c_5$};
\draw (135:1.5cm) node[anchor=south]{$R$};
\end{tikzpicture}
$\qquad$
\begin{tikzpicture}[>=stealth]
\draw (0,0) circle (1.5cm);
\draw (70:1.5cm) node[anchor=south]{$x_1$};
\draw (330:1.5cm) node[anchor=west]{$x_2$};
\draw[->] (70:1.5cm)..controls (0.2,0.6) and (0.4,0.8).. (0.1,0.2);
\draw[<-] (0.1,0.2)..controls (-0.9,-0.3) and (-0.9,1)..(0.1,0.2);
\draw[->] (0.1,0.2)..controls (0.8,-0.3) and (0.5,0).. node[near end,above]{$v$} (330:1.5cm);
\filldraw[font=\footnotesize] (0.3,-0.3) circle (1pt) node[anchor=north]{$c_3$}
          (0.5,0.3) circle (1pt) node[anchor=south]{$c_2$}
          (0,0.5) circle (1pt) node[anchor=south]{$c_1$}
          (-0.3,-0.3) circle (1pt) node[anchor=north]{$c_4$}
          (-0.5,0.3) circle (1pt) node[anchor=south]{$c_5$};
\draw (135:1.5cm) node[anchor=south]{$R$};
\end{tikzpicture}
\end{center}
The function $v$ parametrizes a path with $u([0,1])=v([0,1])$, but it goes along the loop connecting $u(t_*)$ and $u(t_{**})$ with the reversed orientation. The key observation is that this operation does not change the parity of the winding numbers with respect to the centres. Hence $v\in K_l$. Note that $v$ is also an extremal for $M$, since $M(u)=M(v)$.  On the other hand, it is trivially checked that, unless $\dot{u}(t_*)=\dot{u}(t_{**})=0$, $v$ isn't $\mathcal{C}^1$ at those instants. So we have a new minimizer of $M$ on $K_l$, which is collisions-free in an interval $(a,d)\ni t_*$, and hence here should be a classical solution of \eqref{equazione per u}; but this isn't possible since $v|_{(a,d)} \notin \mathcal{C}^1((a,d),\R^2)$.
\end{proof}

Coming back to the properties of $T_c(u)$, we state the following known result (see e.g. \cite{BaFeTe}).
\begin{lemma}\label{collisioni isolate}
If $u$ has a collision at an instant $t_0 \in [0,1]$, then $t_0$ is isolated in $T_c(u)$. In particular, the cardinality of $T_c(u)$ is finite.
\end{lemma}
\begin{proof}
Assume by contradiction that $t_0$ is an accumulation point in the set $T_c(u)$, with $u(t_0)=c_j$. By continuity, only collisions in $c_j$ can accumulate in $t_0$. In this case there exists a sequence of intervals $((a_n,b_n))$ with $(a_n,b_n) \subset [0,1]$, $a_n \to t_0$ and $b_n \to t_0$ as $n \to \infty$, $u(a_n)=c_j=u(b_n)$ for every $n$, and
\[
|u(t)-c_j|>0 \quad \forall t \in (a_n,b_n).
\]
On each of these intervals, since $u$ is close to $c_j$ (at least for $n$ sufficiently large),
\[
|u(t)-c_k| \geq C >0 \qquad \text{for every $k \in \left\lbrace 1,\ldots,N\right\rbrace $, $k \neq j$} .
\]
Let us set $I(t):= |u(t)-c_j|^2$. Since $t \mapsto u(t)$ is a classical solution of \eqref{equazione per u} for $t \in (a_n,b_n)$, by differentiating  twice $I(t)$ we obtain a modified Lagrange-Jacobi identity:
\begin{equation*}
\ddot{I}(t)=-\frac{4}{\o^2}+\frac{2}{\o^2}(2-\a)\frac{m_j}{\a |u(t)-c_j|^\a}
+\frac{2}{\o^2}\sum_{\substack{k=1 \\k \neq j}}^N \frac{m_k}{|u(t)-c_k|^\a}\left(\frac{2}{\a}-\frac{\left\langle u(t)-c_k,u(t)-c_j\right\rangle }{|u(t)-c_k|^2}\right).
\end{equation*}
Let $\xi_n \in (a_n,b_n)$ the maximizer of $I$ in $(a_n,b_n)$. It results $\ddot{I}(\xi_n) \leq 0$ for every $n$. Since in a neighbourhood of $t_0$ the second term in the expression of $\ddot{I}$ becomes arbitrarily large, while the other terms are bounded, we also get
\[
\lim_{n \to \infty} \ddot{I}(\xi_n)=+\infty,
\]
a contradiction. The collisions are isolated and, by compactness, the interval $[0,1]$ contains only a finite number of them.
\end{proof}

\begin{rem}\label{convessita'}
The previous proof shows that, \emph{if $u$ collides in $c_j$, in a sufficiently small neighbourhood of $c_j$ the function $I(t)=|u(t)-c_j|^2$ is strictly convex}.
\end{rem}

\paragraph{\textbf{Step 2})} We would pass from a global analysis of the minimizer $u$ to a local study in a neighbourhood of a collision. This is possible thanks to step 1: \emph{$u$ has an isolated collision at $t_0$ in a centre $c_j$}, $j \in \left\lbrace 1,\ldots,N\right\rbrace $. In particular there exist $c,d \in [0,1]$ such that
\begin{itemize}
\item $c<t_0<d$ and $t_0$ is the unique collision time in $[c,d]$,
\item the function $I$ is strictly convex in $[c,d]$.
\end{itemize}
Let us set $\bar{p}_1:=u(c)$, $\bar{p}_2=u(d)$. Since $u \in \mathcal{C}([c,d],\R^2)$ there exists $\mu > 0$ such that
\[
|u(t)-c_k| \geq 2\mu> 0 \quad \text{for every $t \in [c,d]$ and for every $k \in \left\lbrace 1,\ldots, N \right\rbrace \setminus \{ j\}$}.
\]
This motivates us to write
\beq\label{def di V_0}
V_\eps(y)=\frac{m_j}{\a|y-c_j|^\a}+ V_\eps^j(y), \quad \text{where} \quad V_\eps^j(y):= \sum_{\substack{k=1 \\ k \neq j}}^N\frac{m_k}{\a|y-c_k|^\a}.
\eeq
Indeed, in a neighbourhood $U_j$ of $c_j$ such that $\text{dist}(U_j,c_k) \geq  \mu$ for every $k$, the potential $V_\eps$ splits in a principal component due to the attraction of $c_j$, and a perturbation term $V_\eps^j$ due to the attraction of the other centres. Of course, for $x \in U_j$, $V_\eps^j$ is smooth and bounded.

We define
\begin{multline*}
\wh{\mathcal{K}}_l^{\bar p_1 \bar p_2}:= \left\lbrace v \in H^1\left([c,d],\R^2 \setminus \left\lbrace c_1, \ldots, c_N\right\rbrace \right): v(c)=\bar p_1, v(d)= \bar p_2, \right.\\
\text{the function } \Gamma_v(t):=\begin{cases} u(t) & t \in [0,c) \cup (d,1] \\ v(t) & t \in [c,d] \end{cases} \text{belongs to $K_l$} \big\},
\end{multline*}
and
\[
\mathcal{K}_l^{\bar p_1 \bar p_2}:= \wh{\mathcal{K}}_l^{\bar p_1 \bar p_2} \cup \left\lbrace v \in H^1([c,d],\R^2): v(c)=\bar p_1, v(d)=\bar p_2, \Gamma_v \in \mathfrak{Coll}_l \right\rbrace .
\]
The set $\mathcal{K}_l^{\bar p_1 \bar p_2}$ is weakly closed. We define the restriction of the Maupertuis' functional to $\mathcal{K}_l^{\bar p_1 \bar p_2}$ as
\[
M_l^{\bar p_1 \bar p_2}\;\colon\; \mathcal{K}_l^{\bar p_1 \bar p_2} \to \R \cup \left\lbrace +\infty\right\rbrace  \qquad M_l^{\bar p_1 \bar p_2}(u)=\frac{1}{2} \int_c^d |\dot{v}(t)|^2\,dt \int_c^d \left( V_\eps(v(t))-1\right)\,dt.
\]
It inherits the properties of weak lower semi-continuity and coercivity from $M$, then has a minimum on $\mathcal{K}_l^{\bar p_1 \bar p_2}$ at a positive level. Since $u$ is a minimizer of $M$ on $K_l$, then $u|_{[c,d]}$ is a minimizer of $M_l^{\bar p_1 \bar p_2}$ on $\mathcal{K}_l^{p_1 p_2}$ (see Proposition \ref{localizzazione dei minimi_M}).

\paragraph{\textbf{Step 3})} We introduce some more notation. For $\rho \geq 0$, we define
\[
d(\rho):= \min \left\lbrace M_l^{\bar p_1 \bar p_2}(v): v \in \mathcal{K}_l^{\bar p_1 \bar p_2}, \min_{t \in [c,d]} |v(t)-c_j|=\rho \right\rbrace .
\]
The value $d(0)$ is the minimum of $M_l^{\bar p_1 \bar p_2}$ on the elements of $\mathcal{K}_l^{\bar p_1 \bar p_2}$ which collide in $c_j$; hence $d(0)$ is achieved by $u|_{[c,d]}$.

\begin{lemma}\label{lemma 2.6}
The function $\rho \mapsto d(\rho)$ is continuous in $\rho=0$.
\end{lemma}

\begin{proof}
The proof is exactly the same of that of Lemma 17 in \cite{TeVe}. We have to take into account that in our case collisions occur in $c_j$ and not in $0$, and that we are dealing with the Maupertuis' functional and not with the action functional; nevertheless the same argument works.
\end{proof}

Now, given $0<\rho_1<\rho_2$, we set
\[
\mathcal{K}_l^{\bar p_1 \bar p_2}(\rho_1,\rho_2):=\left\lbrace v \in \mathcal{K}_l^{\bar p_1 \bar p_2}: \min_{t \in [c,d]} |v(t)-c_j| \in [\r_1,\r_2]\right\rbrace .
\]
It is a weakly closed subset of $\mathcal{K}_l^{\bar p_1 \bar p_2}$, so the restriction of $M_l^{\bar p_1 \bar p_2}$ to $\mathcal{K}_l^{p_1 p_2}(\r_1,\r_2)$ has a minimum that we denote as
\[
m(\r_1,\r_2):= \min_{v \in \mathcal{K}_l^{\bar p_1 \bar p_2}(\r_1,\r_2)} M_l^{\bar p_1 \bar p_2}(v).
\]
We also set
\[
\mathcal{M}_{\r_1 \r_2} :=\left\lbrace v \in \mathcal{K}_l^{\bar p_1 \bar p_2}(\r_1,\r_2): M_l^{\bar p_1 \bar p_2}(v)=m(\r_1,\r_2) \text{ and}
 \min_{t \in [c,d]} |v(t)-c_j| <\r_2\right\rbrace .
\]

In this step we aim at proving the following result.

\begin{prop}\label{teorema 2.7}
There exists $\bar{\r}>0$ such that for $\r_1$, $\r_2: 0<\r_1<\r_2 \leq \bar{\r}$ implies $\mathcal{M}_{\r_1 \r_2} = \emptyset$.
\end{prop}

\begin{rem}
The proposition states that if we force the functions to go very close to $c_j$, i.e.
\[
\min_{t \in [c,d]} |v(t)-c_j|< \bar{\r},
\]
then the minima $m(\r_1, \r_2)$ are achieved by elements of $\mathcal{K}_l^{\bar p_1 \bar p_2}(\r_1,\r_2)$ which stay as far as possible from $c_j$.
\end{rem}

Assume by contradiction that the statement is not true. Then there existed two sequences $(\r_n)$, $(\bar{\r}_n)$ such that
\begin{gather}
0<\r_n<\bar{\r}_n \quad \forall n, \qquad \r_n \to 0, \bar{\r}_n \to 0, \quad \text{for $n \to \infty$}, \notag\\
\forall n  \ \exists u_n \in \mathcal{K}_l^{\bar p_1 \bar p_2}: \min_{t \in [c,d]} |u_n(t)-c_j|=\rho_n, \label{successione assurda} \\
 M_l^{\bar p_1 \bar p_2}(u_n)=m(\r_n,\bar{\r}_n)=d(\rho_n). \notag
\end{gather}
We can assume also that for every $n \in \mathbb{N}$
\[
\max \left\{ \inf_{y \in \pa B_{\r_n}(c_j)} |\bar p_1-y|,  \inf_{y \in \pa B_{\r_n}(c_j)} |\bar p_2-y|\right\} > 0.
\]
Thanks to Lemma \ref{lemma 2.6}, $M_l^{\bar p_1 \bar p_2}(u_n) \to d(0)$ for $n \to \infty$, namely $(u_n)$ is a minimizing sequence in $\mathcal{K}_l^{\bar p_1 \bar p_2}$ (we are assuming that the minimum of $M_l^{\bar p_1 \bar p_2}$ is achieved over collisions). Since $M_l^{\bar p_1 \bar p_2}$ is coercive, $(u_n)$ is bounded and up to subsequence is weakly convergent to a function $\wt{u} \in \mathcal{K}_l^{\bar p_1 \bar p_2}$, which is a minimizer of $M_l^{\bar p_1 \bar p_2}$ (possibly different from $u|_{[c,d]}$) due to the weakly lower semi-continuity of $M_l^{\bar p_1 \bar p_2}$. We point out that $\wt{u}$ has to collide in $c_j$ and could collide in centres different from $c_j$ as well.\\
By Lemma \ref{conservazione dell'energia}, the energy of $\wt{u}$ is constant and equal to $-1/\wt{\o}^2$, where
\[
\wt{\o}^2 := \frac{\int_d^c V_\eps(\wt{u})-1}{\frac{1}{2}\int_c^d |\dot{\wt{u}}|^2}.
\]
Now, the same discussion of step 1 shows that the set $T_c(\wt{u})$ of collision times of $\wt{u}$ contains a finite number of elements, and we can assume that
\begin{itemize}
\item there exists a unique collision time $t_0$ in $[c,d]$ such that $\wt{u}(t_0)=c_j$,
\item there exists $\mu >0$ such that $|\wt{u}(t)-c_k| \geq 2\mu > 0$ for every $t \in [c,d]$, for every $k \neq j$.
\item the function $|\wt{u}(t)-c_j|^2$ is strictly convex in $[c,d]$.
\end{itemize}
Otherwise we can replace $[c,d]$ with a smaller interval.

\vspace{1 em}

The paths $u_n$ enjoy some common properties. Firstly, since the weak convergence in $H^1$ implies the uniform one, there exists $n_0 \in \mathbb{N}$ such that
\beq\label{uniforme limitatezza di V_0}
n \geq n_0 \Rightarrow |u_n(t)-c_k| \geq \mu \quad \forall t \in [c,d], \ \forall k \neq j.
\eeq
We rename as $(u_n)$ the sequence obtained by dropping the first $(n_0 -1)$-terms.
Let us set
\[
T_{\r_n} (u_n)=\left\{ t \in [c,d]: |u_n(t)-c_j|=\r_n\right\}.
\]
We also introduce the polar coordinates and the (absolute value of the) angular momentum of $u_n$ with respect to the centre $c_j$:
\begin{gather*}
u_n(t)= c_j+ w_n(t) e^{i \phi_n(t)}, \\
\mathfrak{C}_n^j(t):= \left| \left(u_n(t)-c_j\right) \land \dot{u}_n(t)\right|.
\end{gather*}
Here $w_n:[c,d] \to \R^+$ and $\phi_n(t):[c,d] \to \R$.

\begin{lemma}\label{proprietà di u_n}
For every $n \in \mathbb{N}$ the path $u_n$ has the following properties:
\begin{itemize}
\item[($i$)] If $(c',d')$ is a connected component of $[c,d] \setminus T_{\r_n}(u_n)$, then $u_n|_{(c',d')}$ is $\mathcal{C}^2$ and solves
\beq\label{omega_n}
\o_n^2 \ddot{u}_n(t)=\n V_\eps(u_n(t)) \quad \text{where} \quad \o_n^2:= \frac{\int_c^d \left(V_\eps(u_n)-1\right)}{\frac{1}{2} \int_c^d |\dot{u}_n|^2}.
\eeq
\item[($ii$)] For every $n \in \mathbb{N}$, there exist $t_n^- \leq t_n^+$ such that:
\begin{align*}
&|u_n(t)-c_j| > \rho_n \qquad t \in [c,t_n^-) \cup (t_n^+,d]\\
&|u_n(t)-c_j|=\r_n \qquad t \in [t_n^-,t_n^+],
\end{align*}
namely $T_{\r_n}(u_n)=[t_n^-,t_n^+]$.
\item[($iii$)] The sequence $(\o_n^2)$ is bounded above and uniformly bounded below by a strictly positive constant. Hence there exist a subsequence of $(u_n)$ (still denoted $(u_n)$) and $\Omega>0$ such that
\[
 \lim_{n \to \infty} \o_n=\Omega.
\]
\item[($iv$)] The energy of the function $u_n$ is constant in $[c,d]$:
\[
\frac{1}{2}|\dot{u}_n(t)|^2-\frac{V_\eps(u_n(t))}{\o_n^2}=-\frac{1}{\o_n^2} \qquad \forall t \in [c,d].
\]
Moreover, the sequence $(-1/\o_n^2)$ is bounded in $\R$.
\item[($v$)] The function $\phi_n|_{(t_n^-,t_n^+)}$ is $\mathcal{C}^2$, strictly monotone and is a solution of
\beq\label{equazione per phi_n}
\ddot{\phi}_n(t) = \frac{1}{\r_n \o_n^2} \left\langle \n V_\eps\left(c_j+\r_n e^{i \phi_n(t)}\right), i e^{i \phi_n(t)}\right\rangle .
\eeq
\item[($vi$)] A minimizer of $M_l^{\bar p_1 \bar p_2}$ in $\mathcal{K}_l^{\bar p_1 \bar p_2}$ is of class $\mathcal{C}^1$ in $[c,d]$. In particular, this holds true for $u_n$, for every $n$.
\end{itemize}
\end{lemma}

\begin{proof}
The proof of ($i$) and ($v$) is the same of the points ($i$) and ($ii$) of Lemma \ref{prop del minimo}, respectively.\\
($ii$) On every interval $(c',d') \subset (c,d) \setminus T_{\r_n}(u_n)$ $u_n$ solves the \eqref{omega_n}; hence the uniform convergence of $(u_n)$ to $\wt{u}$ and the computation of
\[
\frac{d^2}{dt^2} |u_n(t)-c_j|^2
\]
(see the proof of Lemma \ref{collisioni isolate}) imply that the function $|u_n(t)-c_j|^2$ is strictly convex over such an interval. Therefore, if there exist $t_1<t_2$ such that $|u_n(t_1)-c_j| = |u_n(t_2)-c_j|=\r_n$ then $|u_n(t)-c_j|=\r_n$ for every $t \in (t_1,t_2)$.\\
($iii$) We have
\beq\label{20}
\o_n^2=\frac{M_l^{\bar p_1 \bar p_2}(u_n)}{\frac{1}{4} \left(\int_c^d |\dot{u}_n|^2\right)^2}=\frac{d(\r_n)}{\frac{1}{4}\|\dot{u}_n\|_{L^2([c,d])}^4}.
\eeq
We know that
\beq\label{20*}
0<d(0)<d(\r_n) \quad \text{and} \quad d(\r_n) \to d(0) \Rightarrow \exists C_1,C_2 > 0: \ C_1 \leq d(\r_n) \leq C_2 \ \forall n.
\eeq
As to the denominator of \eqref{20}, observe that, for every, $n$ the path $u_n$ covers at least a fixed distance; therefore, like in the proof of Lemma \ref{lemma 1.4},  there exists  $C_3>0$ such that
\beq\label{20**}
\|\dot{u}_n\|_{L^2([c,d])} \geq C_3 \qquad \forall n
\eeq
Moreover, being $(u_n)$ a minimizing sequence of a coercive functional, $(u_n)$ is bounded in the $H^1$-norm and a fortiori there exists $C_4 >0$ such that
\beq\label{20***}
\|\dot{u}_n\|_{L^2([c,d])}\leq C_4 \qquad \forall n.
\eeq
Altogether, \eqref{20}, \eqref{20*}, \eqref{20**} and \eqref{20***} imply the assertion.\\
($iv$) The energy is constant, as proved in Lemma \ref{conservazione dell'energia}. The boundedness of $(-1/\o_n^2)$ is a trivial consequence of point ($iii$).\\
($vi$) We take advantage again of the Proposition 3.16 in \cite{BaTeVe}: since $u_n$ is a minimizer of $M_l^{\bar p_1 \bar p_2}$ on $K_l^{\bar p_1 \bar p_2}(\r_n, \bar{\r}_n)$, then one of the following situations occurs:
\begin{itemize}
\item[a)] $t_n^-<t_n^+$ and $u \in \mathcal{C}^1\left([c,d]\right)$,
\item[b)] $t_n^-=t_n^+$ and $u \in \mathcal{C}^1\left([c,d]\right)$,
\item[c)] $t_n^-=t_n^+$ and $u \notin \mathcal{C}^1\left([c,d]\right)$.
\end{itemize}
But applying the line of reasoning already used in the proof of Theorem \ref{|u|<R}, it is easy to check that the situation $c)$ cannot occur.
\end{proof}

We are now in a position to prove the following result.

\begin{prop}\label{variazione dell'angolo limitata}
The minimizer $u_n$ is free of self-intersections in $[c,d]$. In particular, the total variation of the angle $\phi_n$ is smaller then $2\pi$.
\end{prop}

\begin{proof}
The function $u_n$ has no self-intersections for $t \in [c,t_n^-) \cup (t_n^+,d]$. The prove is the same of that of Proposition \ref{no autointersezioni}. If $u_n$ has a self-intersection on the obstacle $\left\lbrace |y-c_j|=\r_n \right\rbrace $, the monotonicity of $\phi_n$ implies that $u_n$ makes a complete wind around it. But then we can consider the function $v$ which parametrizes the same path of $u$, but reverses the orientation on the obstacle. One has $M_l^{\bar p_1 \bar p_2}(u_n)=M_l^{\bar p_1 \bar p_2}(v)$, so that $v$ is a local minimizer of $M_l^{\bar p_1 \bar p_2}$ with $\min_{t \in [c,d]} |v(t)-c_j| = \r_n$. Hence $v$ satisfies the energy integral, so that it cannot approach to the obstacle with velocity $0$. Therefore it should be a minimizer which is not $\mathcal{C}^1$, against point ($vi$) of the previous lemma.
\end{proof}

\begin{prop}\label{prop 2.12}
The estimates
\[
\mathfrak{C}_n^j(t)= \r_n^{\frac{2-\a}{2}} \sqrt{\frac{2 m_j}{\o_n^2 \a}}\left(1+ O(\r_n^\a)\right), \qquad t_n^+ - t_n^-=O(\r_n^{\frac{\a+2}{2}})
\]
hold for $n \to \infty$.
\end{prop}
\begin{proof}
Since $u_n \in \mathcal{C}^1\left([c,d]\right)$, it can lean against the obstacle $\left\lbrace |y-c_j| =\r_n\right\rbrace$ with velocity $\dot{u}_n(t)$ orthogonal to the radial segment joining $c_j$ and $u_n(t)$. Therefore for every $t \in [t_n^-,t_n^+]$ there holds $\mathfrak{C}_n^j(t)=\r_n |\dot{u}_n(t)|= \r_n^2\dot{\phi}_n(t)$. From the expression of the energy and the uniform boundedness of $(V_\eps^j(u_n))$ (see equation \eqref{uniforme limitatezza di V_0}) we get
\beq \label{*25}
\begin{split}
\mathfrak{C}_n^j(t)&= \r_n \sqrt{ \frac{2}{\o_n^2} \left(\frac{m_j}{\a\r_n^\a}+ V_\eps^j\left(u_n(t)\right)-1\right) } \\
&= \r_n^{\frac{2-\a}{2}} \sqrt{\frac{2m_j}{\o_n^2 \a}+ \frac{2\r_n^\a}{\o_n^2} \left(V_\eps^j(u_n(t))-1\right)} = \r_n^{\frac{2-\a}{2}} \sqrt{\frac{2 m_j}{\o_n^2 \a}}\left(1+O(\r_n^\a)\right).
\end{split}
\eeq
Therefore
\[
\dot{\phi}_n(t) = \r_n^{\frac{-2-\a}{2}} \sqrt{\frac{2 m_j}{\o_n^2 \a}}\left(1+O(\r_n^\a)\right),
\]
and the total variation of $\phi_n$ on the obstacle is
\[
\phi_n(t_n^+)-\phi_n(t_n^-) = \r_n^{\frac{-2-\a}{2}} \sqrt{\frac{2 m_j}{\o_n^2 \a}}\left(1+O(\r_n^\a)\right)(t_n^+-t_n^-).
\]
This variation is bounded by $2\pi$, so that $t_n^+ - t_n^-=O(\r_n^{\frac{\a+2}{2}})$.
\end{proof}

In order to obtain a contradiction, we consider a blow-up of our sequence.\\
For every $n \in \mathbb{N}$, let us fix $t_n \in [t_n^-,t_n^+]$. By the previous Proposition the sequence $(t_n)$ tends to the limit $t_0$ which is the unique collision time of $\wt{u}$ in $(c,d)$. Let us set
\[
c_n:=\r_n^{-\frac{\a+2}{2} }(c-t_n), \qquad d_n:=\r_n^{-\frac{\a+2}{2}} (d-t_n).
\]
We also define
\[
s_n^-:= \r_n^{-\frac{\a+2}{2} }(t_n^- - t_n), \qquad s_n^+:= \r_n^{-\frac{\a+2}{2} }(t_n^+ - t_n)
\]
We note that $c_n \to -\infty$, $d_n \to +\infty$ as $n \to \infty$. As far as $(s_n^-)$ and $(s_n^+)$ are concerned, they are two bounded sequences thanks to proposition \ref{prop 2.12}, so that there exists a subsequence of $(\r_n)$ (which we still denote $(\r_n)$) such that they converge to limits $s^-$ and $s^+$ respectively.

\begin{rem}
Consider the change of variable
\[
s(t,n)= \r_n^{-\frac{\a+2}{2} }(t-t_n) \Leftrightarrow t(s,n)= t_n+ \r_n^{\frac{\a+2}{2} }s.
\]
One has
\[
s(t,n) \in [c_n,d_n] \Leftrightarrow t(s,n) \in [c,d], \qquad s(t,n) \in [s_n^-,s_n^+] \Leftrightarrow t(s,n) \in [t_n^-,t_n^+].
\]
\end{rem}

\noindent We introduce the sequence of paths $v_n:[c_n,d_n] \to \R^2$,
\[
v_n(s):=c_j+\frac{1}{\r_n}\left(u_n\left(t_n+\r_n^{ \frac{\a+2}{2}  }s\right)- c_j\right).
\]
In polar coordinates with respect to the centre $c_j$ we write
\[
v_n(s)=c_j+\bar{w}_n(s) e^{i \bar{\phi}_n(s)},
\]
where
\[
\bar{w}_n(s)=\frac{1}{\r_n}w_n\left(t_n+\r_n^{ \frac{\a+2}{2}  }s\right), \qquad \bar{\phi}_n(s)=\phi_n\left(t_n+\r_n^{ \frac{\a+2}{2}  }s\right).
\]
Each $v_n$ is of class $\mathcal{C}^1$ and
\begin{align*}
& |v_n(s)-c_j| = 1 \qquad \text{for $s \in [s_n^-,s_n^+]$},\\
& |v_n(s)-c_j| > 1 \qquad \text{for $s \in [c_n,s_n^-) \cup (s_n^+,d_n]$}.
\end{align*}
The restriction $v_n|_{[c_n,s_n^-) \cup (s_n^+,d_n]}$ is of class $\mathcal{C}^2$ and satisfies the equation
\begin{align*}
\ddot{v}_n(s) & = -\frac{\r_n^{2+\a}}{\o_n^2\r_n} \sum_{k=1}^N \frac{m_k}{\left|u_n\left(t(s,n)\right)-c_k\right|^{\a+2}}\left(u_n\left(t(s,n)\right)-c_k\right) \\
& = - \frac{m_j \left[ \frac{1}{\r_n} \left( u_n\left(t(s,n)\right)-c_j\right) \pm c_j\right]}{\o_n^2 \left|\frac{1}{\r_n} \left( u_n\left(t(s,n)\right)-c_j\right)\pm c_j \right|^{\a+2}} + \frac{\r_n^{\a+1}}{\o_n^2} \n V_\eps^j (u_n\left(t(s,n)\right) \\
& = -\frac{m_j}{\o_n^2 \left|v_n(s)-c_j\right|^{\a+2}} \left(v_n(s)-c_j\right) + O(\r_n^{\a+1}).
\end{align*}
This suggests to consider the quantity
\[
\bar{h}_n(s):= \frac{1}{2}|\dot{v}_n(s)|^2- \frac{m_j}{\o_n^2 \a \left|v_n(s)-c_j\right|^\a},
\]
the energy of the function $v_n$ for the potential of the $\a$-Kepler's problem with centre in $c_j$.  This is not a constant function in $[c_n,d_n]$, however it can be easily controlled.
\begin{align*}
\bar{h}_n(s) & = \r_n^\a \left[ \frac{1}{2}\left| \dot{u}_n\left(t(s,n)\right)\right|^2- \frac{m_j}{\o_n^2 \a \left|u_n\left(t(s,n)\right)-c_j\right|^\a} \right]\\
& = \r_n^\a \left[ -\frac{1}{\o_n^2} + \frac{1}{\o_n^2} V_\eps^j\left(u_n\left(t(s,n)\right)\right)\right].
\end{align*}
Therefore, form the point ($iv$) of Lemma \ref{proprietà di u_n} we deduce
\[
\lim_{n \to \infty} \bar{h}_n(s)=0 \quad \text{for every $s \in [c_n,d_n]$}.
\]
The uniform boundedness of $\left(V_\eps^j(u_n)\right)$  makes the convergence uniform on every closed interval $[a,b] \subset \R$.\\
Let us also define the (absolute value of the) angular momentum of $v_n$ with respect to the centre $c_j$:
\[
\bar{\mathfrak{C}}_n^j(s):= \left| \left(v_n(s)-c_j\right) \land \dot{v}_n(s)\right|.
\]
If $s \in [s_n^-,s_n^+]$, using Proposition \ref{prop 2.12} we obtain
\[
\bar{\mathfrak{C}}_n^j(s)= \r_n^{\frac{\a+2}{2}} \dot{\phi}_n\left(t(s,n)\right) = \r_n^{\frac{\a-2}{2}} \mathfrak{C}_n^j\left(t(s,n)\right) = \sqrt{\frac{2 m_j}{\o_n^2\a}} \left(1+O(\r_n^\a)\right).
\]
Hence
\beq\label{convergenza di c_n}
\lim_{n \to \infty} \bar{\mathfrak{C}}_n^j(s) = \sqrt{\frac{2 m_j}{\Omega^2 \a}}, \quad \text{for every $s \in [s_-,s_+]$},
\eeq
with uniform convergence in $[s^-,s^+]$. For the reader's convenience, we recall that $\Omega=\lim_n \o_n$. The previous computation implies that the sequence $(\bar{\mathfrak{C}}_n^j|_{[s^-,s^+]})$ is uniformly bounded in a neighbourhood of $[s^-,s^+]$.

Recalling the point ($v$) of Lemma \ref{proprietà di u_n}, we obtain an equation for $\bar{\phi}_n$ when $s \in (s_n^-,s_n^+)$:
\begin{align*}
\ddot{\bar{\phi}}_n(s) &= \frac{\r_n^{\a+1}}{\o_n^2}\left\langle \n V_\eps\left(c_j+\r_n e^{i \bar{\phi}_n(s)}\right), i e^{i \bar{\phi}_n(s)}\right\rangle \\
&= -\frac{1}{\o_n^2} \left\langle m_j   e^{i \bar{\phi}_n(s)}, i  e^{i \bar{\phi}_n(s)}\right\rangle  + \frac{\r_n^{\a+1}}{\o_n^2}\left\langle \n V_\eps^j\left(c_j+ \r_n  e^{i \bar{\phi}_n(s)}\right),i  e^{i \bar{\phi}_n(s)}\right\rangle  \\
&= 0+ O(\r_n^{\a+1}).
\end{align*}
Hence the restriction $v_n|_{(s_n^-,s_n^+)}$ is of class $\mathcal{C}^2$ and satisfies
\begin{multline*}
\ddot{v}_n(s)  = \ddot{\bar{\phi}}_n(s) i e^{i\bar{\phi}_n(s)}- \left(\dot{\bar{\phi}}_n(s)\right)^2 e^{i\bar{\phi}_n(s)}  = \ddot{\bar{\phi}}_n(s) i \left(v_n(s)-c_i\right)- \left(\bar{\mathfrak{C}}_n^j(s)\right)^2\left(v_n(s)-c_i\right)=\\
= - \left(\bar{\mathfrak{C}}_n^j(s)\right)^2\left(v_n(s)-c_i\right) + i \left(v_n(s)-c_i\right) O(\r_n^{\a+1}). \qedhere
\end{multline*}

Summing up
\beq\label{equazione per v_n}
\ddot{v}_n = \begin{cases}
                -\frac{m_j \left(v_n-c_i\right)}{\o_n^2 \left|v_n-c_i\right|^{\a+2}}  + O(\r_n^{\a+1})  &  \text{in $[c_n,s_n^-) \cup (s_n^+,d_n]$} \\
                \\
                 - \left(\bar{\mathfrak{C}}_n^j\right)^2\left(v_n-c_i\right) + i \left(v_n-c_i\right) O(\r_n^{\a+1}) & \text{in $(s_n^-,s_n^+)$}.
                 \end{cases}
\eeq
This shows that $v_n$ is not necessarily of class $\mathcal{C}^2$ in $s_n^-$ and $s_n^+$; anyway there exist the right and left limits at these points.

\begin{prop}\label{convergenza per v_n}
Let $[a,b] \subset \R$, $a \leq 0\leq b$. There exists a subsequence of $(v_n)$ which converges in the $\mathcal{C}^1$ topology on $[a,b]$.
\end{prop}

\begin{proof}
There is uniform convergence to $0$  of the energies $\bar{h}_n$ over $[a,b]$; thus the restrictions $(\bar{h}_n|_{[a,b]})$ define a bounded sequence in the uniform topology. Since for every $n$
\[
\inf_{s \in [a,b]} |v_n(s)-c_j|= |v_n(0)-c_j|=1,
\]
for every $s \in [a,b]$
\[
|\dot{v}_n(s)|^2 = 2 \bar{h}_n(s)+\frac{2m_j}{\o_n^2 \a \left|v_n(s)-c_j\right|^\a} \leq 2 \| \bar{h}_n|_{[a,b]}\|_\infty + 2 \frac{m_j}{\o_n^2 \a} .
\]
Therefore
\[
\| \dot{v}_n|_{[a,b]} \|_\infty \leq \sqrt{2} \sup_n \left( \| \bar{h}_n|_{[a,b]}\|_\infty + \frac{m_j}{\o_n^2 \a}\right)^{\frac{1}{2}}<+\infty,
\]
i.e. $\left(\dot{v}_n|_{[a,b]}\right)$ is uniformly bounded. Now,
\begin{enumerate}
\item $(v_n|_{[a,b]})$ is equicontinuous: for every $s_1,s_2 \in [a,b]$, for every $n \in \mathbb{N}$
\[
|v_n(s_1)-v_n(s_2)| \leq \| \dot{v}_n|_{[a,b]} \|_\infty |s_1-s_2| \leq C |s_1-s_2|.
\]
\item $\left(v_n|_{[a,b]}\right)$ is uniformly bounded: for every $s \in [a,b]$, for every $n \in \mathbb{N}$:
\[
|v_n(s)| \leq |v_n(0)|+C|s| \leq \eps +1+C \max \{|a|,|b|\}.
\]
\end{enumerate}
Hence we can apply the Ascoli-Arzel\`a theorem, to obtain a uniformly converging subsequence (still denoted by $(v_n)$).
From equation \eqref{equazione per v_n} we see also that $(\ddot{v}_n|_{[a,b]})$ is uniformly bounded. Indeed
\begin{align*}
&|\ddot{v}_n(s)| \leq \frac{m_j}{\o_n^2}+O(\r_n^{\a+1}) \leq C<+\infty  \qquad \text{for every $s \in [c_n,s_n^-) \cup (s_n^+,d_n]$},\\
&|\ddot{v}_n(s)| \leq \left(\bar{\mathfrak{C}}_n^j(s)\right)^2 + O(\r_n^{\a+1}) \leq C<+\infty \qquad \text{for every $s \in (s_n^-,s_n^+)$} \\
& \max \left\lbrace \lim_{s \to \left(s_n^\pm\right)^\pm}|\ddot{v}(s)|\right\rbrace  = C<+\infty ,
\end{align*}
(recall \eqref{convergenza di c_n} for the second bound) and immediately $\sup_n \|\ddot{v}_n|_{[a,b]} \|_\infty < +\infty$. Moreover
\[
\lim_{n \to \infty} \frac{1}{2}|\dot{v}_n(0)|^2= \lim_{n \to \infty} \bar{h}_n(0)+\frac{m_j}{\o_n^2 \a} = \frac{m_j}{\Omega^2 \a}.
\]
In particular, $\left(\dot{v}_n(0)\right)$ is bounded, too.
Now it is sufficient to repeat the previous argument and use the Ascoli-Arzel\`a theorem for $(\dot{v}_n)$.
\end{proof}

Applying the Proposition on each interval $[-k,k]$ we obtain a subsequence of $(v_n)$ (still denoted by $(v_n)$) which converges in the $\mathcal{C}^1$ topology on every closed interval of $\R$. We call $v: \R \to \R^2$ its limit. By \eqref{equazione per v_n} the sequence $(\ddot{v}_n)$ uniformly converges on every compact subset of $\R \setminus \{s^-,s^+\}$, so $v \in \mathcal{C}^2\left(\R \setminus \{s^-,s^+\}\right)$ and
\begin{itemize}
\item $v$ is a classical solution of the $\a$-Kepler's problem
\[
\ddot{v}(s) = -\frac{m_j}{ \Omega^2 |v(s)-c_j|^{\a+2}}\left(v(s)-c_j\right) \quad \text{for $s \in (-\infty,s^-) \cup (s^+,+\infty)$}.
\]
\item $v$ has constant energy equal to $0$ (even in $[s^-,s^+]$),
\item $v$ has constant angular momentum with respect to $c_j$, equal to $\bar{\mathfrak{C}}^j=\sqrt{\frac{2m_j}{\Omega ^2 \a}}$ (even in $[s^-,s^+]$),
\item $|v(s)-c_j|=1$ for $s \in [s^-,s^+]$,
\item $|v(s)-c_j|>1$ for $s \in (-\infty,s^-)\cup (s^+,+\infty)$.
\end{itemize}
We write $v(s)=c_j+ w(s) \exp{\{i\phi(s)\}}$, and term $\phi^-:=\phi(s^-)$, $\phi^+:=\phi(s^+)$. Thanks to the conservation of the angular momentum, the function $s \mapsto \phi(s)$ is strictly monotone; it is not restrictive to assume that it is increasing, and it makes sense to write
\[
\phi(+\infty)=\lim_{s \to +\infty} \phi(s), \qquad \phi(-\infty)=\lim_{s \to -\infty} \phi(s).
\]
Writing the energy in polar coordinates we get
\[
ds=\frac{dw}{\sqrt{2\left(\frac{m_j}{\a \Omega^2 w^\a}- \frac{\left(\bar{\mathfrak{C}}^j\right)^2}{w^2}   \right)}}.
\]
Hence
\begin{align*}
\phi(+\infty)-\phi^+ & = \int_{s^+}^{+\infty} \frac{d\phi}{ds}\,ds= \int_1^{+\infty} \frac{   \bar{\mathfrak{C}}^j\,dw     }{w^2 \sqrt{\frac{2m_j}{\a \Omega^2 w^\a}- \frac{\left(\bar{\mathfrak{C}}^j\right)^2}{w^2}   } }
 = \int_1^{+\infty} \frac{dw}{w^2\sqrt{\frac{1}{w^\a}-\frac{1}{w^2}}} = \int_0^1 \frac{d\xi}{\sqrt{\xi^\a-\xi^2}}.
\end{align*}
The same computation holds true for $\phi^- - \phi(-\infty)$. With the change of variable $\xi= \eta^{\frac{2}{2-\a}}$ we obtain
\begin{equation*}
\phi(+\infty)-\phi^+ = \phi^- - \phi(-\infty) = \frac{2}{2-\a} \int_0^1 \frac{ \eta^{ \frac{\a}{2-\a} } }{ \sqrt{ \eta^{\frac{2\a}{2-\a}} - \eta^{\frac{4}{2-\a} }  } }\,d\eta =\frac{2}{2-\a} \int_0^1 \frac{ d\eta }{ \sqrt{1-\eta^2} } = \frac{\pi}{2-\a}.
\end{equation*}

\noindent We deduce the following estimate for the total variation of the angle $\phi$:
\beq\label{*40}
\phi(+\infty)-\phi(-\infty) = \frac{2\pi}{2-\a} +\phi^+-\phi^- \geq \frac{2\pi}{2-\a}.
\eeq
On the other hand we know that $\bar{\phi}_n$ uniformly converges to $\phi$ on every closed interval $[a,b]$ of $\R$. For $n$ sufficiently large
\[
\bar{\phi}_n(b)-\bar{\phi}_n(a) \leq \bar{\phi}_n(d_n)-\bar{\phi}_n(c_n) < 2\pi
\]
for Proposition \ref{variazione dell'angolo limitata}. Passing to the limit for $n \to \infty$
\[
\phi(b)-\phi(a) \leq 2\pi.
\]
Since $a$ and $b$ are arbitrarily chosen, we can take $a \to -\infty$, $b \to +\infty$ to obtain
\beq\label{**40}
\phi(+\infty)-\phi(-\infty) \leq 2\pi.
\eeq
If $\a \in (1,2)$, \eqref{*40} and \eqref{**40} give a contradiction, and the proof of Proposition \ref{teorema 2.7} is complete. When $\a=1$ we don't reach yet a contradiction, but each result of this step (except Proposition \ref{teorema 2.7}, of course)  still holds true.

\paragraph{\textbf{Step 4}) Conclusion of the proof of Theorem \ref{teo dinamica interna} for $\a \in (1,2)$.}
From Proposition \ref{teorema 2.7} there exists $\bar{\r} > 0$ such that, if $0<\r_1<\r_2<\r^* \leq \bar{\r}$,
\[
\text{$u$ is a minimizer of $M_l^{\bar p_1 \bar p_2}|_{\mathcal{K}_l(\r_2,\r^*)}$} \Rightarrow \min_{t \in [c,d]} |v(t)-c_i| = \r^*,
\]
and
\[
\text{$u$ is a minimizer of $M_l^{\bar p_1 \bar p_2}|_{\mathcal{K}_l(\r_1,\r_2)}$} \Rightarrow \min_{t \in [c,d]} |v(t)-c_i| = \r_2.
\]
Hence $d(\r^*)<d(\r_2)<d(\r_1)$. We recall that the function $d(\cdot)$ is continuous in $0$, so that taking $\r_1 \to 0^+$ we obtain $d(\r^*) < d(0)$: this is a contradiction, since we are assuming that the minimum of $M_l^{\bar p_1 \bar p_2}$ on $\mathcal{K}_l^{p_1 p_2}$ is achieved over collision paths. Applying the same argument for each collision time of $u$ we obtain, by Proposition \ref{localizzazione dei minimi_M}, that a minimizer of $M$ on $K_l$ is collisions-free, too.

\paragraph{\textbf{Step 5}) The case $\a=1$.}
In case $\a=1$ the third step does not give a contradiction: indeed it is possible that $\phi(+\infty)-\phi(-\infty)=2\pi$. However, we will strongly use the results proved in step 3).

\medskip

 In order to complete the proof of Theorem \ref{teo dinamica interna}, we need the following statement.

\begin{prop}\label{alpha=1}
If the local minimizer $u \in K_l$ of $M$ has a collision, then there exists a possibly different minimizer $\wh{u} \in K_l$ such that the collision set $T_c(\wh{u})$ consists of a unique instant, and $y(t):=\wh{u}(\o t)$ is an ejection-collision solution of \eqref{PI}. In particular, this implies $p_1=p_2$.
\end{prop}

We keep the same notations already introduced. Since $u$ is the minimizer of $M$ in $K_l$, the set $T_c(u)$   is finite  and we set $t_0=\min T_c(u)$. We can define $c,d \in [0,1]$, $\bar p_1, \bar p_2 \in \R^2$, $\ldots$ as in the previous steps. If there exists $\bar{\rho}>0$ such that for every $0<\rho_1<\rho_2<\bar{\rho}$ it result $\mathcal{M}_{\r_1 \r_2}=\emptyset$, then $u$ is collisions-free in $[c,d]$. Otherwise there exist two sequences $(\r_n)$ and $(\bar{\r}_n)$ converging to $0$ such that
\begin{equation}\label{eq:u_n}
0<\r_n<\bar{\r}_n \quad \forall n, \qquad \forall n \ \exists u_n \in \mathcal{K}_l^{\bar p_1 \bar p_2}: \ \min_{t \in [c,d]} |u_n(t)-c_j|=\r_n
\text{ and } M_l^{\bar p_1 \bar p_2}(u_n)=m(\r_n,\bar{\r}_n)=d(\r_n).
\end{equation}
We know that $(u_n)$ uniformly converges to a collision path $\wt{u} \in \mathcal{K}_l^{\bar p_1 \bar p_2}$ (possibly different from $u$), which minimizes $M_l^{\bar p_1 \bar p_2}$. We define
\[
\wh{u}(t):= \begin{cases}
         u(t) & t \in [0,1] \setminus [c,d]\\
         \wt{u}(t) & t \in [c,d]
         \end{cases}
\]
Now we use a classical method to deal with singularities and to extend solutions beyond collisions, firstly introduced in 1920 by Levi-Civita in \cite{LeCi}.
In performing the  Levi--Civita regularization we  see an $\R^2$-valued function as a function in $\mathbb{C}$ and we exploit the conformal equivariance of the problem. In fact, we lift the geodesics with respect to the Jacobi metric into geodesics on the Riemann surface.

\begin{defin}
\textbf{(Local Levi-Civita transform).} For every complex-valued continuous function $u$ we define the set $\Lambda(u)$ of the continuous function $q$ such that
\[
u(t)=q^2(\tau(t))+c_j,
\]
where we reparametrize the time as
\[
dt=|q(\tau)|^2\,d\tau.
\]
\end{defin}
We will denote with $" ' "$ the differentiation with respect to $\tau$, and $\n_q$ the gradient in the Levi-Civita space. We remark that if a path $u$ does not collide in $c_j$, then $\Lambda(u)$ consist in two elements $\pm \sqrt{u(\tau(t)))-c_j}$.

\vspace{1 em}

Actually, we perform the Levi-Civita-type transform along  the sequence given in \eqref{eq:u_n}. It is convenient to define
\[
S_n:= \int_c^d \frac{dt}{|u_n(t)-c_j|}.
\]
\begin{lemma}\label{su S_n}
The sequence $(S_n)$ is bounded above and bounded below by a strictly positive constant. Hence there exist a subsequence (still denoted $(S_n)$) and $\wt{S}>0$ such that
\[
\lim_{n \to \infty} S_n=\wt{S}.
\]
\end{lemma}
\begin{proof}
Assume by contradiction that $(S_n)$ were not bounded above:
\[
\limsup_{n \to \infty} \int_c^d \frac{dt}{|u_n(t)-c_j|} =+\infty.
\]
In the proof of point ($iii$) of Lemma \ref{proprietà di u_n} we showed that
\[
\liminf_{n \to \infty} \int_c^d |\dot{u}_n(t)|^2\,dt >0,
\]
and hence $(M_l^{\bar p_1 \bar p_2}(u_n))$ is unbounded, too, in contradiction with the fact that $(u_n)$ is a minimizing sequence of a coercive functional. Furthermore, since
\[
\int_c^d \frac{dt}{|u_n(t)-c_j|} \geq \frac{d-c}{R+\eps}>0,
\]
$(S_n)$ is also bounded below by a positive constant.
\end{proof}

For every $n$, we define the set $\Lambda(u_n)$ of the continuous function $q_n$ such that
\begin{gather*}
u_n(t)=q_n^2(\tau(t))+c_j\qquad
dt=S_n |q_n(\tau)|^2\,d\tau.
\end{gather*}
We also set
\begin{gather*}
\wt{u}(t)=\wt{q}\,^2(\tau(t))+c_j\qquad
dt=\wt{S} |\wt{q}(\tau)|^2\,d\tau.
\end{gather*}
We point out that the new time $\tau$ depends on $n$ (we keep in mind this dependence, but we don't write it down to ease the notation). However, setting $\tau(c)=0$ for every $n$, the right end of the interval of definition of each function $q_n$ is
\[
\int_0^{\tau(d)} d\tau = \frac{1}{S_n}\int_c^d \frac{dt}{|u_n(t)-c_j|}=1,
\]
so that $q_n$ is defined over $[0,1]$ for every $n$. We set $\tau_n^-:=\tau(t_n^-)$ and $\tau_n^+:=\tau(t_n^+)$ (recall that $t_n^- = \inf \{t \in [c,d]: |u_n(t)-c_j|=\r_n\}$, $t_n^+=\sup \{t \in[c,d]: |u_n(t)-c_j|=\r_n \}$) .\\
Since $u_n$ doesn't collide in $c_j$ for every $n$, we can make a choice of $q_n \in \Lambda(u_n)$, in such a way that the sequence $(q_n)$ is uniformly convergent to a path $\wt{q} \in \Lambda(\wt{u})$ (for instance $q_n=+\sqrt{u_n-c_j}$ for every $n$). The constraint $B_{\r_n}(c_j)$ corresponds trough the transformation to the ball $B_{\sqrt{\r_n}}(0)$, so that $q_n$ satisfies
\begin{align*}
& |q_n(\tau)|>\sqrt{\r_n} \qquad \tau \in [0,\tau_n^-) \cup (\tau_n^+,1]\\
& |q_n(\tau)|=\sqrt{\r_n} \qquad \tau \in [\tau_n^-, \tau_n^+].
\end{align*}
In polar coordinates we write
\[
q_n(\tau)=\kappa_n(\tau) e^{i \sigma_n(\tau)},
\]
where $\kappa_n:[0,1] \to \R^+$, $\sigma_n:[0,1] \to \R$.\\
The next lemma establish a relationship between the variational properties of a function and its Levi-Civita transform.

\begin{lemma}\label{nuovo Maupertuis}
Every $q_n \in \Lambda(u_n)$ is a local minimizer of
\[
\wt{M}(q):=4 \int_0^1 |q'|^2 \int_0^1 \left[m_j+ \left(V_\eps^j(q^2+c_j)-1\right)|q|^2\right]
\]
at a strictly positive level.
\end{lemma}
\begin{proof}
It is sufficient to write the factors of $M$ in terms of $\tau$ and $q_n$:
\[
|\dot{u}_n(t)|^2\,dt=\left|2 q_n(\tau(t)) q_n'(\tau(t)) \frac{d\tau}{dt}(t)\right|^2\,dt  = \frac{4}{S_n}|q_n'(\tau)|^2\,d\tau,
\]
and
\begin{align*}
\left(V_\eps(u_n(t))-1\right)\,dt &=\left( \frac{m_j}{|q_n(\tau(t))|^2}+ V_\eps^j(q_n^2(\tau(t))+c_j)-1\right)\, dt \\
&= S_n\left[m_j+\left(V_\eps^j(q_n^2(\tau)+c_j)-1\right)|q_n(\tau)|^2\right]\,d\tau. \qedhere
\end{align*}
\end{proof}

\begin{rem}
We get a functional of Maupertuis-type. In this case the potential is no more singular, and the mass $m_j$ plays the role of the energy.
\end{rem}

Now a technical result:
\begin{lemma}
For every $n$, let
\[
\wt{\o}_n^2:=\frac{\int_0^1 \left[m_j+ \left(V_\eps^j(q_n^2+c_j)-1\right)|q_n|^2\right]}{\frac{1}{2}\int_0^1 |q_n'|^2}.
\]
The sequence $(\wt{\o}_n^2)$ is bounded above and bounded below by a strictly positive constant. Hence there exist a subsequence (still denoted $(\wt{\o}_n)$) and $\wt{\Omega}>0$ such that
\[
\lim_{n \to \infty} \wt{\o}_n=\wt{\Omega}.
\]
\end{lemma}
\begin{proof}
There holds
\[
\wt{\o}_n^2=\frac{\frac{1}{S_n} \int_c^d V_\eps(u_n)-1}{\frac{S_n}{8}\int_c^d |\dot{u}_n|^2} = \frac{4}{S_n^2}\o_n^2.
\]
Now it is sufficient recall Lemma \ref{su S_n} and the fact that $\o_n^2 \to \Omega^2>0$.
\end{proof}

From now on, we will always consider the subsequence introduced in this statement. Now we can prove the main features of the functions $q_n$.

\begin{lemma}
For every $n$:
\begin{itemize}
\item[($i$)] The function $q_n$ is of class $\mathcal{C}^1\left([0,1]\right)$.
\item[($ii$)] The restrictions $q_n|_{[0,\tau_n^-)}$ and $q_n|_{(\tau_n^+,1]}$ are $\mathcal{C}^2$ solutions of
\[
\wt{\o}_n^2 q_n''(\tau)= \n_{q_n} \left(V_\eps^j(q_n^2(\tau)+c_j)|q_n(\tau)|^2\right)-2q_n(\tau).
\]
\item[($iii$)] The energy of $q_n$ is constant in $[0,1]$:
\[
\frac{1}{2}|q_n'(\tau)|^2-\frac{1}{\wt{\o}_n^2}\left(V_\eps^j(q_n^2(\tau)+c_j)-1\right)|q_n(\tau)|^2=\frac{m_j}{\wt{\o}_n^2} \qquad \forall \tau \in [0,1].
\]
\item[($iv$)] The variation of the angle on the constraint tends to $0$ for $n \to \infty$:
\[
\lim_{n \to \infty} |\sigma_n(\tau_n^+)-\sigma_n(\tau_n^-)|=0.
\]
\item[($v$)] The time interval on the constraint tends to $0$ for $n \to \infty$:
\[
\lim_{n \to \infty} (\tau_n^+-\tau_n^-) = 0.
\]
\end{itemize}
\end{lemma}
\begin{proof}
The point ($i$) is obvious, the points ($ii$) and ($iii$) are consequence of the variational property of $q_n$, Lemma \ref{nuovo Maupertuis}.\\
($iv$) We can use the results already obtained in the step 3) (recall in particular the expression of $u_n$ in polar coordinates, the definition of the sequence $(v_n)$ and the expression of $v_n$ in polar coordinates, the equations \eqref{*40} and \eqref{**40}). \\
The angle of the function $q_n$ with respect to the origin is exactly half of the angle of $u_n$ with respect to $c_j$. Hence we can we prove that
\[
\lim_{n \to \infty}|\phi_n(t_n^+)-\phi_n(t_n^-)| = 0.
\]
or equivalently
\[
\lim_{n \to \infty} |\bar{\phi}_n(s_n^+)-\bar{\phi}_n(s_n^-)|=|\phi^+ - \phi^-| =0.
\]
From \eqref{*40} and \eqref{**40} we get
\[
2\pi +|\phi^+-\phi^-|\leq 2\pi \Leftrightarrow |\phi^+-\phi^-|=0.
\]
Recall that in the proof of \eqref{*40} and \eqref{**40} we supposed (it is not restrictive) the angle $\phi$ increasing. This is why there was no absolute value.\\
($v$) It is a consequence of the same property for $u_n$, Proposition \ref{prop 2.12}:
\begin{align*}
\tau_n^+ - \tau_n^- &= \int_{\tau_n^-}^{\tau_n^+} d\,\tau = \int_{t_n^-}^{t_n^+} \frac{dt}{S_n|q_n(\tau(t))|^2}= \frac{t_n^+ - t_n^-}{S_n \r_n} \\
&= \frac{O(\r_n^{\frac{2+\a}{2}})}{S_n \r_n} \simeq \frac{\r_n^{\frac{\a}{2}}}{S_n} \to 0
\end{align*}
for $n \to \infty$.
\end{proof}

\begin{prop}
The path $\wt{q}$ is a classical solution of
\beq\label{equazione per q tilda}
\wt{\Omega}^2 \wt{q}\,''(\tau) = \n_{\wt{q}} \left(V_\eps^j(\wt{q}\,^2(\tau)+c_j)|\wt{q}\,^2(\tau)|\right)-2\wt{q}(\tau) \qquad \forall \tau \in [0,1].
\eeq
\end{prop}
\begin{proof}
The point ($v$) of the previous lemma implies that the sequences $(\tau_n^-)$ and $(\tau_n^+)$ converge to $\tau_0 \in (0,1)$ such that $\wt{q}(\tau_0)=0$, which corresponds to the unique collision time $t_0 \in (c,d)$ of $\wt{u}$. Since every $q_n$ is $\mathcal{C}^1$, the vectors $q_n(\tau)$ is tangent to the boundary $\{q \in \mathbb{C}: |q|=\sqrt{\r}_n\}$ for every $\tau \in [\tau_n^-,\tau_n^+]$. Moreover, the variation of the angle on the constraint tends to $0$, so that
\[
\lim_{\tau \to \tau_0^-}\wt{q}\,'(\tau)=\lim_{\tau \to \tau_0^+} \wt{q} \,'(\tau),
\]
i.e. $\wt{q}$ passes trough the origin without any change of direction. We know that $q_n$ uniformly converges to $\wt{q}$ over $[0,1]$, and it is not difficult to see that the restrictions $q_n|_{[0,\tau_n^-)}$ and $q_n|_{(\tau_n^+,1]}$ converge to $\wt{q}$ in the $\mathcal{C}^1$-topology (follow the proof of Proposition \ref{convergenza per v_n}).
Next, the minimality of $q_n$, the coercivity and the weak lower semi-continuity of $\wt{M}$ imply that $\wt{q}$ is a local minimizer of $\wt{M}$ itself. As a consequence it is a weak (and, by regularity, strong) solution of
\[
\wt{\Omega}^2 \wt{q}\,'' = \n_{\wt{q}} \left(V_\eps^j(\wt{q}^2+c_j)|\wt{q}^2|\right)-2\wt{q}. \qedhere
\]
\end{proof}

\begin{proof}[Conclusion of the proof of Proposition \ref{alpha=1}]
Let us consider the functions
\[
\wt{q}_1(\tau)=\wt{q}(\tau_0+\tau), \qquad \wt{q}_2(\tau)= -\wt{q}(\tau_0-\tau).
\]
They are both solutions of \eqref{equazione per q tilda} (as far as $q^2$ is concerned, pay attention to the change of sign in the transformation $\n_{\wt{q}} \leadsto \n_{\wt{q}^2}$) with the same initial values. Thanks to the regularity of \eqref{equazione per q tilda}, the uniqueness of the solution for the Cauchy's problem and the definition of the Levi-Civita transform gives
\[
\wt{q}(\tau_0+\tau)=-\wt{q}(\tau_0-\tau) \Rightarrow \wt{u}(t_0+t)=\wt{u}(t_0-t):
\]
if the function $\wt{u}$ has a collision, then necessarily bounce against one centre and comes back along the same trajectory until the point $p_1=p_2$. Letting $x(t)=\wt{u}(\Omega t)$ for $t \in [c/\Omega,d/\Omega]$, the function can be uniquely extended over $[0,1/\Omega]$ as an ejection-collision solution of problem \eqref{P} connecting $x_1$ and $x_2=x_1$ (this uniqueness is a consequence of the uniqueness of the solutions for smooth Cauchy's problem.
\end{proof}

We end this section with some remarks about our peculiar use of the Levi-Civita regularization.

\begin{rem}\label{LC convergenza}

 We proved that, if the minimum of the restriction of $M$ over $K_l$ is achieved over a collision path, then we can find an ejection-collision minimizer in the same class. To do this, we built a minimizing sequence and then we passed to the limit in the Levi-Civita space. Thanks to the regularity of the transformed problem, we obtained an equation satisfied by the limit, and this implied the collision-ejection condition. Actually, the same procedure works if we consider a collision minimizer $u\in K_l^{p_1 p_2}([0,1])$ which is a uniform limit of a sequence $u_n \in K_l^{p_1^n p_2^n}([0,1])$ (of course, necessarily $p_1^n \to p_1$ and $p_2^n \to p_2$ when $n\to \infty$), where $u_n$ is
\begin{itemize}
\item a collisions-free minimizer for $M$ if $p_1^n \neq p_2^n$.
\item a collisions-free minimizer for $M$ if $p_1=p_2$ and the minimum of $M$ over $K_l^{p_1^n p_2^n}$ is achieved over collisions-free paths.
\item an ejection-collision minimizer for $M$ if $p_1^n = p_2^n$ is achieved by a collision path.
\end{itemize}
Up to a subsequence,  because of the uniform convergence,  we can assume that either $u_n$ is collisions-free for every $n \geq n_0$, or $u_n$ is an ejection-collision solution for every $n \geq n_0$.
In this latter case,  the uniform convergence suffices  to imply that $u$ is an ejection collision path (reasoning as  in step 5)).
\end{rem}

\begin{Levi-Civita}\label{regolarizzazione}

As clearly explained in \cite{Kn},  the $N$-center problem admits a global Levi-Civita regularization.  It consists in extending  the pullback of the Jacobi metric on the concrete Riemann surface
\[
\mathcal R=\left\{ (u,Q)\;:\; Q^2=\prod_{j=1}^N (u-c_j)\right\}
\]
to a smooth metric. The projection from $\mathcal R\mapsto \mathbb C$ on the first factor is  a branched  covering of $\mathbb C$ whose branch points $C_j=(c_j,0)$ are of order one and project on the centers $\{c_j\}$.  The Riemann surface $\tilde{ \mathcal R}=\mathcal R\setminus \{C_j\}$ doubly covers the configuration space  $\mathbb C\setminus\{c_j\}$; moreover, there is a unique way of lifting the Jacobi metric to $\tilde{\mathcal R}$ and this extend in an unique way to a smooth metric on $\mathcal R$.  Geodesics on $\mathcal R$ can be classified according with the fundamental group $\pi_1(\mathcal R)$, which is known to be isomorphic to the free group on $N-1$ generators.  The main reason why we choose to use the local L-C transform is that  we  want to keep track of  the topology of the true configuration space, and specially, of the number of self intersections and the rotation vectors with respect to the centres.  This is a common  feature of all the cases $\alpha\in[1,2)$.

This explains why we are lead to swinging back and forth from the configuration space to the Riemann surface.

Of course, also the local Levi-Civita transform induces a regularization of the flow associated with the first order system \eqref{sistema dinamico}. Indeed, let us consider an ejection-collision solution $\wh{y}$ of \eqref{P} starting from $p_0 \in \pa B_R(0)$, coming from an ejection-collision minimizer $\wh{u} \in K_{l}^{p_0 p_0}([0,1])$. The Levi-Civita transform $\wh{q}$ of $\wh{u}$ is a regular solution of \eqref{equazione per q tilda}. Let us define the reparametrization $\wh{\mathfrak{q}}(\tau):=\wh{q}(\wt{\Omega} \tau)$; it is a regular solution of
\beq\label{lc1}
\mathfrak{q}''(\tau)=\n_{\mathfrak{q}} \left(V_\eps^j(\mathfrak{q}(\tau)^2+c_j)|\mathfrak{q}(\tau)|^2 \right)-2\mathfrak{q}(\tau)
\eeq
with energy $m_j$, starting from $\wh{\mathfrak{x}}_0 \in \Lambda(x_0)$ and arriving to $ \wt{\mathfrak{x}}_0 \in \Lambda(x_0)$, with $\wh{\mathfrak{x}}_0 \neq \wh{\mathfrak{x}}_0$.
Now let us consider a collisions-free solution $y_l$ of problem \eqref{P}, with initial data $(p_1,\dot{x}_l(0))$ close to the initial data of $\wh{y}$. This solution comes from a collisions-free minimizer $u_l \in K_l^{p_1 p_2}([0,1])$ of $M$, for some $p_2 \in \pa B_R(0)$ (see Remark \ref{viceversa}). Even in this case we can consider the Levi-Civita transform $\Lambda(u_l)$ (centred in $c_j$), given by
\begin{gather*}
u_l(t)=q_l^2(\tau(t))+c_j\\
dt=S |q_l(\tau)|^2\,d\tau\\
S=\int_0^1 \frac{dt}{|u_l(t)-c_j|}.
\end{gather*}
Each component $q_l \in \Lambda(u_l)$ is a local minimizer of $\wt{M}$ at a positive level. Setting,
\[
\o_l^2:= \frac{\int_0^1 \left[ m_j + \left(V_\eps^j(q_l^2+c_j)-1\right)|q_l|^2\right]}{\frac{1}{2}\int_0^1 |q_l|^2},
\]
we infer
\begin{align*}
\o_q^2 q_l''(\tau)= \n_q \left(V_\eps^j(q_l(\tau)^2+c_j)|q_l(\tau)|^2 \right)-2q_l(\tau) \qquad \forall \tau \in [0,1]\\
\frac{1}{2}|q_l'(\tau)|^2-\frac{1}{\wt{\o}_n^2}\left(V_\eps^j(q_l^2(\tau)+c_j)-1\right)|q_l(\tau)|^2=\frac{m_j}{\wt{\o}_n^2} \qquad \forall \tau \in [0,1].
\end{align*}
The reparametrization $\mathfrak{q}_l(\tau)=q_l(\o_l \tau)$ is a solution of equation \eqref{lc1} with energy $m_j$ and initial data close to those of $\wh{\mathfrak{q}}$. This is a smooth equation, hence the continuous dependence of the solutions holds true: since the initial values of $\mathfrak{q}_l$ and of $\wh{\mathfrak{q}}$ are close together, these solutions stays close together in a right neighbourhood of $0$. In particular it is not difficult to see that this continuous dependence holds true if the solutions stays in the set which corresponds to $B_R(0)$ trough the Levi-Civita transform.
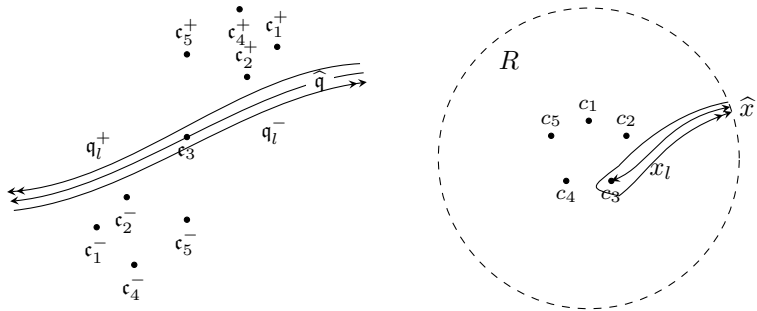
\begin{center}
\begin{tikzpicture}[>=stealth]
\draw[->,font=\footnotesize] (20:2.5cm).. controls (0.8,0.7) and (-0.8,-0.7).. node[very near start,fill=white]{$\wh{\mathfrak{q}}$} (200:2.5cm);
\draw[->>,font=\footnotesize] (23:2.5cm).. controls (0.7,0.9) and (-0.7,-0.6)..node[near end, above]{$\mathfrak{q}_l^+$}(197:2.5cm);
\draw[<<-,font=\footnotesize] (17:2.5cm).. controls (0.7,0.5) and (-0.7,-0.8)..node[near start, below]{$\mathfrak{q}_l^-$}(203:2.5cm);
\filldraw[font=\footnotesize] (0,0) circle (1pt) node[anchor=north]{$\mathfrak{c}_3$}
          (0.8,0.8) circle (1pt) node[anchor=south]{$\mathfrak{c}^+_2$}
          (1.2,1.2) circle (1pt) node[anchor=south]{$\mathfrak{c}^+_1$}
          (0.7,1.7) circle (1pt) node[anchor=north]{$\mathfrak{c}^+_4$}
          (0,1.1) circle (1pt) node[anchor=south]{$\mathfrak{c}^+_5$}
(-0.8,-0.8) circle (1pt) node[anchor=north]{$\mathfrak{c}^-_2$}
          (-1.2,-1.2) circle (1pt) node[anchor=north]{$\mathfrak{c}^-_1$}
          (-0.7,-1.7) circle (1pt) node[anchor=north]{$\mathfrak{c}^-_4$}
          (0,-1.1) circle (1pt) node[anchor=north]{$\mathfrak{c}^-_5$};
\end{tikzpicture}
$\qquad$
\begin{tikzpicture}[>=stealth]
\draw[dashed] (0,0) circle (2cm);
\draw[<->] (20:2cm)..controls (1,0.5) and (0.8,-0.1)..(0.3,-0.3);
\draw (20:2cm) node[anchor=west]{$\wh{x}$};
\draw(22:2cm)..controls (1,0.6) and (0.4,-0.1)..(0.4,-0.1);
\draw(0.4,-0.1)..controls (0,-0.4) and (0,-0.4)..(0.4,-0.5);
\draw[<<-] (18:2cm)..controls (1,0.3) and (0.8,-0.2)..node[below]{$x_l$}(0.4,-0.5);
\filldraw[font=\footnotesize] (0.3,-0.3) circle (1pt) node[anchor=north]{$c_3$}
          (0.5,0.3) circle (1pt) node[anchor=south]{$c_2$}
          (0,0.5) circle (1pt) node[anchor=south]{$c_1$}
          (-0.3,-0.3) circle (1pt) node[anchor=north]{$c_4$}
          (-0.5,0.3) circle (1pt) node[anchor=south]{$c_5$};
\draw (135:1.5cm) node[anchor=south]{$R$};
\end{tikzpicture}

\end{center}
The picture represents a comparison between the Levi-Civita space (on the left), centred in $\mathfrak{c}_3=\Lambda(c_3)$, and the configuration space (on the right) of the true $N$-centre problem. We have the ejection-collision solution $\wh{x}$, with its collision with $c_3$. In the Levi-Civita space, the corresponding path $\wh{\mathfrak{q}}$ solves the regular differential equation \eqref{lc1} (we fix an orientation of this solution given by the arrow). If we take one solution $\mathfrak{q}_l^+$ with similar initial data, we can apply the continuous dependence theorem: hence $\wh{\mathfrak{q}}$ is close (in the uniform topology) to $\mathfrak{q}_l$. If we had chosen the inverse orientation for $\wh{\mathfrak{q}}$, we would got the $\mathfrak{q}_l^-$. Coming back to the physical space, this means that if we take a solution with initial data close to those of a collision-ejection one, a continuous dependence exists (despite the lack of regularity of the potential!).\\
Let us note that, with the exception of $c_3$, each points of $\R^2$ corresponds to two points of the Levi-Civita space. For instance $\mathfrak{c}_i^\pm \in \Lambda(c_i)$ for $i=1,2,4,5$. This is due to the fact that in the Levi-Civita space two points which are poles apart are identified when we come back to the physical space. Therefore, our ejection-collision solution correspond to a path crossing the origin and showing a central symmetry, connecting two points which are identified in $p_0 \in \R^2$. We could choose both the orientations for $\wh{\mathfrak{q}}$, and the identification would give the same path in the physical space.\\
Let us also note that the angles with respect to the point $c_3$ in the physical space are cut by half in the Levi Civita one.
\end{Levi-Civita}

\section{A finite dimensional reduction}\label{incollamento}

In this section we glue the paths alternating outer and inner arcs in order to construct periodic orbits of the $N$-centre problem on the whole plane.  Thus, our building blocks will be the fixed end trajectories  found in the discussions of sections \ref{dinamica esterna} and \ref{dinamica interna}.  In order to obtain smooth  junctions, we are going to use a variational argument.   To do this, we need to know that the time interval of each of them is conveniently bounded. This is the object of the following lemmas. We recall that $\eps_2$, $\eps_3$ have been introduced in Theorem \ref{teorema 0.1} and Theorem \ref{teo dinamica interna} respectively

\begin{lemma}\label{bound per i tempi esterni}
Let $0<\eps<\min\{\eps_2,\eps_3\}$, let $p_0,p_1 \in \pa B_{R}(0)$, $|p_1-p_0| < \d$; let \\
$y_{\text{ext}}(\cdot\,;p_0,p_1;\eps)$ the "exterior" solution of \eqref{P}, found in Theorem \ref{teorema 0.1}; let $[0,T_{\text{ext}}(p_0,p_1;\eps)]$ its domain.\\
Then there exist $C_1,C_2>0$ such that
\[
C_1 \leq T_{\text{ext}}(p_0,p_1;\eps) \leq C_2 \qquad \forall (p_0,p_1) \in \left(\pa B_R(0)\right)^2.
\]
\end{lemma}
\begin{proof}
It is a straightforward consequence of the continuous dependence of the solutions by initial data and of the construction of $y_{\text{ext}}(\cdot\,;p_0,p_1;\eps)$ as a perturbed solution.
\end{proof}

\begin{lemma}\label{bound per i tempi}
Let $0<\eps<\min\{\eps_2,\eps_3\}$, let $p_1,p_2 \in \pa B_{R}(0)$; let $y_{P_j}(\cdot\,;p_1,p_2;\eps)$ be a solution of \eqref{P}, coming from a minimizer $u_{P_j}(\cdot\,;p_1,p_2;\eps) \in K_{P_j}^{p_1 p_2}([0,1])$ of $M$, for some $P_j \in \mathcal{P}$; let $[0,T_{P_j}(p_1,p_2;\eps)]$ be its domain.\\
Then there exist $C_3,C_4>0$ such that
\[
C_3 \leq T_{P_j}(p_1,p_2;\eps) \leq C_4 \qquad \forall (p_1,p_2) \in \left(\pa B_R(0)\right)^2,\ \forall P_j \in \mathcal{P}.
\]
\end{lemma}
\begin{proof}
Letting $T_{P_j}(p_1,p_2;\eps)=1/\o_{P_j}(p_1,p_2;\eps)$, we recall that
\[
\o_{P_j}(p_1,p_2;\eps)=\frac{\int_0^1 \left(V_\eps(u_{P_j}(t;p_1,p_2;\eps))-1 \right)\,dt}{\int_0^1 |\dot{u}_{P_j}(t;p_1,p_2;\eps)|^2\,dt}.
\]
Therefore we can prove that there exist $C_3,C_4 >0$ such that
\[
\frac{1}{C_4} \leq \o_{P_j}(p_1,p_2;\eps) \leq \frac{1}{C_3} \qquad \forall (p_1,p_2) \in \left(\pa B_R(0)\right)^2, \ \forall P_j \in  \mathcal{P}.
\]
Since $\mathcal{P}$ is a discrete and finite set, we can fix $P_j \in \mathcal{P}$ and apply the same reasoning for every $j$. Let us fix $\wt{p}_1, \wt{p}_2 \in \pa B_R(0)$. There exist $\wt{u}_* \in \wh{K}_{P_j}^{\wt{p}_1 \wt{p}_2}([0,1])$ and $C,\mu>0$ such that
\begin{itemize}
\item $|\dot{\wt{u}}_*|=C$ for every $t \in [0,1]$.
\item $|\wt{u}_*(t)-c_k| \geq \mu$ for every $t \in [0,1]$, for every $k =1\ldots,...N$.
\end{itemize}
There holds
\[
M(\wt{u}_*)= \frac{1}{2}\int_0^1 |\dot{\wt{u}}_*|^2 \int_0^1 \left(V_\eps(\wt{u}_*)-1\right) = \frac{C^2}{2} \int_0^1 \left(\sum_{k=1}^N \frac{m_k}{\a|\wt{u}_*-c_k|^\a} +h\right)
\leq \frac{C^2}{2}\left(\frac{M}{\a \mu^\a}+h\right)=:C_5,
\]
with $C_5>0$ since the ball of radius $R$ is a subset of $\{V_\eps > 1\}$.
Also, for every $u \in \bigcup_{p_1,p_2 \in \pa B_{R}(0)} K_{P_j}^{p_1 p_2}([0,1])$,
\beq\label{sd1}
\int_0^1 \left(V_\eps(u)-1\right) \geq \frac{M}{\a\left(\bar{R}+\max_{1 \leq k \leq N} |c_k| \right)} + h=:C_6
\eeq
with $C_4>0$ for our choice of $R$.
For a minimizer $\wt{u}=\wt{u}_{P_j}(\cdot\,;\wt{p}_1,\wt{p}_2; \eps) \in K_{P_j}^{\wt{p}_1 \wt{p}_2}([0,1])$, one has
\[
M(\wt{u})=\frac{1}{2}\int_0^1|\dot{\wt{u}}|^2 \int_0^1 \left(V_\eps(\wt{u})-1\right) \leq M(\wt{u}_*),
\]
which together with \eqref{sd1} gives
\beq\label{sd2}
\int_0^1 |\dot{\wt{u}}|^2 \leq \frac{2C_5}{C_6}.
\eeq
Starting from this bound for one single minimizer, it is not difficult to obtain a uniform bound (with respect to the ends) for every minimizers. Indeed if $(p_1,p_2) \neq (\wt{p}_1,\wt{p}_2)$, we consider the  path
\[
\wh{u}_*(t):=\begin{cases}
                 \sigma_{R}(t;p_1,\wt{p}_1) & t \in [0,1/3]\\
                 \wt{u}_*(3t-1) & t \in (1/3,2/3]\\
                 \sigma_{R}(t;\wt{p}_2,p_2) & t \in (2/3,1],
                 \end{cases}
\]
where, for $p_*,p_{**} \in \pa B_R(0)$, $\sigma_R(\cdot\,;p_*,p_{**})$ is the shorter (in the Euclidean metric) arc of $\pa B_R(0)$ connecting $p_*$ and $p_{**}$ with constant angular velocity. As far as the angular velocity is concerned, it is easy to see that it is uniformly bounded with respect to $p_*,p_{**}$. This, together with the assumptions on $\wt{u}_*$, implies that also the velocity of $\wh{u}_*$ is bounded in $[0,1]$, and
\[
M(\wh{u}_*)\leq \frac{C^2}{2} \int_0^1 \left(V_\eps(\wh{u}_*)-1\right) = C+3C \int_0^1 \left(V_\eps(\wt{u}_*)-1\right)=:C_7.
\]
This (positive) constant does not depend on the ends $p_1$ and $p_2$, so that for the family of the minimizers there holds
\beq\label{sd3}
M(u_{P_j}(\cdot\,;p_1,p_2;\eps)) \leq C_7 \qquad \forall p_1,p_2 \in \pa B_{R}(0).
\eeq
Collecting \eqref{sd1} and \eqref{sd3} we obtain
\beq\label{sd7}
\int_0^1 |\dot{u}_{P_j}(\cdot\,;p_1,p_2;\eps)|^2 \leq \frac{2 C_7}{C_6}=:C_8 \qquad \forall p_1,p_2 \in \pa B_R(0).
\eeq
A few more observations: as we have already repeated many times, the paths in
$\bigcup_{p_1,p_2 \in \pa B_R(0)} K_{P_j}^{p_1 p_2}([0,1])$ are uniformly non-constant since they have to cover at least a distance $R-\eps>0$. Thus there exists $C_9>0$ such that
\beq\label{sd4}
\|\dot{u}\|_2^2 \geq C_9 \qquad \forall u \in \bigcup_{p_1,p_2 \in \pa B_{\bar{R}}(0)} K_{P_j}^{p_1 p_2}([0,1]).
\eeq
From \eqref{sd3} and \eqref{sd4} it follows
\beq\label{sd5}
\int_0^1 \left(V_\eps(u_{P_j}(\cdot\,;p_1,p_2;\eps))-1\right) \leq \frac{C_4}{C_6}=:C_{10} \qquad \forall p_1,p_2 \in \pa B_R(0).
\eeq
Collecting \eqref{sd1}, \eqref{sd7}, \eqref{sd4}, \eqref{sd5} we obtain
\[
C_9 \leq \inf_{p_1,p_2 \in \pa B_R(0)} \|\dot{u}_{P_j}(\cdot\,;p_1,p_2;\eps)\|_2^2 \leq \sup_{p_1,p_2 \in \pa B_R(0)} \|\dot{u}_{P_j}(\cdot\,;p_1,p_2;\eps)\|_2^2 \leq C_8
\]
and
\[
C_6 \leq \inf_{p_1,p_2 \in \pa B_R(0)} \int_0^1 \left(V_\eps(u_{P_j}(\cdot\,;p_1,p_2;\eps))-1\right)
\leq \sup_{p_1,p_2 \in \pa B_R(0)} \int_0^1 \left(V_\eps(u_{P_j}(\cdot\,;p_1,p_2;\eps))-1\right) \leq C_{10}.
\]
The assertion is now an immediate consequence of the definition of $\omega_{P_j}(p_1,p_2;\eps)$.
\end{proof}

For $0<\eps<\min\{\eps_2,\eps_3\}$, $n \in \mathbb{N}$, let us fix a finite sequence of partitions $(P_{k_1},P_{k_2}, \ldots, P_{k_n})\in \mathcal{P}^n$.\\
We define
\[
D=\left\lbrace (p_0, \ldots, p_{2n}) \in \left(\pa B_R(0)\right)^{2n+1}: |p_{2j+1} - p_{2j}| \leq \d \right.
\text{ for $j =0,\ldots, n-1$}, \  p_{2n}=p_0 \Big\},
\]
where $\d$ has been introduced in Theorem \ref{teorema 0.1}.
Let $(p_0, \ldots, p_{2n}) \in D$. For every $j \in \{0,\ldots,n-1\}$, we can apply Theorem \ref{teorema 0.1} to obtain the uniquely determined path
\begin{equation}\label{notazione estesa per y}
y_{2j}(t):=y_{\text{ext}}(t;p_{2j},p_{2j+1};\eps)
= r_{\text{ext}}(t;x_{2j},x_{2j};\eps) \exp\{i \t_{\text{ext}}(t;x_{2j},x_{2j};\eps)\} \qquad t \in [0,T_{2j}],
\end{equation}
where $T_{2j}:= T_{\text{ext}}(p_{2j},p_{2j+1};\eps)$. Namely
\[\begin{cases}
\ddot{y}_{2j}(t)= \n V_{\eps}(y_{2j}(t)) \qquad &t \in [0,T_{2j}],\\
\frac{1}{2}|\dot{y}_{2j}(t)|^2-V_\eps(y_{2j}(t))=-1 \qquad &t \in [0,T_{2j}],\\
|y_{2j}(t)|>R \qquad &t \in (0,T_{2j}),\\
y_{2j}(0)=p_{2j}, \quad y_{2j}(T_{2j})=p_{2j+1}.
\end{cases}\]
We recall that $y_{2j}$ depends on $p_{2j}$ and $p_{2j+1}$ in a $\mathcal{C}^1$ manner.
On the other hand, for every $j=0,\ldots,n-1$, we can find trough Corollary \ref{dinamica interna partizioni} a path
\begin{equation}\label{notazione estesa per v}
y_{2j+1}(t):=y_{P_{k_{j+1}}} (t;p_{2j+1},p_{2j+2};\eps)
= r_{P_{k_{j+1}}}(t;p_{2j+1},p_{2j+2};\eps) \exp\{i \t_{P_{k_{j+1}}}(t;p_{2j+1},p_{2j+2}\eps)\} \qquad t \in [0,T_{2j+1}],
\end{equation}
where $T_{2j+1}:=T_{P_{k_{j+1}}}(p_{2j+1},p_{2j+2};\eps)$. Namely $y_{2j+1}$ is a path of $K_{P_{k_{j+1}}}$ such that
\[\begin{cases}
\ddot{y}_{2j+1}(t)= \n V_\eps(y_{2j+1}(t)) \qquad & t \in \left[0,T_{2j+1}\right],\\
\frac{1}{2}|\dot{y}_{2j+1}(t)|^2-V_\eps(y_{2j+1}(t))=-1 \qquad & t \in \left[0,T_{2j+1}\right],\\
|y_{2j+1}(t)|<R \qquad & t \in \left(0,T_{2j+1}\right),\\
y_{2j+1}(0)=p_{2j+1}, \quad y_{2j+1}(T_{2j+1})=p_{2j+2}.
\end{cases}\]
We know that if $\a \neq 1$ or $\a=1$ and $p_{2j+1} \neq p_{2j+2}$ then $y_{2j+1}$ is collisions-free, while if $\a=1$ and $p_{2j+1}=p_{2j+2}$ $y_{2j+1}$ can be an ejection-collision solution. Due to the invariance under reparametrizations, $y_{2j+1}$ is a local minimizer of the functional $L\left(\left[0,T_{2j+1}\right];\cdot\right)$. \\
Let us set $\mathfrak{T}_k:=\displaystyle{\sum_{j=0}^k T_j}$, $k=0,\ldots, 2n-1$. We define
\[
\g_{(p_0,\ldots,p_{2n})}(s):= \begin{cases}
                            y_0(s) & s \in [0,\mathfrak{T}_0] \\
                            y_1(s-\mathfrak{T}_0) & s \in \left[\mathfrak{T}_0,\mathfrak{T}_1\right]\\
                            \vdots \\
                            \text{$\displaystyle{y_{2n-2}\left(s- \mathfrak{T}_{2n-3}\right)}$} &  \text{$\displaystyle{s \in \left[\mathfrak{T}_{2n-3},\mathfrak{T}_{2n-2}\right]}$} \\
                            \text{$\displaystyle{y_{2n-1}\left(s- \mathfrak{T}_{2n-2}\right)}$} & \text{$\displaystyle{s \in \left[\mathfrak{T}_{2n-2},\mathfrak{T}_{2n-1}\right]}$}.
                            \end{cases}
\]
The function $\g_{(p_0,\ldots,p_{2n})}$ is a piecewise differentiable $\mathfrak{T}_{2n-1}$-periodic function; to be precise, if $\a \in (1,2)$ it is a classical solution of the $N$-centre problem \eqref{P} with energy $-1$ in $[0,\mathfrak{T}_{2n-1}] \setminus \left\{0,\mathfrak{T}_0,\ldots,\mathfrak{T}_{2n-1}\right\}$; in general, it is not $\mathcal{C}^1$ in $\left\{0,\mathfrak{T}_0, \ldots, \mathfrak{T}_{2n-1}\right\}$, but the right and left limits in these points are finite, so that it is in $H^1$. If $\a=1$ it is possible also that $\g_{(p_0,\ldots,p_{2n})}$ has a finite number of collisions. Let us observe that, thanks to Lemmas \ref{bound per i tempi esterni} and \ref{bound per i tempi}, we are sure that the time interval of $\g_{(p_0,\ldots,p_{2n})}$ is bounded above and bounded below by a positive constant for every $(p_0,\ldots,p_{2n})$, so that the period is neither trivial, nor infinite.

\vspace{1 em}

Finally, we introduce the function $F_{((P_{k_1},\ldots,P_{k_n});\eps)}:D \to \R$ defined by
\[
F(p_0, \ldots, p_{2n}) := L\left([0,\mathfrak{T}_{2n-1}]; \gamma_{(p_0,\ldots,p_{2n})}\right)
= \sum_{j=0}^{2n-1} \int_0^{T_j} \sqrt{\left(V(y_j)-1\right)\left|\dot{y}_j\right|^2} = \sum_{j=0}^{2n-1} L\left([0,T_j]; y_j\right).
\]
We point out that $F$ depends on $(P_{k_1},\ldots,P_{k_n})$ and $\eps$ trough the dependence on these quantity of $\{y_j\}$. Since we fixed $(P_{k_1},\ldots,P_{k_n})$ and $\eps$, we will omit subscripts and instead of $F_{((P_{k_1},\ldots,P_{k_n});\eps)}$ we will simply write $F$.\\
The main goal of this section is to prove the following theorem.

\begin{teor}\label{teorema principale 2}
There exists $(p_0, \ldots, p_{2n}) \in D$ which minimizes $F$.
There exists $\bar{\eps}>0$ such that, if $\eps \in (0,\bar{\eps})$, then the associated function $\g_{(p_0,\ldots,p_{2n})}$ is a periodic solution of the $N$-centre problem \eqref{P} with energy $-1$. The value $\bar{\eps}$ depends neither on $n$, nor on $(P_{k_1},\ldots,P_{k_n}) \in \mathcal{P}^n$. Moreover:
\begin{itemize}
\item[($i$)] if $\alpha \in (1,2)$ then $\g_{(p_0,\ldots,p_{2n})}$ is collisions-free.
\item[($ii$)] if $\alpha=1$ there are three possibilities:
\begin{itemize}
\item[$a$)] $\g_{(p_0,\ldots,p_{2n})}$ is collisions-free.
\item[$b$)] $\g_{(p_0,\ldots,p_{2n})}$ has a collision with one centre $c_j$, covers a certain trajectory, rebounds against a centre $c_k$ (it can occur $c_j=c_k$) and come back along the same trajectory. This is possible just if $n$ is even and $(P_{k_1},\ldots, P_{k_n})$ is equivalent to $(P_{k_1}',\ldots, P_{k_n}')$ such that
\begin{gather*}
P_{k_1}'=Q_j \in \mathcal{P}_1, \quad  P_{j_{n/2+1}}' =Q_k \in \mathcal{P}_1 \quad \text{and (if $n > 2$)}\\ P_{k_{n}}'=P_{k_2}',\ P_{k_{n-1}}'=P_{k_3}',\ \ldots, \ P_{k_{n/2+2}}'=P_{k_{n/2}}'.
\end{gather*}
\item[$c$)] $\g_{(p_0,\ldots,p_{2n})}$ has a collision iwith one centre $c_j$, covers a certain path, "rebounds" against the surface $\left\{x \in \R^2: V_\eps(x)=1\right\}$ with null velocity and come back along the same trajectory. This is possible just if $n$ is odd and $(P_{k_1},\ldots, P_{k_n})$ is equivalent to $(P_{k_1}',\ldots, P_{k_n}')$ such that
\begin{gather*}
P_{k_1}'=Q_j \in \mathcal{P}_1 \quad \text{and (if $n>1$)} \\
P_{k_n}'=P_{k_2}', \ P_{k_{n-1}}'=P_{k_3}',\  \ldots, \ P_{k_{(n+1)/2+1}}'=P_{k_{(n+1)/2}}.
\end{gather*}
\end{itemize}
\end{itemize}
\end{teor}

\begin{rem}
Theorem \ref{esistenza di soluzioni periodiche} is a trivial consequence of this result, see also Remark \ref{da eps a h}: given $0<\eps<\bar{\eps}$, for every $n \in \mathbb{N}$ and for every $(P_{k_1},\ldots,P_{k_n}) \in \mathcal{P}^n$ we found a periodic solution $\g_{(p_0,\ldots,p_{2n})}$ of \eqref{P}, whose behaviour is determined by $(P_{k_1},\ldots,P_{k_n})$. Let us set $\bar{h}=-\zeta(\bar{\eps})$. Now, given $\bar{h}<h<0$, for every $n \in \mathbb{N}$ and $(P_{k_1},\ldots,P_{k_n}) \in \mathcal{P}^n$ we obtain a periodic solution $x_{(P_{k_1},\ldots,P_{k_n})}$ of the problem \eqref{1} with energy $h$ via Proposition \ref{problema normalizzato}. As we pointed out at the end of section \ref{sezione 2}, the behaviour of $x_{(P_{k_1},\ldots,P_{k_n})}$ and of $\g_{(p_0,\ldots,p_{2n})}$ is the same.
\end{rem}

\begin{proof}
The proof requires several steps. It is essential to keep in mind the notations introduced above.
\paragraph{\textbf{Step 1) Existence of a minimizer.}} The set $D$ is compact since it is a closed subset of the compact set $\left(\pa B_R(0)\right)^{2n+1}$. It remains to show that $F$ is continuous. Let $((p_0^m, \ldots, p_{2n}^m))$ a convergent sequence in $D$: $\lim_m (p_0^m, \ldots, p_{2n}^m)=(p_0, \ldots,p_{2n}) \in D$. Let us consider
\[
F(p_0^m, \ldots, p_{2n}^m) = \sum_{j=0}^{n-1} L\left([0,T_{2j}^m]; y_{2j}^m\right) + \sum_{j=0}^{n-1} L\left([0,T_{2j+1}^m];y_{2j+1}^m\right).
\]
Here $y_{2j}^m$ (resp. $y_{2j+1}^m$) is defined as $y_{2j}$ (resp. $y_{2j+1}$); it has boundary values $p_{2j}^m$, $p_{2j+1}^m$ (resp. $p_{2j+1}^m$, $p_{2j+2}^m$), and domain $[0,T_{2j}^m]$ (resp. $[0,T_{2j+1}^m]$).  \\
The first sum is continuous in $D$, since the function $y_{2j}^m$ depends in a differentiable way on its ends. As far as the second sum is concerned, we can consider the first addendum and repeat the reasoning for the others. For $x_*,x_{**} \in \pa B_R(0)$ we define $\sigma_R(\cdot\,;p_*,p_{**})$ as the shorter (in the Euclidean metric) arc of $\pa B_R(0)$ connecting $p_*$ and $p_{**}$, parametrized in $[0,1]$.
Obviously,
\[
\forall \l > 0 \exists \varrho > 0: |p_*-p_{**}|<\varrho \Rightarrow L\left([0,1];\sigma_R\left(\cdot\,;p_*,p_{**}\right)\right)<\l.
\]
Since $y_1$ minimizes $L$ among the paths connecting $p_1$ and $p_2$ which separates the centres according to $P_{k_1}$, we have
\begin{equation}\label{*50}
L\left([0,T_1];y_1\right) \leq L\left([0,T_1^m];y_1^m\right) +
+ L\left([0,1];\sigma_R\left(\cdot\,;p_1^m,p_1\right)\right)  + L\left([0,1];\sigma_R\left( \cdot\,;p_2^m,p_2\right) \right).
\end{equation}
Here we use the invariance of $L$ under reparametrizations, so that it is possible to compare the values of $L$ for functions defined over different time-intervals.\\
Analogously, the minimal property of $y_1^m$ implies
\begin{equation}\label{**50}
L\left([0,T_1^m];y_1^m\right) \leq L\left([0,T_1];y_1\right)+
 + L\left([0,1];\sigma_R\left(\cdot\,;y_1^m,y_1\right)\right)  + L\left([0,1];\sigma_R\left( \cdot\,;y_2^m,y_2\right) \right).
\end{equation}
Passing to the limit for $m \to \infty$ in \eqref{*50} and \eqref{**50}, we get
\[
\lim_{m \to \infty} L\left([0,T_1^m];y_1^m\right) = L\left([0,T_1];y_1\right).
\]
Therefore $F$ is continuous on $D$, and has a minimum.
\paragraph{\textbf{Step 2) $F$ has partial derivatives in $D^\circ$.}} We point out that we are not proving that $F$ is differentiable in $D^\circ$. Let us fix $k \in \{0,\ldots,2n\}$, $(p_0,\ldots,p_{2n}) \in D^\circ$.
\[
F(p_0, \ldots, p_{2n}) = \sum_{j=0}^{n-1} L\left([0,T_{2j}]; y_{2j}\right) + \sum_{j=0}^{n-1} L\left([0,T_{2j+1}];y_{2j+1}\right).
\]
The first sum is differentiable with respect to $p_k$, since $y_{2j}$ depends smoothly on its ends for every $j$. As far as the second sum, just one term depends on $p_k$. It is the same to consider $k$ even or odd, so we define $k=2j+1$ for some $j=0,\ldots,n-1$. We introduce a strongly convex neighbourhood $U$ of the point $p_{2j+1}$ with respect to the Jacobi's metric. We can assume that there exists $t_* \in (0, T_{2j+1}]$ such that
\[
y_{2j+1}(t_*) \in \left(\pa U \cap y_{2j+1}\left([0,T_{2j+1}]\right)\right) \quad \text{and} \quad t \in [0,t^*) \Rightarrow y_{2j+1}(t) \in U^\circ
\]
There exists a unique minimizing geodesic $\wt{y}$ for the Jacobi metric, parametrized with respect to the arc length, connecting $p_{2j+1}$ and $p_{2j+1}(t_*)$ in a certain time $\bar{t}$ and lying in $U$, which depends smoothly on its ends. For the uniqueness and the minimality of $y_{2j+1}$, this geodesic has to be a reparametrization of $y_{2j+1}$. Hence
\[
L\left([0,\bar{t}]; \wt{y}\right) = L\left([0,t^*];y_{2j+1}\right),
\]
and the differentiability of the right side with respect to $p_{2j+1}$ is a consequence of the differentiability of the left one. Now it is sufficient to note that
\[
L\left([0,T_{2j+1}];y_{2j+1}\right) = L\left([0,t^*];y_{2j+1}\right)+ L\left([t^*, T_{2j+1}];y_{2j+1}\right);
\]
hence the left side is differentiable with respect to $p_{2j+1}$.

\paragraph{\textbf{Step 3) Computation of the partial derivatives.}}
Let us assume $k=1$ to ease the notation (for the other $k$ the computation is exactly the same). For $(p_0,\ldots,p_{2n}) \in D^\circ$, there holds
\beq\label{51*}
\frac{\pa F}{\pa p_1}\left(p_0,\ldots,p_{2n}\right) = \frac{\pa}{\pa p_1} L\left([0,T_0];y_0\right)+\frac{\pa }{\pa p_1}L\left([0,T_1]; y_1\right).
\eeq
We point out that this is a linear operator from the tangent space $T_{p_1} (\pa B_R(0))$ into $\R$. Let us consider the first term in the right side.
\begin{multline*}
\frac{\pa}{\pa p_1} L\left( [0,T_0] ; y_0 \right)  = d L \left( [0,T_0] ;y_0 \right) \left[ \frac{\pa y_0}{\pa p_1}\right]
 = \frac{1}{\sqrt{2}}\int_0^{T_0} \left[\left\langle \dot{y}_0, \frac{d}{dt}\frac{\pa y_0}{\pa p_1}\right\rangle + \left\langle \n V_\eps(y_0),\frac{\pa y_0}{\pa p_1}\right\rangle  \right]  \\
= \frac{1}{\sqrt{2}} \int_0^{T_0} \left[ \left\langle -\ddot{y}_0 + \n V_\eps(y_0),\frac{\pa y_0}{\pa p_1}\right\rangle  \right] + \frac{1}{\sqrt{2}} \left[\left\langle \dot{y}_0(t),\frac{\pa y_0}{\pa p_1}(t)\right\rangle \right]_{t=0}^{t=T_0}
  = \frac{1}{\sqrt{2}}\left[\left\langle \dot{y}_0(t),\frac{\pa y_0}{\pa p_1}(t)\right\rangle \right]_{t=0}^{t=T_0}
\end{multline*}
In the second equality we use the conservation of the energy for $y_0$, in the last one we use the fact that $y_0$ is a classical solution of the motion equation.\\
Every $\f \in T_{p_1}(\pa B_R(0))$ is of the form
\[
\f= \beta'(0) \quad \text{for some $\beta:I \to \pa B_R(0)$ of class $\mathcal{C}^1$, $\beta(0)=p_1$};
\]
in the next step it will be useful to notice that, if $p_1= R \exp{\{i \t_1\}}$, then $T_{p_1}(\pa B_R(0))$ is spanned by $i \exp{\{i \t_1\}}$.\\
For $\f=\beta'(0) \in T_{p_1}(\pa B_R(0))$ there holds
\[
\frac{\pa }{\pa p_1} y_0(0)[\beta'(0)] = \lim_{\l \to 0} \frac{y_{\text{ext}}(0;x_0,\beta(\l);\eps)-y_{\text{ext}}(0;x_0,x_1;\eps)}{\l}= 0
\]
and
\[
\frac{\pa }{\pa p_1} y_0(T_0)[\beta'(0)]
= \lim_{\l \to 0} \frac{y_{\text{ext}}(T_{\text{ext}}(p_0,\beta(\l);\eps);p_0,\beta(\l);\eps)-y_0(T_0;p_0,p_1;\eps)}{\l}= \beta'(0),
\]
where $y_{\text{ext}}(\cdot\,;p_0,\beta(l);\eps)$ is the exterior solution of \eqref{P} connecting $p_0$ and $\beta(\l)$ in time $T_{\text{ext}}(p_0,beta(\l);\eps)$ (recall that, by the proof of Theorem \ref{teorema 0.1}, we can find such a solution if $\beta(\l)$ is sufficiently close to $p_0$, even if it is not on $\pa B_R(0)$). Therefore, for every $\f \in T_{p_1}(\pa B_R(0))$
\[
\frac{\pa}{\pa p_1} L\left( [0,T_0] ; y_0\right) \left[\f\right] =\frac{1}{\sqrt{2}}\left\langle \dot{y}_0\left(T_0\right), \f \right\rangle .
\]
As far as the second term in the right side of the \eqref{51*} is concerned, we can repeat the same computation obtaining
\[
\frac{\pa}{\pa p_1} L\left( [0,T_1] ; y_1\right) [\f]= -\frac{1}{\sqrt{2}}\left\langle \dot{y}_1\left(0\right), \f \right\rangle \qquad \forall \f \in T_{p_1}(\pa B_R(0)).
\]

\paragraph{\textbf{Step 4) The minimizer is an inner point of $D$.}}
Assume by contradiction that there exists $j \in \{0,\ldots,n-1\}$ such that $|p_{2j}-p_{2j+1}|=\d$. We can produce an explicit variation of $p_{2j+1}$ such that $F$ decreases along this variation, in contradiction with the minimality of $(p_0,\ldots,p_{2n})$. It is not restrictive assume $j=0$, the same argument applies for the other cases. In this step we use the notations \eqref{notazione estesa per y} and \eqref{notazione estesa per v} introduced above. The function $y_0(\cdot)=y_{\text{ext}}(\cdot \,;p_0,p_1;\eps)$ is a solution of
\[
\begin{cases}
\ddot{y}(t)=\n V_\eps(y(t))\\
\text{ $\displaystyle{ y(0)=p_0 = R e^{i\t_0}}$}, & \text{$\displaystyle{\dot{y}(0)= \dot{r}_\eps e^{i \t_0}+ R \dot{\t}_0 e^{i \t_0}}$,}
\end{cases}
\]
where
\[
\dot{\t}_0=\dot{\t}_0(p_0,p_1;\eps), \qquad \dot{r}_\eps=\dot{r}_\eps(\dot{\t}_0),
\]
and exists $T_0=T_{\text{ext}}(p_0,p_1;\eps)$ such that $y(T_0;p_0,p_1;\eps)=p_1$ (see section \ref{dinamica esterna}).\\
As far as the function $y_1=y_{P_{k_1}}(\cdot\,;p_1,p_2;\eps)$ is concerned, it solves
\[
\begin{cases}
\ddot{y}(t)=\n V_\eps(y(t))\\
\text{$\displaystyle{ y(0)=p_1 = R e^{i\t_1}}$}, & v(T_1)= p_2 = R e^{i \t_2}\\
\text{$\displaystyle{y \in \wh{K}_{P_j}\left([0,T_1]\right)}$}.
\end{cases}
\]
For $\eps=0$ the angular momentum of the exterior solution is constant, so that
\[
\dot{\t}_0(p_0,p_1;\eps)= \dot{\t}_{\text{ext}}(T_{\text{ext}}(p_0,p_1;0);p_0,p_1;0).
\]
Let us set $\dot{\bar{\t}}_0:=\dot{\t}_0(p_0,p_1;0)$, $\bar{T}_0:=T_{\text{ext}}(p_0,p_1;0)$, and assume $\dot{\bar{\t}}_0>0$ (the case $\dot{\bar{\t}}_0<0$ is analogue). The continuous dependence of the solutions by vector field and initial data implies that
\[
 \forall \lambda > 0 \ \exists \eps_4 > 0 : \ 0<\eps<\eps_4 \Rightarrow \left|\dot{\t}_{\text{ext}}\left(T_{\text{ext}}(p_0,p_1;\eps);p_0,p_1;\eps\right)- \dot{\bar{\t}}_0\right| < \lambda.
 \]
With the choice $\lambda= \dot{\bar{\t}}_0/2$ we get
\beq\label{*52}
0< \frac{1}{2} \dot{\bar{\t}}_0< \dot{\t}_{\text{ext}}\left(T_{\text{ext}}(p_0,p_1;\eps);p_0,p_1;\eps\right) < \frac{3}{2} \dot{\bar{\t}}_0 \qquad \text{if $0<\eps<\eps_4$}.
\eeq
Coming back to the function $y_{P_{k_1}}(\cdot\,;p_1,p_2;\eps)$, we define $S=S(p_1,p_2;\eps) \in \R^+$ by
\[
t \in (0,S) \Rightarrow \frac{R}{2}<|y_1(t)| <R \text{ and $|y_1(S)|=R$}.
\]
The energy integral makes this quantity uniformly bounded from below by a positive constant $C$, as function of $\eps$.\\
Letting $\eps \to 0^+$ the centres collapse in the origin, so that for the angular momentum of $y_{P_{k_1}}(\cdot\,;p_1,p_2;\eps)$ it results
\[
\mathfrak{C}_{y_{P_{k_1}}(\cdot\,;p_1,p_2;\eps)}\left(t\right) = o(\eps)  \quad \text{for} \quad \eps \to 0^+,
\]
uniformly in $[0,C]$ (recall Proposition \ref{su Keplero}). Consequently,
\[
\forall \lambda > 0 \ \exists \eps_5 > 0: 0<\eps<\eps_5 \Rightarrow \left|\dot{\t}_{P_{k_1}}\left(0;p_1,p_2;\eps\right)\right|<\lambda.
\]
The choice $\lambda=\dot{\bar{\t}}_0/3$ gives
\beq\label{**52}
\left| \dot{\t}_{P_{k_1}}(0;p_1,p_2;\eps) \right| < \frac{1}{3} \dot{\bar{\t}}_0 \qquad \text{if $0<\eps<\eps_5$}.
\eeq
To conclude, we consider a variation $\f \in T_{p_1}(\pa B_R(0))$ of $p_1$ directed towards $p_0$ on $\pa B_R(0)$: since we are assuming $\dot{\bar{\t}}_0>0$, this variation is a positive multiple of $-i\exp{\{i \t_1\}}$. Collecting \eqref{*52}, \eqref{**52} and the step 3, we obtain that if $0<\eps<\min\{\eps_2,\eps_3,\eps_4,\eps_5\}=: \bar{\eps}$, then
\begin{align*}
\frac{\pa F}{\pa p_1} &(p_0,\ldots,p_{2n})[\f]
=\frac{CR}{\sqrt{2}}\left\langle  \left(\dot{\t}_{\text{ext}}\left(T_{\text{ext}}(p_0,p_1;\eps);p_0,p_1;\eps\right) - \dot{\t}_{P_{k_1}}\left(0;p_1,p_2;\eps\right)\right) i e^{i \t_1}, -i e^{i \t_1}\right\rangle\\
& <\frac{CR}{\sqrt{2}}\left(\frac{\dot{\bar{\t}}_0}{3}- \frac{\dot{\bar{\t}}_0}{2}\right)<0,
\end{align*}
against the minimality of $(p_0,\ldots,p_{2n})$. We point out that $\bar{\eps}$ is independent on $n \in \mathbb{N}$ and on $(P_{k_1},\ldots,P_{k_n}) \in \mathcal{P}^n$.

\paragraph{\textbf{Step 5) Regularity of the minimizers.}} Let $(p_0, \ldots, p_{2n}) \in D^\circ$ be a minimizer of $F$. Since $(p_0,\ldots,p_{2n})$ is an inner point of $D$, the existence of the partial derivatives implies that
\[
\frac{\pa F}{\pa p_k}\left(p_0,\ldots,p_{2n}\right)=0 \qquad \forall k=0,\ldots,2n.
\]
We consider for instance $k=2j+1$. From step 3) we know that for all $\f \in T_{p_k}(\pa B_R(0))$
\[
\frac{1}{\sqrt{2}}\left\langle \dot{y}_{2j}\left(T_{2j}\right) - \dot{y}_{2j+1}\left(0\right) ,\f\right\rangle =0.
\]
The tangent space $T_{p_k}(\pa B_R(0))$ is spanned by a single vector, $i e^{i \t_{2j+1}}$. We deduce
\[
|\dot{y}_{2j}\left(T_{2j}\right)| \cos{\wh{(\dot{y}_{2j}(T_{2j}), ie^{i \t_{2j+1}}})}  =  |\dot{y}_{2j+1}\left(0\right)| \cos{\wh{(\dot{y}_{2j+1}(0), ie^{i \t_{2j+1}}})}
\]
Here $\wh{(\dot{y}_{2j}(T_{2j}), ie^{i \t_{2j+1}})}$, (resp. $\wh{(\dot{y}_{2j+1}(0), ie^{i \t_{2j+1}})}$) denotes the angle between the vectors $\dot{y}_{2j}(T_{2j})$ and $ie^{i \t_{2j+1}}$ (resp. $\dot{y}_{2j+1}(0)$ and $ ie^{i \t_{2j+1}}$). As a consequence of the conservation of the energy
\beq\label{101*}
|\dot{y}_{2j}\left(T_{2j}\right)|=|\dot{y}_{2j+1}\left(0\right)|
\eeq
so that $\cos{\wh{(\dot{y}_{2j}(T_{2j}), ie^{i \t_{2j+1}}})} = \cos{\wh{(\dot{y}_{2j+1}(0), ie^{i \t_{2j+1}}})}$. Both $\dot{y}_{2j}\left(T_{2j}\right)$ and $\dot{y}_{2j+1}\left(0\right)$ point towards the interior of $B_{R/2}(0)$, so that
\beq\label{101**}
\wh{(\dot{y}_{2j}(T_{2j}), ie^{i \t_{2j+1}}})=\wh{(\dot{y}_{2j+1}(0), ie^{i \t_{2j+1}}}).
\eeq
Collecting \eqref{101*} and \eqref{101**} we can conclude that $\dot{y}_{2j}(T_{2j})=\dot{y}_{2j+1}(0)$; hence
\[
\text{$\left(p_0, \ldots, p_{2n}\right) \in D$  is a minimizer of $F$ $\Rightarrow$ $\gamma_{\left(p_0, \ldots, p_{2n}\right)} \in \mathcal{C}^1\left([0,\mathfrak{T}_{2n-1}]\right)$}.
\]

\paragraph{\textbf{Step 6) Conclusion of the proof.}} Let $(p_0, \ldots, p_{2n})$ be a minimizer of $F$ in $D^\circ$. If $\a \in (1,2)$ the function $\gamma_{(p_0,\ldots,p_{2n})}$ is a classical solution of the $N$-centres problem with energy $-1$ in $[0,\mathfrak{T}_{2n-1}] \setminus \left\{0,\mathfrak{T}_0,\ldots,\mathfrak{T}_{2n-1}\right\}$, and it is of class $\mathcal{C}^1$ in the interval $[0,\mathfrak{T}_{2n-1}]$. Since
\[
\gamma_{(p_0,\ldots, p_{2n})}(0) = \gamma_{(p_0,\ldots, p_{2n})}(\mathfrak{T}_{2n-1}),
\qquad \dot{\gamma}_{(p_0,\ldots, p_{2n})}(0) = \dot{\gamma}_{(p_0,\ldots, p_{2n})}(\mathfrak{T}_{2n-1}),
\]
it can be defined over all $\R$ by periodicity. If we prove that it is of class $\mathcal{C}^2$, we can say that $\gamma_{(p_0,\ldots, p_{2n})}$ is a classical periodic solution and the proof is complete. Let us fix $k=2j+1$, $j \in \{0,\ldots,n-1\}$ (for $k$ even the same reasoning applies). It results
\begin{multline*}
\lim_{t \to \mathfrak{T}_{2j+1}^-} \ddot{\gamma}_{(p_0,\ldots,p_{2n})}(t) = \lim_{t \to T_{2j+1}^-} \ddot{y}_{2j+1}(t) = \lim_{t \to T_{2j+1}^-} \n V(y_{2j+1}(t)) = \\
= \lim_{t \to 0^+} \n V(y_{2j+2}(t)) = \lim_{t \to 0^+} \ddot{y}_{2j+2}(t) = \lim_{t \to \mathfrak{T}_{2j+1}^+} \ddot{\gamma}_{(p_0,\ldots,p_{2n})}(t);
\end{multline*}
this completes the proof for $\a \in (1,2)$. If $\a=1$ it is possible that $\g_{(p_0,\ldots,p_{2n})}$ is collisions-free; in such a case the same line of reasoning leads to alternative ($ii$)-$a$) in Theorem \ref{teorema principale 2}. If a collision occur, we aim at showing that necessarily we are in cases ($ii$)-$b$) or ($ii$)-$c$). From Corollary \ref{dinamica interna partizioni}, a necessary condition for the presence of collisions is the existence of $k_j \in \mathcal{P}_1$ for some $j \in \{1,\ldots,n\}$; by possibly applying the right shift a number of times, it is not restrictive to assume that $j=1$.
First of all we prove that $\g_{(p_0,\ldots,p_{2n})}$ has to bounce again against a centre or against the surface $\{y \in \R^2: V_\eps(y)=1\}$. Let $t^*$ its  first collision time. Since $v_1$ is an ejection-collision trajectory, $\g_{(p_0,\ldots,p_{2n})}$ has the same property:
\beq\label{102**}
\g_{(p_0,\ldots,p_{2n-1})}(t^*+t)= \g_{(p_0,\ldots,p_{2n-1})}(t^*-t) \qquad \forall t \in \R;
\eeq
this is a consequence of the uniqueness of the solutions for the Cauchy's problem with initial point different from a singularity of the potential.
On the other hand, since $\g_{(p_0,\ldots,p_{2n-1})}$ has period $\mathfrak{T}_{2n-1}$, it has a reflection symmetry also with respect to $t^*+\mathfrak{T}_{2n-1}/2$. This second reflection can be smooth just if $\dot{\g}_{(p_0,\ldots,p_{2n})}(t^* + \mathfrak{T}_{2n-1}/2)=0$, namely if
\(
V_\eps\left(\g_{(p_0,\ldots,p_{2n})}\left(t^*+\frac{\mathfrak{T}_{2n-1}}{2}\right)\right) =1;
\)
otherwise $t^* + \mathfrak{T}_{2n-1}/2$ has to be another collision instant. In conclusion, we note that the reflection symmetry of the solution impose some symmetry restrictions on the sequence $(P_{k_1},\ldots,P_{k_n})$, which we stated in Theorem \ref{esistenza di soluzioni periodiche}.
\end{proof}

\section{Symbolic dynamics}\label{dinamica simbolica}

In this section we fix $\a \in [1,2)$ and $h \in (\bar{h},0)$.
Let us rewrite some partial results obtained for the normalized problem (energy $-1$ with parameter $\eps \in (0,\bar{\eps})$) in term of the "original" $N$-centre problem (to find solution of \eqref{1} with energy $h$).
From Corollary \ref{h--epsilon} we detect a unique $\eps\in(0,\bar{\eps})$ such that $h=\zeta(\eps)$. In section \ref{dinamica esterna} we found a solution $y_{\text{ext}}(\cdot\,;p_0,p_1;\eps)$ of \eqref{P} which stays outside $\pa B_R(0)$, and connects two points $p_0,p_1 \in \pa B_R(0)$ if their distance is smaller then $\d$. Via Proposition \ref{problema normalizzato} we get a correspondent solution $x_{\text{ext}}(\cdot\,;x_0,x_1;h)$ for equation \eqref{1} with energy $h=\zeta(\eps)$, defined over an interval $[0,T_{\text{ext}}(x_0,x_1;h)]$. This solution connects $ x_0,x_1 \in \pa B_{\bar{R}}$ close together (whose distance is smaller then $\bar{\d}$), too, and stay outside $\pa B_{\bar{R}}(0)$.
In section \ref{dinamica interna} we found a solution $y_{P_j}(\cdot\,;p_1,p_2;\eps)$ of \eqref{P} connecting $p_1,p_2 \in \pa B_R(0)$, which comes from a minimizer $u$ of the Maupertuis' functional (with energy $-1$ and potential $V_\eps$) in the class $K_{P_j}^{p_1 p_2}([0,1])$. Via Proposition \ref{problema normalizzato} we get a correspondent solution $x_{P_j}(\cdot\,;x_1,x_2;h)$ for equation \eqref{1} with energy $h=\zeta(\eps)$, connecting $x_1,  x_2 \in \pa B_{\bar{R}}(0)$, and defined over an interval $[0,T_{P_j}(x_1,x_2;h)]$. We set $T_{P_j}(x_1,x_2;h)=1/\omega(x_1,x_2;h)$. As we mentioned in Remark \ref{viceversa}, $x_{P_j}(\cdot\,;x_1,x_2;h)$ is a reparametrization of a critical point $\mfu_{P_j}(\cdot\,;x_1,x_2;h)$ of the Maupertuis' functional (with energy $h$ and with potential $V$) at a positive level. To be precise, since there is a correspondence between the space of the original problem \eqref{1} and the space of the normalized problem \eqref{P}, the path $\mfu_{P_j}(\cdot\,;x_1,x_2;h)$ is a minimizer of $M_h$ in $\mathfrak{K}_{P_j}^{x_1 x_2}([0,1])$, which is the closure in the weak topology of $H^1$ of
\begin{multline*}
\wh{\mathfrak{K}}_{P_j}^{x_1 x_2}([0,1]):=\left\{ \mathfrak{v} \in H^1\left([0,1],\R^2\right): \mfv(0)=x_1, \mfv(1)=x_2, |\mfv(t)| \leq \bar{R} \text{ and }\right.\\
 \mfv(t) \neq c_j \text{ for every $t \in [0,1]$, for every $j \in \{1,\ldots,N\}$,} \\
 \left.\text{$\mfv$ separates the centres according to the partition $P_j$}\right\};
\end{multline*}
namely in the correspondence \eqref{1}$\leftrightsquigarrow$\eqref{P} there holds
\[
\wh{\mathfrak{K}}_{P_j}^{x_1 x_2}([0,1]) \leftrightsquigarrow \wh{K}_{P_j}^{x_1 x_2}([0,1]) \quad \mathfrak{K}_{P_j}^{x_1 x_2}([0,1]) \leftrightsquigarrow K_{P_j}^{x_1 x_2}([0,1]).
\]

In what follows \emph{we consider $h \in (\bar{h},0)$ and fixed}. Hence we will omit the dependence on $h$ for the pieces of solutions of equation \eqref{1}, to ease the notation. As we stated in Corollary \ref{symbolic dynamic}, Theorem \ref{esistenza di soluzioni periodiche} enables us to characterized the dynamical system of the $N$-centre problem restricted on the energy shell
\[
\mathcal{U}_h=\left\{(x,v) \in \R^2\setminus \{c_1,\ldots,c_N\} \times \R^2: \frac{1}{2}|v|^2-V(x) = h\right\}
\]
with a symbolic dynamics, where the symbols are the element of $\mathcal{P}$.
Let us rewrite the Hamilton's equations
\beq\label{Hamilton}
 \begin{cases}
  \dot{x}(t)=v(t)\\
\dot{v}(t)=\n V(x(t)).
 \end{cases}
\eeq
Such a system defines the vector field
\begin{align*}
X:&\R^2\setminus \{c_1,\ldots,c_N\} \times \R^2 \to T(\R^2\setminus \{c_1,\ldots,c_N\}) \times \R^2\\
 &(x, v) \mapsto (v, \n V(x)),
\end{align*}
which in turns generates the flow
\begin{align*}
\f^t:&\R^2\setminus \{c_1,\ldots,c_N\} \times \R^2 \to \R^2\setminus \{c_1,\ldots,c_N\} \times \R^2 \\ &(x_0,v_0)\mapsto (x(t;x_0,v_0),v(t;x_0,v_0)).
\end{align*}
It associates with $(x_0,v_0)$ the solution of \eqref{Hamilton} having initial value $(x(0)=x_0,v(0)=p_0)$ evaluated at time $t$, and it is well defined for $t$ in an open neighbourhood of $0$. In general the flow is not complete (i.e. given $(x_0,v_0)$ the solution $(x(t;x_0,v_0),v(t;x_0,v_0))$ is not defined for every $t \in \R$), due to the collisions, but we can complete it with the agreement that if there exists $t_*\in \R$ such that $x(\cdot\,;x_0,v_0)$ has a collision at $t_*$, then we extend the corresponding solution as an ejection-collision solution:
\[
\f^{t_*+t}(x_0,v_0):=\f^{t_*-t}(x_0,v_0) \quad \forall t \in \R.
\]
This implies in particular that at most two collisions occurs for every $(x,v) \in \R^2\setminus \{c_1,\ldots,c_N\} \times \R^2$, up to periodicity. \\
Furthermore the resulting flow is the same given by the Levi-Civita regularization (see Remark \ref{regolarizzazione}); hence \emph{$\f^t$ is continuous for every $t$}.

\medskip

The energy shell $\mathcal{U}_h$ is a $3$-dimensional submanifold of $\R^2\setminus \{c_1,\ldots,c_N\} \times \R^2$, which is invariant for $X$; hence it makes sense to consider the restriction $X_h:=X|_{\mathcal{U}_h}$, for every $h \in (\bar{h},0)$. We consider the $2$-dimensional submanifolds
\[
\mathcal{U}_{h,\bar{R}}^\pm:=\left\{(x,v) \in \mathcal{U}^h: |x|=\bar{R} \text{ and } \left\langle v,x \right\rangle \gtrless 0
\right\}
\]
which are some sort of cylinders in $\R^4$; thinking at $(x,v)$ as a pair position-velocity, $\mathcal{U}_{h,\bar{R}}^+$ (respectively $\mathcal{U}_{h,\bar{R}}^-$) is the set of couples with position $x \in \pa B_{\bar{R}}(0)$, and velocity which points towards the external of (resp. towards the inner of) the ball $B_{\bar{R}}(0)$ and is not tangent to $\pa B_{\bar{R}}(0)$. For a point $(x,v) \in \mathcal{U}_{h,\bar{R}}^+$, the normal field to $\mathcal{U}_{h,\bar{R}}^+$ is
\[
 \mathcal{N}_{h,\bar{R}}(x,v)=\left( \frac{x}{\bar{R}}, 0 \right).
\]
The vector field $X_h$ is transverse to $\mathcal{U}_{h,\bar{R}}^+$, in the sense that for every $(x,v) \in \mathcal{U}_{h,\bar{R}}^+$ 
\[
\left\langle X_h(x,v),N_{h,\bar{R}}(x,v)\right\rangle = \frac{\left\langle x,v\right\rangle }{\bar{R}}>0.
\]
For every $(x,v)\in \mathcal{U}_{h,\bar{R}}^+$ we can define
\[
 \mathfrak{T}^\pm(x,v):=\left\lbrace t \in (0,+\infty): \f^t(x,v) \in \mathcal{U}_{h,\bar{R}}^\pm\right\rbrace \]
which in general can be empty. Let us term
\[
\left(\mathcal{U}_{h,\bar{R}}^+\right)^\pm:= \left\{(x,v) \in \mathcal{U}_{h,\bar{R}}^+: \mathfrak{T}^\pm(x,v) \neq \emptyset
\right\}.\]

There are points $(x,v) \in \left(\mathcal{U}_{h,\bar{R}}^+\right)^\pm$, since the periodic solutions we found in Theorem \ref{esistenza di soluzioni periodiche} do cross the circle $\{|x|=\bar{R}\}$ with velocity $\dot{x}$ satisfying the transversality condition $\left\langle x,\dot{x}\right\rangle \gtrless 0$ an infinite number of times. The continuous dependence of the solution on initial data and the transversality of $\mathcal{U}_{h,\bar{R}}^+$ with respect to $X_h$ implies that $\left(\mathcal{U}_{h,\bar{R}}^+\right)^+$ is open in $\mathcal{U}_{h,\bar{R}}$.

\medskip

We point out that, in order to fulfil the transversality condition, if $(x,v) \in \left(\mathcal{U}_{h,\bar{R}}^+\right)^+$ then its trajectory could pass trough the cylinder $\{|x|<\bar{R}\}$ (hence $(x,v)$ could stay in $\left(\mathcal{U}_{h,\bar{R}}^+\right)^-$). Therefore, for $(x,v)\in \left(\mathcal{U}_{h,\bar{R}}^+\right)^+$, it makes sense to set
\[
T_{\min}^\pm:= \inf \mathfrak{T}^\pm(x,v) .
\]
For every $(x,v) \in \left(\mathcal{U}_{h,\bar{R}}^+\right)^+$ such that $T_{\min}^- < T_{\min}^+$, we consider $\{\f^t(x,v)\}_{t \in [T_{\min}^-, T_{\min}^+]}$, i.e. the restriction of the trajectory starting from $(x,v)$ to the first time interval needed to cross $B_{\bar{R}}(0)$.  We define
\begin{multline*}
\mathcal{U}_{h,\bar{R}}^\mathcal{P}:= \left\lbrace (x,v) \in \left(\mathcal{U}_{h,\bar{R}}^+\right)^+ : T_{\min}^-<T_{\min}^+, \  \{\f^t(x,v)\}_{t \in [T_{\min}^-, T_{\min}^+]} \text{ parametrizes a self-intersections-free} \right. \\
\left. \text{minimizer of $L_h$  in $\mathfrak{K}^{x(T_{\min}^-) x(T_{\min}^+)}_{P_j}$, for some $P_j \in \mathcal{P}$}\right\rbrace .
\end{multline*}
It is non-empty, since our periodic orbits provide an infinite number of points satisfying these conditions. It is possible to define a \emph{first return map} on $\mathcal{U}_{h,\bar{R}}^\mathcal{P}$ as
\[
\mathcal{R}(x,v):= \f^{T_{\min}^+}(x,v) .
\]
We can also introduce an application $\chi:\mathcal{U}_{h,\bar{R}}^\mathcal{P} \to \mathcal{P}$ given by
\[
\chi(x,v):=  \begin{cases}
P_j & \text{if $\{\f^t(x,v)\}_{t \in [T_{\min}^-, T_{\min}^+]}$ parametrizes a path which}\\
& \text{ separates the centres according to $P_j$, with $P_j \in \mathcal{P}$}\\
Q_j & \text{if $\{\f^t(x,v)\}_{t \in [T_{\min}^-, T_{\min}^+]}$ parametrizes }\\
& \text{an ejection-collision path, which collides in $c_j$} .
 \end{cases}
\]
Finally, let us term
\[
\Pi_h:= \bigcap_{j \in \mathbb{Z}} \mathcal{R}^j(\mathcal{U}_{h,\bar{R}}^\mathcal{P}),
\]
the set of initial data such that the corresponding solutions cross the circle $\pa B_{\bar{R}}(0)$ with velocity directed towards the exterior of the ball $B_{\bar{R}}(0)$ an infinite number of time in the future and in the past. Again, the periodic solutions found in Theorem \ref{esistenza di soluzioni periodiche} provide an infinite number of points in $\Pi_h$. Now, for every $(x,v)\in \Pi_h$, we set $\pi:\Pi_h \to \mathcal{P}^\Z$ as
\[
\pi(x,v)=(P_{j_k})_{k \in \mathbb{Z}} \quad \text{where} \quad P_{j_k}:=\chi(\mathcal{R}^{k-1}(x,p)).
\]
Introduced the restriction $\mathfrak{R}:=\mathcal{R}|_{\Pi}$, the proof of Corollary \ref{symbolic dynamic} reduces to the following Proposition.
\begin{prop}\label{prop symb dyn}
Under the assumption of Theorem \ref{esistenza di soluzioni periodiche}, the map $\pi$ is continuous and surjective, and the diagram
\[
 \xymatrix{
\Pi_h \ar[r]^{\mathfrak{R}} \ar[d]^\pi & \Pi_h \ar[d]^\pi \\
\mathcal{P}^\Z \ar[r]^{T_r} & \mathcal{P}^\Z,
}
\]
commutes.
\end{prop}

We need some preliminary results.
The first step is to obtain uniform bounds, below and above, for the time interval of the pieces of solution found in sections \ref{dinamica esterna} and \ref{dinamica interna}.

\begin{lemma}\label{bound per i tempi problema originale}
There exist $C_1,C_2>0$ such that for every $(x_0,x_1) \in \left(\pa B_{\bar{R}}(0)\right)^2$ such that $|x_1-x_0| < \bar{\d}$, and for every $(x_2,x_3) \in \left(\pa B_R(0)\right)^2$, for every $P_j \in \mathcal{P}$, there holds
\begin{gather*}
C_1 \leq T_{\text{ext}}(x_0,x_1) \leq C_2 \\
C_1 \leq T_{P_j}(x_2,x_3) \leq C_2.
\end{gather*}
\end{lemma}
\begin{proof}
It is a straightforward consequence of Lemmas \ref{bound per i tempi esterni} and \ref{bound per i tempi}, and of Proposition \ref{problema normalizzato}.
\end{proof}

It will be useful to prove that, for a sequence of minimizers of $M_h$ which separate the centres according to the same partition $P_j$, the convergence of the ends to $(\bar{x}_1, \bar{x}_2)$ is sufficient for the weak convergence in $H^1$ of the minimizers themselves; the limit path turns out to be minimal for $M_h$ in $\mathfrak{K}_{P_j}^{\bar{x}_1 \bar{x}_2}([0,1])$
\begin{lemma}\label{conv estremi->conv debole}
Let $(x_1^n,x_2^n) \subset \left(\pa B_{\bar{R}}(0)\right)^2$ such that $(x_1^n,x_2^n) \to (\bar{x}_1, \bar{x}_2)$, let $P_j \in \mathcal{P}$; let $\mfu_n$ be a local minimizers of $M_h$ in $\mathfrak{K}_{P_j}^{x_1^n x_2^n}([0,1])$.\\
Then there exists a subsequence $(\mfu_{n_k})$ of $(\mfu_n)$ and a minimizer $\bar{\mfu} \in \mathfrak{K}_{P_j}^{\bar{x}_1 \bar{x}_2}([0,1])$ of $M_h$ such that $\mfu_{n_k} \wc \bar{\mfu}$ in $H^1$.
\end{lemma}
\begin{proof}
In order to prove that, up to subsequence, $(\mfu_n)$ is weakly convergent, it is sufficient to show that $(\mfu_n)$ is bounded in $H^1$. We know that
\[
\|\mfu_n\|_2^2 \leq \bar{R}^2 \qquad \forall n,
\]
hence it remains to check that there exists $C>0$ such that
\[
\|\dot{\mfu}_n\|_2^2 \leq C \qquad \forall n.
\]
In the previous proof we saw that the minimality of $\mfu_n$ implies that this inequality is satisfied with $C=C_8$ (see \eqref{sd7}).
Now let us prove that the limit $\mfu$ is a minimizer of $M_h$. We recall that in step 1) of the proof of Theorem \ref{teorema principale 2}, we proved the continuity of the function which associate to each couple of ends $p_1,p_2 \in \pa B_R(0)$ the length $L$ (in the "normalized" problem) of the minimizers $u_{P_j}(\cdot;p_1,p_2;\eps)$, for every $P_j$. It is straightforward to check that the same property holds true for $L_h$ with $h \neq -1$, and $x_1, x_2 \in \pa B_{\bar{R}}(0)$.
Assume by contradiction that $\mfu$ were not a local minimizer of $M_h$; by Proposition \ref{minimo L->M} it follows that $\mfu$ cannot be a minimizer also of $L_h$, then there exists a path $\mfv \in L_{P_j}^{\bar{x}_1 \bar{x}_2}([0,1])$ such that $L_h(\mfv) \leq L_h(\mfu)$. Let $\sigma_{\bar{R}}(\cdot\,;x_*,x_{**})$ be the shorter (in the Euclidean metric) arc of $\pa B_{\bar{R}}(0)$ connecting $x_*$ with $x_{**}$, parametrized with constant velocity. As $|x_*-x_{**}|$ tends to zero, the length of $\sigma_{\bar{R}}$ tends to $0$; hence there exists $n_0 \in \mathbb{N}$ such that
\[
\wh{\mfv}_n(t):=\begin{cases}
                 \sigma_{\bar{R}}(t;x_1^n,\bar{x}_1) & t \in [0,1/3]\\
                 \mfv(3t-1) & t \in (1/3,2/3]\\
                 \sigma_{\bar{R}}(t;\bar{x}_2,x_2^n) & t \in (2/3,1],
                 \end{cases}
\]
is a path of $\mathfrak{K}_{P_j}^{x_1^n x_2^n}([0,1])$ such that $L_h(\mfv_n)<L_h(\mfu_n)$, a contradiction with the minimality of $\mfu_n$.
\end{proof}

\begin{proof}[Proof of Proposition \ref{prop symb dyn}]
\textbf{Step 1)} We start with surjectivity. Let $(P_{j_n})_{n \in \Z} \subset \mathcal{P}^\Z$. We can consider the finite sequences
\[
(P_{j_0}),  \quad (P_{j_{-1}},P_{j_0},P_{j_1}), \quad \ldots \quad (P_{j_{-n}},\ldots,P_{j_{-1}},P_{j_0},P_{j_1},\ldots,P_{j_n}),\quad \ldots.
\]
To each sequence we associate the corresponding periodic solution of equation \eqref{1} with energy $h$ given by Theorem \ref{esistenza di soluzioni periodiche}, according to the notation
\[
(P_{j_{-n}},\ldots,P_{j_{-1}},P_{j_0},P_{j_1},\ldots,P_{j_n}) \leftrightsquigarrow x^{n}(\cdot).
\]
Up to time translations, we can take initial data $(x^{n}(0),\dot{x}^{n}(0)) \in \Pi_h$, in such a way that the first partition (or collision) determined by the solution $x^{n}(\cdot)$ is $P_{j_0}$, for every $n$. \\
The path parametrized by $x^{n}(\cdot)$ detects a sequence of points $(x^{n}_k)_{k \in \Z}$ of $\pa B_{\bar{R}}(0)$ given by the intersections of the trajectories in $\R^2$ with the circle itself, taken in the temporal order (of course, since $x^n(\cdot)$ is periodic, the sequence will be periodic, too).\\
We get a sequence of sequences:
\[
\left(x^{n}_k\right)_{n \in \mathbb{N}} \quad \forall k \in \mathbb{Z}.
\]
Now, $\left(x^{n}_0\right)_{n}$ stays in the compact set $\pa B_{\bar{R}}(0)$, therefore we can extract a subsequence $\left(x^{n_0}_0\right)_{n_0}$ which converges to $\bar{x}_0$. Analogously, $\left(x^{n_0}_1\right)_{n_0}$ stays in $\pa B_{\bar{R}}(0)$, therefore we can extract a subsequence $\left(x^{n_1}_1\right)_{n_1}$ which converges to $\bar{x}_1$. Proceeding in this way, for every $k \in \Z$ we have a sequence $\left(x^{n_k}_k\right)_{n_k}$ which converges to $\bar{x}_k$. Then we relabel as $(x^n_k)_n$ the diagonal sequence, namely $(x^{n_n}_k)_n$. It results
\beq\label{sd8}
\lim_{n \to \infty} x^{n}_k=\bar{x}_k \qquad \forall k \in \Z.
\eeq
For every $k \in \mathbb{Z}$, we connect the points $\bar{x}_{2k}$, $\bar{x}_{2k+1}$ with the unique external solution of \eqref{1} given by Theorem \ref{teorema 0.1}. Analogously, we connect $\bar{x}_{2k+1}$ and $\bar{x}_{2k+2}$ with the inner solution given by Theorem \ref{teo dinamica interna}. We point out that a collision can occur just if $\a=1$ and $\bar{x}_{2k+1}=\bar{x}_{2k+2}$. We can juxtapose these paths in a continuous manner, following the same gluing procedure already carried on in section \ref{incollamento} to define $\g_{(p_0,\ldots,p_{2n})}$; in this way we obtain a continuous function $\bar{x}(\cdot): \R \to \R^2$. We claim that it is a solution of \eqref{1} (in case $\a=1$, it can be an ejection-collision solution) such that $(\bar{x}(0), \dot{\bar{x}}(0)) \in \Pi_h$ and $\pi((\bar{x}_0, \dot{\bar{x}}(0)))= (P_{j_k})_k$.
\medskip
The first step is to prove that, up to a subsequence, $(x^n(\cdot))$ converges to $\bar{x}(\cdot)$ uniformly  on every compact set of $\R$. If $[a,b] \subset \R$ such that $\bar{x}(a)=\bar{x}_{2k}$ and $\bar{x}(b)=\bar{x}_{2k+1}$, with $k \in \Z$, then the uniform convergence in $[a,b]$ is a straightforward consequence of the continuous dependence of the external solutions by the end points (Theorem \ref{teorema 0.1}). On the other hand, if $[c,d] \subset \R$ with $\bar{x}(c)=\bar{x}_{2k+1}$ and $\bar{x}(d)=\bar{x}_{2k+2}$, then the uniform convergence has been proved in Lemma \ref{conv estremi->conv debole}. From this, it is easy to obtain the uniform convergence for every compact subset of $\R$. Let us observe that since $\bar{x}|_{[c,d]}$ is a uniform limit of minimizers of $L_h$ (and hence, up to reparametrizations, also of $M_h$), if $\bar{x}|_{[c,d]}$ has a collision, necessarily $\bar{x}|_{[c,d]}$ parametrizes an ejection-collision path (see Remark \ref{LC convergenza}). Now,  assume first that $\bar{x}(\cdot)$ has no collisions in $\R$. Let $[a,b] \subset \R$ be compact. In this case there exists $\bar{n} \in \mathbb{N}$ such that $x^n(\cdot)$ is collisions-free in $[a,b]$, as well. The function $V(\bar{x}(\cdot))$ is well defined in $\R$, and by regularity
\beq\label{sd20}
\lim_{n \to \infty} \ddot{x}^n(t)= \lim_{n \to \infty} \n V(x^n(t))= \n V(\bar{x}(t)),
\eeq
with uniform convergence in $[a,b]$. Also, the derivative of $x^n(\cdot)$ is uniformly bounded in $[a,b]$ for the conservation of the energy:
\[
|\dot{x}^n(t)|=\sqrt{2(V(x^n(t))+h)}\leq \sqrt{2\left(C+h\right)} \qquad \forall t \in [a,b], \forall n \geq \bar{n}.
\]
Hence, up to subsequence, there exists a point $\bar{t} \in (a,b)$ such that $\dot{x}_n(t)$ is convergent in $\R^2$. This fact, together with \eqref{sd20}, implies that $(\dot{x}^{n}(\cdot))$ converges in $\mathcal{C}^1([a,b])$, and hence $(x^n(\cdot))$ converges in the $\mathcal{C}^2([a,b])$ to $\bar{x}(\cdot)$, for every compact subset $[a,b] \in\R$. This means that $\bar{x}$ is a $\mathcal{C}^2$ solution of \eqref{1} with energy $h$ on $[a,b]$ and this argument works in every compact subset of $\R$. We point out that the uniform convergence is sufficient to say that, in its $k$-th passage inside $B_{\bar{R}}(0)$, $\bar{x}(\cdot)$ separates the centres according to $P_{j_k}$, namely $\pi(\bar{x}(0),\dot{\bar{x}}(0)))=(P_{j_k})_{k \in \Z}$.
\medskip
Now, we are left to examine what happens if a collision occurs. Let
\[
T_c(\bar{x}):=\{t \in \R: \bar{x}(t) = c_j \text{ for some $j \in \{1,\ldots,N\}$}\}.
\]
The $\mathcal{C}^2$-convergence of $(x^n(\cdot))$ to $\bar{x}(\cdot)$ is still true in every compact subset of $\R \setminus T_c(\bar{x})$, hence we obtain an ejection-collision solution of \eqref{1} with energy $h$ and $\pi(\bar{x}_0,\dot{\bar{x}}(0)))= (P_{j_k})_{k \in \Z}$. \\
We point out that this is possible just for $\a=1$ and $(P_{j_k}) \in \mathcal{P}^\Z$ such that
\begin{itemize}
\item $(P_{j_k})$ is periodic and satisfies the conditions of points ($ii$-$b$) or ($ii$-$c$) of Theorem \ref{esistenza di soluzioni periodiche}.
\item up to a finite number of applications of the right shift, $P_{j_0} \in \mathcal{P}_1$ and the sequence is symmetric, i.e. $P_{j_{-n}}=P_{j_n}$ for every $n$.
\end{itemize}
\textbf{Step 2)} It remains to show that $\pi$ is continuous. Let $(x_0,v_0) \in \Pi_h$. We would like to prove that given $\l>0$ there exists $\varrho>0$ such that for every $(x,v)\in \Pi_h$:
\[
|(x,v)-(x_0,v_0)|<\varrho \Rightarrow \sum_{m \in \Z} \frac{d_1(\pi_m(x,v),\pi_m(x_0,v_0)) }{2^{|m|}}<\l,
\]
where $\pi_m$ is the projection $\pi_m:\Pi_h \to \mathcal{P}$ defined by
\[
\pi_m(x,v):=\chi(\mathfrak{R}^{m-1}(x,v)),\
\]
i.e. $\pi_m$ associate to $(x,v)$ the partition that the corresponding solution induces in its $m$-th passage inside $B_{\bar{R}}(0)$.\\
Let us observe that there exists $m_0 \in \mathbb{N}$ such that
\[
\sum_{|m| > m_0} \frac{1}{2^{|m|}} <\l.
\]
Hence it is sufficient to show that, if we take two initial data sufficiently close, then the corresponding solutions induce the same partitions $P_{j_k}$ of the centres, for $k \in \{-m_0,\ldots,m_0\}$.

\medskip

Thanks to Lemmas \ref{bound per i tempi problema originale} and \ref{bound per i tempi}, we can fix a time interval $[-a,a]$ such that each solution with initial data in $\Pi_h$ passes at least $2m_0+1$-times inside $B_{\bar{R}}(0)$ in $[-a,a]$. If the solution of \eqref{1} with starting point $(x_0,v_0)$ is collisions-free, then there exists $\mu>0$ such that
\[
|x(t;x_0,v_0)-c_j| \geq \mu \qquad \forall t \in [-a,a].
\]
If $(x,v)$ is sufficiently close to $(x_0,v_0)$, then the continuous dependence applies:
\[
\forall \l \in \left(0,\frac{\mu}{2}\right) \ \exists \varrho > 0: |(x,v)-(x_0,v_0)|<\varrho \Rightarrow
|x(t;x,v)-x(t;x_0,v_0)|<\l.
\]
This implies that $x(\cdot\,;x,v)$ is collisions-free and detects the same partitions of $x(\cdot;x_0,v_0)$ in $[-a,a]$. In particular, $\pi_m(x,v)=\pi_m(x,v)$ for every $m \in\{-m_0,\ldots,m_0\}$. This proves the continuity for non-collision initial data. But nothing change if we consider $(x_0,v_0) \in \Pi_h$ such that $x(\cdot;x_0,v_0)$ has a collision: indeed we introduced a regularization trough the Levi-Civita transform (see Remark \ref{regolarizzazione} on the Levi-Civita transform), so that the continuous dependence applies also in this case.
\end{proof}

\end{document}